\documentclass[11pt]{article}

\usepackage[utf8]{inputenc}
\usepackage{amsmath}
\usepackage{amssymb,amsfonts,amsthm}
\usepackage{mathrsfs}
\usepackage{algorithm}
\usepackage{algorithmic}
\usepackage{xcolor}
\usepackage{tikz}
\usepackage{fullpage}
\usepackage{dsfont}
\usepackage{thmtools}
\usepackage{thm-restate}
\usepackage{enumitem}

\usepackage{hyperref}
\usepackage{xcolor}
\hypersetup{
  colorlinks   = true, 
  urlcolor     = blue, 
  linkcolor    = blue, 
  citecolor   = blue 
}

\usepackage{tikz}
\usetikzlibrary{decorations.pathreplacing}
\usetikzlibrary{arrows,positioning}
\DeclareRobustCommand\shape{
 \lower5pt\hbox{
 \hskip-7pt
  \tikzset{circ/.style={circle, draw, fill=black, scale=.15}}
  \begin{tikzpicture}[semithick,scale=.3]
  \node (l1) at (0,.5) [circ]{};
  \node (l3) at (0.5,0.3) [circ]{};
  \draw[-] (l1) to node [auto] {} (l3);
    \end{tikzpicture}
  \hskip-8pt}
}

\newcommand{\bern}{\mathsf{Bern}}
\newcommand{\unif}{\mathsf{Unif}}

\newtheorem{theorem}{Theorem}[section]

\newtheorem{innercustomlem}{Lemma}
\newenvironment{mylem}[1]
  {\renewcommand\theinnercustomlem{#1}\innercustomlem}
  {\endinnercustomlem}

  \newtheorem{innercustomthm}{Theorem}
\newenvironment{mythm}[1]
  {\renewcommand\theinnercustomthm{#1}\innercustomthm}
  {\endinnercustomthm}

\newtheorem{lemma}{Lemma}[section]
\newtheorem{corollary}[lemma]{Corollary}
\newtheorem{fact}{Fact}
\theoremstyle{definition}
\newtheorem{remark}{Remark}[]
\newtheorem{definition}{Definition}

\newcommand{\var}{\mathsf{Var}}
\DeclareMathOperator*{\E}{{\rm I}\kern-0.18em{\rm E}}
\renewcommand{\Pr}{\,{\rm I}\kern-0.18em{\rm P}}
\newcommand{\1}{\mathds{1}}
\newcommand{\ind}[1]{\1\{#1\}}
\newcommand{\poi}{\mathrm{Poisson}}
\newcommand{\bino}{\mathsf{Binomial}}
\newcommand{\multino}{\mathsf{Multinomial}}
\newcommand{\TV}{{d_\mathsf{TV}}}
\newcommand{\RES}{\mathsf{Res}}
\newcommand{\law}{\mathcal{L}}
\newcommand{\rgt}{\mathsf{RGT}}
\newcommand{\rig}{\mathsf{RIG}}
\newcommand{\G}{\mathsf{G}}
\newcommand{\disteq}{\stackrel{d}{=}}

\newcommand{\defeq}{\triangleq}
\newcommand{\rg}{\mathsf{P_0}}
\newcommand{\prg}{\mathsf{P_1}}
\newcommand{\prgp}{\mathsf{P'_1}}
\newcommand{\rigp}{\mathsf{RIG_P}}
\newcommand{\rigpf}{\mathsf{RIG_{P4}}}

\newcommand{\spacedot}{\,\cdot\,}

\newcommand{\Af}{\cA_{ \triangle}}
\newcommand{\Ab}{\cA }

\newcommand{\Xp}{X^+}

\newcommand{\ine}{{T(e)}}
\newcommand{\nine}{{\sim T(e)}}

\newcommand{\Xpe}{\Xp_\ine }
\newcommand{\Xpet}{\tilde X^+_\ine }
\newcommand{\Xpne}{\Xp_{\nine}}
\newcommand{\Xm}{X^-}
\newcommand{\xm}{x^-}

\newcommand{\Yp}{Y^+}

\newcommand{\e}{\mathsf{e}}

\newcommand{\B}{\Big}
\renewcommand{\b}{\big}

\newcommand{\cg}[1]{{\color{green}\texttt{[#1]}}}

\newcommand{\dotspace}{\,\cdot\,}

\newcommand{\mug}[1]{\mu_{#1}}

\newcommand{\IM}[2]{\Inf^M_{#1\rightarrow #2}}

\newcommand{\I}[2]{\Inf_{#1\rightarrow #2}}

\newcommand{\Inf}{\mathbf{I}}

\newcommand{\ER}{{Erd\H{o}s-R\'enyi}}

\newcommand{\aux}{X^{\mathsf{aux}}}

\newcommand{\enew}{E^{\mathsf{new}}}

\newcommand{\RGTi}[1]{\rgt_{#1}}

\newcommand{\wed}[1][]{W_{#1}}

\newcommand{\bwed}[1][]{\bar{W}_{#1}}

\newcommand{\Twed}{\bar{T}}

\newcommand{\cX}{\mathcal{X}}
\newcommand{\cA}{\mathcal{A}}
\newcommand{\cL}{\mathcal{L}}
\newcommand{\cG}{\mathcal{G}}
\newcommand{\cY}{\mathcal{Y}}

\newcommand{\cP}{\mathcal{P}}

\newcommand{\quadand}{\quad\text{and}\quad}

\newcommand{\gps}{\mathsf{GPS}}
\newcommand{\easy}{\mathsf{easy}}
\newcommand{\hard}{\mathsf{hard}}
\newcommand{\impo}{\mathsf{impossible}}
\newcommand{\pd}{\mathsf{PD}}
\newcommand{\gl}{\mathsf{D}_{gl}}

\newcommand{\cov}{\mathsf{Cov}}

\newcommand{\pprime}{p'_*}

\newcommand{\Ar}{\mathcal{A}_r}

\newcommand{\maxI}{I_m}

\newcommand{\pval}{{1/(n\log n)}}

\usepackage[titles]{tocloft}
\title{Algorithmic Decorrelation and Planted Clique in Dependent Random Graphs: The Case of Extra Triangles}

\author{Guy Bresler\thanks{Supported in part by NSF Career award CCF-1940205.}
\and Chenghao Guo\thanks{Supported in part by NSF TRIPODS grant DMS-2022448, NSF Career award CCF-1940205, CCF-2131115 and the MIT-IBM Watson AI Lab.} \and Yury Polyanskiy\thanks{Supported in part by the NSF grant CCF-2131115 and the MIT-IBM Watson AI Lab.}}

\date{Massachusetts Institute of Technology}

\begin{document}

\maketitle

\pagenumbering{gobble}

\begin{abstract}
We aim to understand the extent to which the noise distribution in a planted signal-plus-noise problem impacts its computational complexity. To that end, we consider the planted clique and planted dense subgraph problems, but in a different ambient graph. Instead of {{\ER} } $G(n,p)$, which has independent edges, we take the ambient graph to be the \emph{random graph with triangles} (RGT) obtained by adding triangles to $G(n,p)$. 
We show that the RGT can be efficiently mapped to the corresponding $G(n,p)$, and moreover, that the planted clique (or dense subgraph) is approximately preserved under this mapping. 
This constitutes the first average-case reduction transforming dependent noise to independent noise. Together with the easier direction of mapping the ambient graph from {\ER } to RGT, our results yield a strong equivalence between models.
In order to prove our results, we develop a new general framework for reasoning about the validity of average-case reductions based on \emph{low sensitivity to perturbations}. 
\end{abstract}

\newpage
\setcounter{tocdepth}{2}
\tableofcontents

\clearpage
\pagenumbering{arabic}

\section{Introduction}

Most modern statistical inference problems exhibit a striking phenomenon: the best efficient algorithm requires substantially more data, or lower noise level, than the information-theoretic limit achieved by inefficient algorithms. Such problems are said to exhibit a \emph{statistical-computational gap}. 
In this paper we are interested in \emph{average-case} planted problems, meaning that their inputs are sampled according to some probability distribution with planted structure.
In order to develop methodology for rigorously reasoning about the computational complexity of these average-case problems, a vibrant line of research at the interface of statistics, probability, and computational complexity has emerged. 
There are two main approaches to substantiating computational limits: hardness against specific classes of algorithms and average-case reductions.

The first approach attempts to determine the best possible performance for restricted classes of algorithms, such as low-degree polynomials~\cite{hopkinsThesis,wein2022average}, sum-of-squares (SoS) relaxations~\cite{barak2016nearly}, \cite{hopkins2017power}, statistical query algorithms~\cite{feldman2013statistical,feldman2015complexity,diakonikolas2017statistical}, first-order methods~\cite{celentano2020estimation}, or classes of circuits~\cite{rossman2008constant,rossman2014monotone}. Beyond hardness against specific classes of algorithms, the overlap gap method introduced in \cite{gamarnik2014limits} and refined in \cite{huang2022tight} has revealed how structural properties of a solution landscape can rule out multiple classes of algorithms. The power of distinct classes of algorithms for solving families of statistical problems have been related: low-degree polynomials vs. statistical query \cite{brennan2020statistical}, low-degree polynomials vs. approximate message passing \cite{montanari2022equivalence}, and SoS vs. spectral applied to matrices of low-degree polynomials \cite{hopkins2017power}.
This broad approach continues to see intense activity and yields insight into the limits of current algorithms. However, it has the drawbacks that (1) each given class of algorithms is known to be suboptimal in certain settings \cite{koehler2022reconstruction,wein2019kikuchi,zadik2022lattice}, which reduces confidence in lower bounds against any specific class, and (2) one must prove fresh lower bounds for each problem of interest and for each new class of algorithms that emerges.

The second approach entails devising
average-case reductions that map one statistical problem to another using a polynomial-time algorithm, whereby hardness of the first problem is transferred to the second. The reduction approach is a foundational tool across complexity theory and cryptography for elucidating relationships between problems and constructing complexity hierarchies. Moreover, it provides insights applicable to all algorithms. The main drawback of this approach is that in the average-case setting, one must precisely map not just problem instances, but rather entire probability distributions over inputs. This is notoriously challenging due to a lack of techniques, as emphasized by Goldreich \cite{goldreich2011notes}, Barak \cite{Barak2017}, and Bogdanov and Trevisan \cite{bogdanov2006average}.

A common starting hardness assumption in the literature on average-case reductions is the Planted Clique (PC) Conjecture. This conjecture posits that no polynomial-time algorithm can detect a planted clique of size $k$ in an \ER\ random graph with edge density $p$, when $p$ is constant and $k=o(\sqrt{n})$. Many problems have been shown to be computationally hard based on this assumption including sparse principle component analysis \cite{berthet2013complexity,wang2016statistical,gao2017sparse,brennan2018reducibility,brennan2019optimal}, submatrix detection or biclustering \cite{ma2015computational,brennan2018reducibility,brennan2019universality}, planted dense subgraph  \cite{hajek2015computational,brennan2018reducibility}. A variety reduction techniques were introduced in \cite{brennan2018reducibility} and  \cite{brennan2020reducibility}, resulting in a web of reductions from PC to problems including robust sparse mean estimation, tensor PCA, general planted dense subgraph, dense stochastic block model, mixtures of sparse linear regressions, robust sparse mean estimation, and many others. 
Some works show reductions from hypergraph planted clique or hypergraph planted dense subgraph to other problems, including tensor clustering \cite{luo2022tensor} and SVD for random 3-tensors \cite{zhang2017tensor}.
Another line of research has also explored reductions to statistical problems based on cryptographic assumptions \cite{gupte22,song2021cryptographic,chen2022hardness,daniely2021local}.

Despite these advances, the current literature on average-case reductions for statistical problems exhibits a notable limitation: the absence of two-way equivalences. Most reductions rely on a single hardness assumption, such as the Planted Clique Conjecture or Learning Parity with Noise (LPN), and demonstrate a statistical-computational gap for other problems. 
Going beyond this, one might dream of partitioning statistical problems into equivalence classes, just as has been done so fruitfully in worst-case complexity. 
We mention here, only briefly, that the notion of equivalence is itself nontrivial because each statistical ``problem" is itself an entire \emph{parameterized family} of problems. 
We rigorously define what we mean by strong computational equivalence in Section~\ref{sec:equiv}.

The primary obstacle to proving equivalence between statistical problems, just as for one-way reductions, is that average-case reductions are challenging: we must transform one noise distribution into another while preserving the planted signal. To date, all reductions have converted models with independent noise to models with either independent or dependent noise. The challenge in establishing equivalence between planted problems is therefore in developing techniques for removing dependence in the noise of high-dimensional distributions.
In this paper, we introduce a general framework for the validity of reductions based on low sensitivity to perturbations and apply it to show the first non-trivial computational equivalence between two statistical problems. 

A second (related) motivation for the present paper is towards developing an understanding of how general are the phenomena observed in planted statistical problems. 
For planted matrices or tensors with independent entries, a reduction was devised from planted clique or planted dense subgraph to a broad class of entry distributions by \cite{brennan2019universality}, thereby showing that the observed phenomena are \emph{universal} for this class. Thus, we understand how to change the distribution of the entries, but as of yet there are no general techniques for connecting problems with different \emph{dependence} between the entries\footnote{The exception is sparse PCA, which has a very specific dependence structure.}.

\subsection{The Problems We Consider}
\paragraph{General Hypothesis Testing Setup} 
In a simple-versus-simple hypothesis testing problem there is a pair of distributions $ P_0$ and $P_1$ on space $\Omega$, and one observes $X$ generated according to one of the distributions. The task is to decide between the two hypotheses
$$
H_0: X\sim P_0\quad \text{versus} \quad H_1: X\sim P_1\,.
$$
We consider noisy signal detection problems indexed by problem size parameter $n$, where the space is $\Omega_n$, $P_{0,n}$ represents the pure noise distribution, and $ P_{1,n}$ is the distribution with planted signal. We often keep the dependence on $n$ implicit. An algorithm $\Phi:\Omega \to \{0,1\}$ \emph{solves} the hypothesis testing problem if the sum of type I and type II errors asymptotically vanish, i.e.,
$$
\Pr_{X\sim P_0}(\Phi(X) = 1) + \Pr_{X\sim P_1}(\Phi(X) = 0)  \to 0\quad\text{as}\quad n\to \infty\,.
$$


\paragraph{Planted Dense Subgraph Problem}
In this paper, we consider a slight generalization of Planted Clique known as the \emph{Planted Dense Subgraph (PDS)} problem. 
Here, the 
pure noise distribution is the simple \ER\ random graph, $G(n,p)$. The alternative hypothesis, denoted by $G(n,p,k,q)$, is generated by starting with $G(n,p)$, choosing a uniformly random subset $S$ of size $k$ from $[n]$, erasing all edges in $S$, and independently resampling the edges in $S$ to be included with probability $q$. 
The PDS problem is, given a graph $G$, to distinguish between the two hypotheses $$H_0: G\sim G(n,p)\quadand H_1: G\sim G(n,p,k,q)\,.$$

It is believed that the PDS problem with constant $p$ and $q$ has the same $k=o(\sqrt{n})$ computational threshold as the planted clique problem~\cite{barak2016nearly}, \cite{hopkinsThesis}. 

\paragraph{Random Graph with Triangles}
Next, instead of the ambient graph being \ER, we consider a simple graph distribution with dependent edges. 

\begin{definition}[Random Graph with Triangles]\label{def:rgt}
The random graph with triangles distribution $\rgt(n,p,p')$ is the law of a graph generated as follows. Let $G\sim G(n,p)$ be an {\ER} random graph and for every triple of nodes $\{i,j,k\}\in {[n]\choose 3}$ with probability $p'$ independently add the triangle consisting of the three edges $(i,j), (j,k)$, and $(i,k)$ to $E(G)$.
\end{definition}

\begin{remark}
The expected edge density of $G(n,p,p')$ is $p+(1-p)\left(1-(1-p')^{n-2}\right)$. It is always assumed that $p'=O(1/n)$ so that the edge density remains bounded away from 1. 
\end{remark}
\begin{remark}
In certain parameter regimes, the RGT model is close in total variation distance to the Random Intersection Graph, a model inspired by real-world networks. This will be explained in more depth in Section~\ref{sec:results}.
\end{remark}


\paragraph{Planted Dense Subgraph in Random Graph with Triangles}
We plant a signal in the random graph with triangles (\emph{planted RGT}) to obtain $\rgt(n,p,p',k,q)$, the law of a graph generated as follows. Start with $\rgt(n,p,p')$, choose a random subset $S$ of $k$ vertices, erase all edges within $S$ and then independently include each edge within $S$ with probability $q$. 
The Planted RGT problem is, given a graph $G$, to distinguish between the two hypotheses $$H_0: G\sim \rgt(n,p,p')\quadand H_1: G\sim \rgt(n,p,p',k,q)\,.$$

\begin{remark}
    There are multiple options for the definition of Planted RGT. We will discuss a substantial generalization in the next section. One could also, for example, plant a smaller RGT with different edge or triangle densities within a larger RGT. We leave this for future work. 
\end{remark}

\subsection{Our results}\label{sec:results}

Our main result is a strong computational equivalence between the planted dense subgraph (PDS) and planted random graph with triangles (Planted RGT) problems. To achieve this, we must design transformations that connect these problems in both directions. Here, and throughout the paper, for two distributions $\mu$ and $\nu$ we use $\mu\approx_\epsilon\nu$ as shorthand for $\TV(\mu,\nu)\leq \epsilon$.


    

\begin{theorem}\label{thm:main}
     For any $k= o(n^{1/4}\log^{-17/4}n), p'\le 1/(n\log n)$ and $0<p,q<1$, there are efficient algorithms $\Af$ and $\cA$ that satisfy the following transition properties. 
     \begin{enumerate}
         \item $\Af$ maps both $H_0$ and $H_1$ hypotheses of PDS to planted RGT with parameter map $f$:  
         \[\Af(G(n,p))\approx_{o_n(1)} \rgt(n,p,p')\,,
    \quadand
    \Af(G(n,p,k,q))\approx_{o_n(1)} \rgt(n,p,p',k,f(q))\,,\]
    \item    $
    \cA$ maps both $H_0$ and $H_1$ hypotheses of planted RGT to PDS with parameter map $g$: \[\cA(\rgt(n,p,p')) \approx_{o_n(1)} G(n,p)\,,
        \quadand\cA(\rgt(n,p,p',k,q)) \approx_{o_n(1)} G(n,p,k,g(q))\,.\]
     \end{enumerate}
Moreover, the parameter mappings $f$ and $g$ satisfy $f\circ g(q) = q+o_n(1)$ and $g\circ f (q) = q+o_n(1)$. 
\end{theorem}

Part 1 entails transforming from PDS to planted RGT, i.e., an independent noise model to a dependent noise model. In the case of $G(n,p)$ to $\rgt(n,p,p')$, the mapping itself is straightforward based on the definition of RGT: one simply adds each triangle with probability $p'$. However, the fact that the planted distributions are mapped correctly via the very same procedure does still require a nontrivial argument.

In Part 2, the other direction, we show how to map from RGT to {\ER }. Here the mapping even in the unplanted pure-noise case is not obvious, and we describe it in Section~\ref{sec:reverse}. It is obtained by viewing the triangle addition procedure as a Markov transition kernel on graphs, and implementing its time-reversal. 

To constitute a valid reduction from Planted RGT to PDS, the same mapping must also work when a dense subgraph has been planted. 
We face a delicate balance between two competing objectives: (1) maintaining the algorithm's behavior outside of the dense subgraph, ensuring that its performance remains approximately invariant to the placement of the dense subgraph, and (2) removing the dependence between triangles, which inherently necessitates dependence in the algorithm's behavior. 

Validity of both reductions is proved via application of a general framework that we develop, stated as Theorem~\ref{thm:general} in Section~\ref{sec:general}. 
Applying our general theorem presents a number of technical challenges in the form of bounding total variation distances between perturbed distributions. Methods for proving the proximity of two high-dimensional distributions in total variation distance are notably limited, and even more so when neither of them is a product distribution. To this end, we introduce several technical innovations that will be discussed in detail in the next section.

We conjecture that the statement of Theorem~\ref{thm:main} holds for $k$ up to $\tilde O(\sqrt{n})$, while our result only applies for $\tilde O(n^{1/4})$. The range of $p'$ is optimal up to log factors: if $p' \geq 2.1n^{-1}\ln n$, then with high probability the RGT is the complete graph and planted RGT is information-theoretically impossible, so the planted RGT cannot be equivalent to PDS in this regime. We leave open the problem of obtaining a similar result for the full range of $k, p,q$, and $p'$.

\paragraph{General Planted Signal}

Our results not only demonstrate that the detection of planted dense subgraphs problem is equivalent in \ER\  graphs and RGT, but also reveal that a wide range of planted signal detection problems are equivalent in both models. Let us define the general planted signal (GPS) model. 
\begin{definition}[General Planted Signal (GPS)] For edge sequence 
$\vec e=(e_1,\dots, e_K)$
and 
$\vec p = (p_1,\cdots,p_K)$
let 
$\gps_{\vec e, \vec p}$ take as input a graph, and for each $1\le i\le K$, edge $e_i$ is resampled to be included in the graph with probability $p_i$. 
For a pure-noise distribution $\rg$, let $\gps(\rg)$ denote the distribution of $\gps(G)$ where $G\sim \rg$.
\end{definition}

Planted dense subgraph is a special case of GPS with $e_1,\cdots, e_{\binom{k}{2}}$ being the edges inside the dense subgraph, and $p_i=q$ is the planted edge density. The GPS problem in random graph distribution $\rg$ is the hypothesis testing problem between $\rg$ and $\gps(\rg)$.

Our bi-directional reduction result holds also in this general setting.
\begin{theorem}
    Let $\gps_n=\gps_{\vec{e},\vec{p}}$ be a sequence of general planted signal problems.
    For any $K=o(\sqrt{n}\log^{-17/2}n), p'\le 1/(n\log n)$ and $p_1,\dots, p_K$ each bounded away from 0 and 1 by a constant, there are distributions $\gps_n^g = \gps_{\vec{e},g(\vec{p})}$ and $\gps_n^f = \gps_{\vec{e},f(\vec{p})}$ (where $f(\vec{p})$ stands for $f:\mathbb{R}\rightarrow 
    \mathbb{R}$ applied to all elements of the vector) and efficient algorithms $\cA$, $\cA'$ 
    satisfying the transition properties.
    \begin{enumerate}
        \item     $\cA(G(n,p))\approx_{o_n(1)} \rgt(n,p,p')
    $ and
$\cA(\gps_n(G(n,p)))\approx_{o_n(1)} \gps^f_n( \rgt(n,p,p') )$, and
\item     $\cA'(\rgt(n,p,p')) \approx_{o_n(1)} G(n,p)$ and
    $\cA'(\gps_n(\rgt(n,p,p'))) \approx_{o_n(1)} \gps^g_n(G(n,p))$.
    \end{enumerate}
Moreover, the parameter mappings $f$ and $g$ satisfy $f\circ g(q) = q+o_n(1)$ and $g\circ f (q) = q+o_n(1)$. 
\end{theorem}

\begin{remark}
   GPS includes Subgraph Stochastic Block Model (SSBM), which has a rank-1 signal (i.e., a small SBM) on $k$ vertices planted inside an {\ER } graph. Its complete phase diagram was determined by \cite{brennan2018reducibility}, with computationally hard region obtained via reduction from the planted clique conjecture. Our result shows that for $k= o(n^{1/4}\log^{-17/4}n)$, the computational complexity of SSBM is the same in ambient graph being \ER\ or RGT. 
\end{remark}

\paragraph{Approximation of Random Intersection Graph}
Aside from RGT being intrinsically interesting as a natural random graph model with low-order dependence, it turns out that the RGT non-trivially approximates the random intersection graph. 
\begin{definition}[Random Intersection Graph]
A sample from $\rig(n,d,\delta)$ is defined by sampling $n$ sets $S_1,\cdots,S_n\subset [d]$ where each $S_i$ includes each element of $[d]$ independently with probability $\delta$. Vertices $i$ and $j$ are connected in $G$ if and only if $S_i$ and $S_j$ have non-empty intersection.
\end{definition}


\vspace{-3mm}
\begin{restatable}{theorem}{rigrgt}\label{thm:rig-rgt}
For an RIG with constant edge density, i.e., $d=\Theta(1/\delta^2)$, if $d\gg n^{2.5}$, then $$\TV(\rig(n,d,\delta),\rgt(n,1-e^{-d\delta^2 +(n-2)d\delta^3(1-\delta)^{n-3}},1-e^{-d\delta^3(1-\delta)^{n-3}}))=o(1).$$
\end{restatable}

It was shown in \cite{brennan2019phase} that the threshold for distinguishing RIG from \ER\ occurs at $d\sim n^3$. Thus, the RGT is close in total variation in the range $n^{2.5}\ll d\ll n^3$, whereas in this range the RIG and \ER\ have total variation distance close to 1.

\section{General Result and Application to RGT}
\label{sec:general}
We start by introducing a general theorem that shows how a mapping between unplanted distributions yields a mapping also in the planted case if the mapping satisfies a certain \emph{perturbation invariance}. We then describe our specific mappings between {\ER } and RGT. After that, in the following section, we give the main ideas for showing that our mappings satisfy the conditions of our general perturbation theorem, which constitutes the bulk of the paper's technical innovation. 

\subsection{Reductions and Sensitivity to Perturbation}

We introduce a novel and general framework 
for validity of a reduction from one general graph model to another. The general result will be used to show that planted dense subgraph is computationally equivalent in \ER\  and in $\rgt$. In this section, the graph models and transformations are abstract and specific constructions will be discussed in Section~\ref{sec:reverse}.

Let $\rg$ and $\rg'$ be two distributions of random graphs on vertex set $[n]$ and consider the general planted signal $\gps_{\Vec{e},\Vec{p}}$ on both random graphs.  

Suppose that  $\cA$ is a randomized algorithm (equivalently, a Markov kernel) satisfying $\cA(\rg)=\rg'$.  
The following theorem gives a general sufficient condition for $\cA$ to approximately correctly map the \emph{planted distribution} $\prg:=\gps_{\Vec{e},\Vec{p}}(\rg)$ to the planted distribution $\prgp:=\gps_{\Vec{e},\Vec{q}}(\rg')$.
For each $0\le i\le K$, let $\prg_i$ (or $\prgp_i$) be the planted version of $\rg$ (or $\rg'$), defined by starting with $\rg$ (or $\rg'$) and resampling edges $e_1,e_2,\cdots, e_i\in \binom{[n]}{2}$, respectively, with probability $p_1,p_2\cdots, p_i$ (or $q_1,q_2\cdots, q_i$). 

\begin{theorem}\label{thm:general}
    Let $\cG$ be a set of graphs on $n$ vertices such that for any $0\le i\le K$,
    $\Pr_{G\sim \prg_i}(G\in \cG)\ge 1-\delta$.
Let $\cA$ be a randomized mapping between graphs on $n$ vertices satisfying $\cA(\rg)=\rg'$.
    Suppose that for each $e\in \binom{[n]}{2}$, there exists $p_-^e,p_+^e\in [0,1]$ such that for every $G\in \cG$:
\begin{enumerate}[label=\textbf{\textup{C\arabic*}}]
    \item \label{cond-gen1} Presence of edge $e$ in the input to $\cA$ has little influence on the other edges in the output of $\cA$: 
    $$
    \cA (G-e)|_{\sim e}\approx_{\epsilon} \cA (G+e)|_{\sim e}\,.
    $$
    \item \label{cond-gen2} In the output graph, edge $e$ is approximately independent of the rest of the edges: 
    \[
    \cA(G-e)\approx_{\epsilon} \cA(G-e)|_{\sim e} \times \cA(G-e)|_{e}
    \quadand
    \cA(G+e)\approx_{\epsilon} \cA(G+e)|_{\sim e} \times \cA(G+e)|_{e}\,.
    \]
    \item \label{cond-gen3} The marginal probability on edge $e$ is approximately constant as a function of the input graph: 
    \[
    \b|\Pr(e\in \cA(G-e))- p^e_-\b|\le \epsilon
\quadand
    \b|\Pr(e\in \cA(G+e))- p^e_+\b|\le \epsilon\,. 
    \]
\end{enumerate}
Then we have 
    \[\TV\b(\cA(\gps_{\Vec{e},\Vec{p}}(\rg)),\gps_{\Vec{e},\Vec{q}}(\rg')\b)=O(K(\epsilon+\delta))\,,\]
    where $q_i = p_ip_+^{e_i} + (1-p_i)p_-^{e_i}$ for each $i\in [K]$.
\end{theorem}

The main technical challenge of this paper is to prove that our proposed mapping $\cA$ from RGT to \ER\ satisfies conditions \ref{cond-gen1} and \ref{cond-gen2}. The mapping is given in Section~\ref{sec:reverse}.

\begin{remark}\label{rem:equivalent-general-theorem}
It turns out
that the conditions in the theorem are equivalent (up to constant factors) to the following more compact conditions:
    \[\TV\b(\cA(G-e), \cA(G-e)|_{\sim e}\times \bern(p_-^e)\b)\le \epsilon\quadand\TV\b(\cA(G+e), \cA(G-e)|_{\sim e}\times \bern(p_+^e)\b)\le \epsilon\,.\]
    We state Theorem~\ref{thm:general} in the form above because it corresponds to how we apply it.
\end{remark}

\subsubsection{Proof of Theorem~\ref{thm:general} via Single Edge Lemma}
The idea is to resample edges of the input one at a time and bound the effect of each step on the output of $\cA$. 
Let $\RES_{e}^p$ be the operation of resampling edge $e$ to be present with probability $p$. 

\begin{lemma}[Single Edge Lemma]\label{lem:key}
Under the conditions of Theorem~\ref{thm:general},
for any $1\le i\le K$,
\[
\TV\b(\cA\circ\RES_{e_i}^{p_i}(\prg_{i-1}),\RES_{e_i}^{q_i}\circ \cA(\prg_{i-1})\b)=O(\epsilon+\delta)\,.
\]
\end{lemma}

Thus, the operations of applying $\cA$ and resampling edge $e_i$ \emph{approximately commute}, with $p_i$ being replaced by $q_i$.

Note that $\prg_i=\RES_{e_i}^{p_i}\prg_{i-1}$, so Lemma~\ref{lem:key} equivalently states that  
\[
\TV\b(\cA(\prg_i ),\RES_{e_i}^{q_i}\circ \cA(\prg_{i-1})\b)=O(\epsilon+\delta)
\,.
\]
Applying the triangle inequality $K$ times and using the last display for each term yields
\[
\TV\b(\RES_{e_1}^{q_1}\circ\cdots\circ \RES_{e_K}^{q_K}\circ \cA(\rg), \cA(\prg_K)\b) = O(K(\epsilon+\delta))\,.
\]
Note that $\cA(\rg) = \rg'$, and by definition $\RES_{e_1}^{q_1}\cdots \RES_{e_K}^{q_K} \rg' = \prg_K'$, so Theorem~\ref{thm:general} is proved. It remains to prove Lemma~\ref{lem:key}. 

\subsubsection{Proof of Single Edge Lemma~\ref{lem:key}}
\begin{proof}[]
Suppose $G_{i-1}\sim \prg_{i-1}$ and $G_i=\RES_{e_i}^{p_i} G_{i-1}$.
Let $\cA^{e_i}(G)$ be the following transformation that decouples edge $e_i$ with rest of the edges,
\begin{enumerate}
    \item Sample a graph from $\cA(G-e)$.
    \item If $e\notin G$, resample edge $e$ to be present with probability $p_-^e$. Otherwise, resample edge $e$ to be present with probability $p_+^e$.
\end{enumerate}
In other words,  $\cA^{e_i}(G)=\RES_{e_i}^{p_-^{e_i}\ind{e\notin G}+p_+^{e_i}\ind{e\in G}}\circ \cA(G-e)$. 
From \ref{cond-gen2}, $\cA(G)\approx_{\epsilon} \cA(G)|_{\sim e_i}]\times \cA(G)|_{e_i}$. From $\ref{cond-gen3}$, if $e\in G$, $\cA(G)|_{e_i}\approx_{\epsilon} \bern(p_+^{e_i})$, if $e\notin G$, $\cA(G)|_{e_i}\approx_{\epsilon} \bern(p_-^{e_i})$.
Therefore, $\TV(\cA^{e_i}(G),\cA(G))\le 2\epsilon$.
If $G_{i-1}\in \cG$, we have 
\[
\RES_{e_i}^{q_i}\circ \cA(G_{i-1}) \approx_{2\epsilon} \RES_{e_i}^{q_i}\circ \RES_{e_i}^{p_-^{e_i}\ind{e_i\notin G_{i-1}}+p_+^{e_i}\ind{e_i\in G_{i-1}}}\circ \cA (G_{i-1}-e) = \RES_{e_i}^{q_i}\circ\cA (G_{i-1}-e)\,.
\]
Because $G_{i-1}\in \cG$ with probability at least $1-\delta$, $\RES_{e_i}^{q_i}\circ \cA(\prg_{i-1})\approx_{2\epsilon+\delta} \RES_{e_i}^{q_i}\circ\cA (\prg_{i-1}-e)$.

Similarly, fixing $G_i\in \cG$, we have $\cA(G_i) \approx_{\epsilon} \cA^{e_i}(G_i)$.
Therefore, 
\[
\cA(\prg_i) \approx_{2\epsilon+\delta} \cA^{e_i}(\prg_i)
\,.
\]
Use $\prg_{i} = \RES_{e_i}^{p_i}(\prg_{i-1}) = \E_{G_{i-1}}[ \RES_{e_i}^{p_i}G_{i-1}]$ and the definition of $\cA^{e_i}$, the left-hand side is equal to $\cA\circ \RES_{e_i}^{p_i} (\prg_{i-1}) $ and the right-hand side is equal to 
\begin{align*}
&
\E_{G_{i-1}}\B[ \RES_{e_i}^{p_-^{e_i}\ind{e\notin \RES_{e_i}^{p_i} (G_{i-1})}+p_+^{e_i}\ind{e\in \RES_{e_i}^{p_i} (G_{i-1})}}\circ \cA (G_{i}-e) \B]\\
&=\ \E_{G_{i-1}}\B[ \RES_{e_i}^{p_ip_+^{e_i}+(1-p_i)p_-^{e_i}}\circ \cA (G_{i-1}-e) \B]\\
&=\ \E_{G_{i-1}}\B[ \RES_{e_i}^{q_i}\circ\cA (G_{i-1}-e)\B]\\
&=\ \RES_{e_i}^{q_i}\circ\cA (\prg_{,i-1}-e)
\,.
\end{align*}
In other words, we have
\[\cA\circ \RES_{e_i}^{p_i} (\prg_{,i-1}) \approx_{4\epsilon+2\delta} \RES_{e_i}^{q_i}\circ\cA (\prg_{,i-1}-e) \,.\qedhere\]
\end{proof}

\subsection{Mapping between {\ER } and RGT}
\label{sec:reverse}

To use Theorem~\ref{thm:general} to relate ordinary PDS to the version in $\rgt$, we first need to specify transformations between the ambient graph models. Our mapping from $\G(n,p)$ to $\rgt(n,p,p')$ simply applies the definition of $\rgt$.


\begin{definition}[Forward Transition]\label{def:forward}
Given any graph $G$, let $\Af(G)$ be the graph obtained from $G$ by the following process: independently for each set of three vertices, add the three edges between them with probability $p'$. This defines a Markov transition kernel on the space of graphs.
\end{definition}

By the definition of random graph with triangles, $\Af$ transfers $\G(n,p)$ to $\rgt(n,p,p')$.
The reverse transition is more complicated. We will first describe the distribution of the set of  triangles that were added to $G_0$, conditioned on observing $G = \Af(G_0)$. 
Let $X\in \{0,1\}^{[n]\choose 3}$ be an indicator of a set of triangles.  We use $|X|$ to denote the size of this set, $E(X)$ to denote the set of edges included in at least one of the triangles, and $\e(X)= |E(X)|$ to denote the total number of edges.

\begin{definition}[Triangle Distribution]\label{def:triangle_dist}
For a given graph $G$, the \emph{triangle distribution} $\mu_G$ is a distribution over subsets of triangles $x\in \{0,1\}^{\binom{[n]}{3}}$. The probability mass function $\mu_G$ is given by
\begin{equation*}
    \mu_G(x)=
    \frac{1}{Z_G}\b(\frac{p'}{1-p'}\b)^{|x|}p^{-\e(x)}
    \,,
\end{equation*}
if $E(x)\subseteq G$ and $ \mu_G(x)=0$ if $E(x)\not\subseteq G$. Here
 $Z_G=\sum_{x:E(x)\subset G}\b(\frac{p'}{1-p'}\b)^{|x|}p^{-e(x)}$ normalizes $\mu_G$. 

It will be convenient to let $Y\in \{0,1\}^{[n]\choose 2}$ be the indicator vector of the edge set $E(X)$, 
\begin{equation*}
Y_e=\ind{e\in E(X)}\quad \text{for each } e\in {[n]\choose 2}\,.
\end{equation*}
To distinguish between different graphs, we use $\cL_G(Y)$ to denote the law of $Y$ for a given graph $G$.
\end{definition}

\begin{definition}[Reverse Transition]\label{def:reverse}
Given any graph $G$, let $G'=\Ab(G)$ be the graph obtained from $G$ by the following process:
\begin{enumerate}
    \item Sample a set of triangles $X\sim \mu_G$
    \item Let $G'$ be equal to $G$ on the set $E(X)^c$ 
    \item For each $e\in E(X)$, include $e$ in $G'$ independently with probability $p$.
\end{enumerate}
\end{definition}
\begin{remark}
The reverse transition is only analyzed for $p'=1/(n\log n)$, not for any $p'\le 1/(n\log n)$. Nonetheless, we can still construct a reverse map that works for any $p'\le 1/(n\log n)$ by first adding triangles to increase $p'$ to $1/(n\log n)$, and then applying $\Ab$. This is formalized in Corollary~\ref{cor:true-reverse}.
\end{remark}


\begin{lemma}\label{lem:reverse}
The reverse transition $\Ab$ maps $\rgt(n,p,p')$ to the {\ER} $\G(n,p)$ distribution.

\end{lemma}
\begin{proof}
Recall the definition of $\rgt(n,p,p')$: we start with $\G(n,p)$ and add triangles. Let $T$ denote the set of triangles that are selected, and $G$ denote the random graph after adding triangles. Let $|G|$ be the number of edges in graph $G$.

Let $P(G|G')$ denote the distribution on graphs $G$ obtained by applying the forward transition on $G'$, i.e., the distribution of $\Af(G')$.
We will show that the reverse transition is exactly the chain given by the posterior distribution $P(G'|G)$. To start,
\begin{align*}
P(G'|G)\propto P(G|G')P(G')
&=\sum_{X} \ind{G=G'+E(X)} p'^{|X|}(1-p')^{\binom{n}{3}-|X|}p^{|E(G')|}(1-p)^{\binom{n}{2}-|E(G')|}\\
&\propto\sum_{X}  \ind{G=G'+E(X)} \B(\frac{p'}{1-p'}\B)^{|X|}\B(\frac{p}{1-p}\B)^{|E(G')|}\,.
\end{align*}

Now, let $Q(G'|G)$ be the distribution given by applying the reverse transition with input $G$, i.e., the distribution of $\Ab(G)$:
\begin{equation*}
    Q(G'|G)=\sum_X  \ind{G=G'+E(X)} \mu_{G}(X)p^{\e(X)-|E(G)|+|E(G')|}(1-p)^{|E(G)|-|E(G')|}\,,
\end{equation*}
where $\mu_{G}(X)\propto  \b(\frac{p'}{1-p'}\b)^{|X|}p^{-\e(X)}$. Dropping the factor depending only on $|G|$ shows that $Q(G'|G)=P(G'|G)$ and proves the lemma.
\end{proof}

For the purpose of reasoning about polynomial-time algorithms, it is crucial that our reduction $\cA$ can be implemented in polynomial time. Fortunately, producing a sample $X\sim \mu_G$ can be done efficiently via the Glauber dynamics Markov chain. 
\begin{restatable}{lemma}{glaubermix}\label{lem:glauber-triangle-dist}
For any fixed graph $G$ over $n$ vertices, $p'\ll 1/n$ and constant $p$, the Glauber dynamics on $\mu_G$ mixes in $O(n^3\log n)$ time. 
\end{restatable}
The Glauber dynamics is defined in Section~\ref{sec:glauber} and the lemma is proved in Section~\ref{sec:conc-mix}.

\subsection{Applying Theorem~\ref{thm:general} to Triangle Removal Algorithm}\label{sec:apply-thm-to-reverse}
Having Theorem~\ref{thm:general} and the transformations in hand, we are ready to prove Theorem~\ref{thm:main}. In this section we focus on the reverse transition, $\Ab$, as it contains a wider range of technical ideas. The proof for the forward transition $\Af$, provided in Section~\ref{sec:forward}, is significantly easier and follows the same high-level outline.

To begin, for the unplanted case Lemma~\ref{lem:reverse} states that
$\Ab(\rgt(n,p,p')) = G(n,p)$, so it remains to prove that $\Ab(\rgt(n,p,p',k,q)) \approx_{o_n(1)} G(n,p,k,g(q))$. We focus here on the case $p'=\Tilde{\Theta}(1/n)$. 

\begin{restatable}[Triangle Removal in Planted Case]{theorem}{PDSreverse}\label{thm:PDSreverse}
    For $k=o(n^{1/4}\log^{-17/4}n)$, $p'=1/(n\log n)$ and $0<p<q<1$ being constant,  
    \[\TV\b(\Ab(\rgt(n,p,p',S,q)),G(n,p,S,q\cdot p_e)\b)=o_n(1)\,,\]
    where $p_e=\E_{G\sim \rgt(n,p,p')}\Pr_{\Ab}(e\in \Ab(G+e))$ for an arbitrary edge $e$.
\end{restatable}

Here $\rgt(n,p,p',S,q)$ stands for the random graph generated by planting a dense subgraph $G(k,q)$ at vertex set $S$ in $\rgt(n,p,p')$. Similarly, $G(n,p,S,q\cdot p_e)$ stands for the random graph generated by planting a dense subgraph $G(k,q\cdot p_e)$ at vertex set $S$ in $G(n,p)$.

To apply our general Theorem~\ref{thm:general}, we need to specify two items: the intermediate planted RGT models and the class of graphs $\cG$. We will then check conditions \ref{cond-gen1}, \ref{cond-gen2}, and \ref{cond-gen3}.

\subsubsection{Intermediate Planted Models}\label{sec:ideas-inter-planted-model}
Define $\RGTi{i}(n,p,p'S,q)$ to be a random graph generated by starting with $\rgt(n,p)$ and independently resampling each edge $e_1,e_2,\cdots, e_i$ to be included with probability $q$.

\subsubsection{Defining Class of Graphs $\cG$} \label{sec:defining-class-of-graphs}
It is not hard to devise examples of graphs $G$ for which each of the conditions \ref{cond-gen1}, \ref{cond-gen2}, \ref{cond-gen3} are violated. 
We define the class of graphs $\cG$ to avoid these bad examples.


The class of graphs $\cG$ is chosen to be $\cG_1\cap \cG_2$, where these are as follows: 
\begin{enumerate}
    \item[$\mathbf{\cG_1}$]  is the set of graphs that are $p^2/3$-uniformly 2-star dense, 
    where a graph is \emph{$c$-uniformly 2-star dense} if
for any pair of nodes $i,j$, there are at least $c(n-2)$ nodes $k$ such that both $(i,k)$ and $(j,k)$ are in the graph;
    and
    \item [$\mathbf{\cG_2}$]is the set of graphs $G$ that satisfy for every $e\in {[n]\choose 2}$ that
\[\b|\mug{G+e}(Y_e=1)-p_e^+\b| = 
C_p n^{-5/2}p'^{-2}\sqrt{\log n}\,,
\quad \text{where}\quad 
p_e^+ = \E_{G'\sim \rgt(n,p,p')}[\mug{G'+e}(Y_e=1)]\,,
\]
and $C_p$ is a fixed constant depending on $p$.
\end{enumerate} 


The following lemma states that $\cG$ is a high probability set for each $\rgt_i$ as required by Theorem~\ref{thm:general}.
It is proved in Section~\ref{sec:graph-class}.

\begin{restatable}{lemma}{classG}\label{lem:classG}
    For any $\RGTi{i}$, $\Pr_{G\sim\RGTi{i}}(G\in \cG)= 1-o(1/n)$.
\end{restatable}

\subsubsection{Checking the Conditions of Theorem~\ref{thm:general}}

We now state the main lemma needed to verify the conditions in Theorem~\ref{thm:general}. The lemma shows insensitivity of $\Ab$ to perturbing the input by a single edge, and contains the bulk of our technical contributions. The key ideas will be presented in Section~\ref{sec:perturb-ideas} with the full proof deferred to Section~\ref{sec:reverse-preserves}.

\begin{restatable}[Perturbation Insensitivity]{lemma}{insensitivity}\label{lem:edge_on_other_edge}
If $G\in \cG_1$,
then
$$
\cL_{G-e}(Y_{\sim e}) = 
\cL_{G+e}(Y_{\sim e}|Y_e=0)\approx_{O(\log^{17/2}n/\sqrt{n})} \cL_{G+e}(Y_{\sim e}|Y_e=1)\,,
$$
and from this it follows that
\[\cL_{G+e}(Y)\approx_{O(\log^{17/2}n/\sqrt{n})} \cL_{G+e}(Y_{\sim e})\times \cL_{G+e}(Y_e)\,.\]
\end{restatable}



%

We first check that condition \ref{cond-gen1} follows from the lemma. 
Since $\cL_{G+e} (Y_{\sim e})$ is a mixture of the laws $\cL_{G+e}(Y_{\sim e}|Y_e=0)$ and $\cL_{G+e}(Y_{\sim e}|Y_e=1)$, \cg{change every $\Tilde O$ or only change it in thm statements}Lemma~\ref{lem:edge_on_other_edge} implies $\cL_{G+e} (Y_{\sim e})\approx_{\tilde O(1/\sqrt{n})} \cL_{G-e} (Y_{\sim e})$. Note that $\Ab$ simply resamples edges in $Y$, so by the data processing inequality for TV,
\[\Ab(G+e)|_{\sim e} \approx_{\tilde{O}(1/\sqrt{n})} \Ab(G-e)|_{\sim e}\,.\]

Condition \ref{cond-gen2} for $G-e$ is trivial: $\cA(G-e)$ has no edge in position $e$, so for any $G$,
\[\cA(G-e)=\cA(G-e)|_{\sim e}\times \bern(0)\,.
\]
For $G+e$, 
the second display of Lemma~\ref{lem:edge_on_other_edge} implies, again by data processing inequality, that
\[
\Ab(G+e)
\approx_{\tilde{O}(1/\sqrt{n})}
\Ab(G+e)|_{\sim e}\times \Ab(G+e)|_e\,.
\]

Lastly, condition \ref{cond-gen3} is immediate for $G-e$, since $\Ab(G-e)|_e = \bern(0)$ for all $G$.
For $G+e$,
it follows from the definition of $\cG_2$: 
For any $G\in \cG_2$, we have by conditioning on the value of $Y_e$ that
\begin{align*}
\Ab(G+e)|_e 
&\sim \mug{G+e}(Y_e=1)\cdot \bern(p) + [1-\mug{G+e}(Y_e=1)] \cdot \bern(1)\\
&= \bern\b(1-(1-p)\mug{G+e}(Y_e=1)\b) 
\\&\approx_{\tilde{O}(1/\sqrt{n})} \bern\B(1-(1-p)\cdot \E_{G'\sim \rgt(n,p,p')}[\mug{G'+e}(Y_e=1)]\B)\,.\end{align*}
The last step used the fact that $\TV(\bern(a), \bern(b)) = |a-b|$.
Let $p_e = \E_{G'\sim \rgt(n,p,p')}\Pr(e\in \Ab(G'+e)) =1-(1-p)\cdot \E_{G'\sim \rgt(n,p,p')}[\mug{G'+e}(Y_e=1)]$.
Combining the last two displayed equations,
\[\Ab(G+e)\approx_{\tilde{O}(1/\sqrt{n})}\Ab(G+e)|_{\sim e}\times \bern(p_e)\,.\]
Therefore, the conditions of Theorem~\ref{thm:general} hold with $q' = q\cdot p_e$, proving Theorem~\ref{thm:PDSreverse}.

\subsection{Showing Perturbation Insensitivity: Main Ideas Behind Lemma~\ref{lem:edge_on_other_edge}}\label{sec:perturb-ideas}

Let us examine more closely the generation of variables in  Lemma~\ref{lem:edge_on_other_edge}.
We fix $G$ and let $X\sim \mu_{G+e} (\spacedot| Y_e=0) = \mu_{G-e}$ and $\Xp\sim \mu_{G+ e}(\,\cdot\, |Y_e=1)$, where recall that $Y$ and $\Yp$ are the corresponding edge indicator vectors as defined in Section~\ref{sec:reverse}.
 Without loss of generality, assume $e\notin G$, so $G-e=G$. 
Our objective is to show that $Y_{\sim e}$ and $\Yp_{\sim e}$ are close in total variation.

Let $X_{\ine}$, $\Xp_{\ine}$, and $X_{\nine}$, $\Xp_{\nine}$ be the triangle indicator vectors restricted to the set of triangles that contain and do not contain $e$, respectively.
The proof of Lemma~\ref{lem:edge_on_other_edge} is divided into two conceptual parts: (1) insensitivity to perturbing $Y_e$ of triangles $X_{\nine}$ that do not include $e$, and (2) conditioning on $X_{\nine}$, addressing the difference of edges brought about by triangles $X_{\ine}$ that include $e$. 
 We decompose $\TV(Y_{\sim e},\Yp_{\sim e})$ into two terms (using Lemma~\ref{lem:TVdecomp} in Section~\ref{sec:TV}):
\begin{equation}\label{eq:invariance-decomposition}
    \TV(Y_{\sim e},\Yp_{\sim e})
    \le \TV(X_{\nine},\Xp_{\nine}) + \TV\b(Y_{\sim e}, \E_{X'\sim X_{\nine}}
    \law(\Yp_{\sim e}|\Xp_{\nine}=X')
    \b)
    \,.
\end{equation}
We next describe how to bound each of the two terms on the right-hand side. 

\subsubsection{Edge $e$ Has Low Influence on Non-Containing Triangles}\label{sec:non-incident-triangles}

We first bound $\TV(X_{\nine},\Xp_{\nine})$.
Only in Section~\ref{sec:non-incident-triangles} we will 
pretend that $\Xp\sim \mu_{G+e}$ rather than $\mu_{G+e}(\spacedot|Y_e=1)$, and this turns out to be valid as the two distributions have roughly the same variation distance to $X_{\nine}$ because $\mu_{G+e}$ is a constant-weight mixture of $\mu_{G+e}(\spacedot|Y_e=1)$ and $\mu_{G-e}$. 

Note that because $e\notin G$, $X_{\ine}$ is always 0, so $X_{\nine}$ has the same distribution as $X$, defined by $\mu_{G}$. As for $\Xp_{\nine}$, letting $T_e$ be the set of triangles containing $e$,
$$
    \law(\Xp_{\nine}) = \sum_{x\in \{0,1\}^{T_e}} \Pr(\Xp_{\ine}=x)\law(\Xp_{\nine}|\Xp_{\ine}=x)
    $$
is a mixture of Gibbs distributions indexed by the value of $\Xp_{\ine}$. 
By Pinsker's inequality we have $\TV(X_{\nine},\Xp_{\nine})^2 \leq \chi^2(X_{\nine}\| \Xp_{\nine})$, and we bound the latter quantity via Ingster's 2nd moment method \cite{ingster2003nonparametric}. We emphasize that both $\law(X_{\nine})$ and $\law(\Xp_{\nine})$ are complicated dependent distributions, while all prior works to the best of our knowledge have always shown bounds between mixtures of \emph{product distributions} (see, for instance, \cite{ingster2003nonparametric,berthet2013optimal,hajek2015computational,brennan2019phase}). We show that despite this dependence it is still possible to derive a tractable bound.

\begin{restatable}[$\chi^2$-divergence for mixture of Gibbs measures]{lemma}{mixturegibbs}\label{lem:chi-square-gibbs}
Let $P$ be a distribution defined by 
\[P(X)=f(X)/Z.\]
Let $U$ be a discrete random variable and $Q$ be a mixture of Gibbs distributions defined by 
\[Q(X)=\E_U [f_U(X)/Z_U]\,.\]
Letting $\rho_U(X)=f_U(X)/f(X)$, we have that
\[\chi^2(Q\|P)+1 = \E_{U,U'}\frac{\E_{X\sim P} \rho_U(X)\rho_{U'}(X)}{\E_{X\sim P} \rho_U(X)\E_{X\sim P} \rho_{U'}(X)}\,,\]
where $U'$ is an independent and identically distributed copy of $U$.
\end{restatable}

It turns out that when applying Lemma~\ref{lem:chi-square-gibbs} to $X_{\nine}$ versus $\Xp_{\nine}$, we can bound the $\chi^2$-divergence via marginal influences of the edge distribution $\law(Y)$. Here the marginal influence $\IM{A'}{e}$ characterizes how much configurations on $A'$ can affect the conditional marginal probability on edge $e$. A formal definition of marginal influence can be found in Definition~\ref{def:marginal-influence}.


\begin{restatable}{lemma}{expressioninfluence}\label{lem:expression-influence}
Suppose $G\in \cG$ and consider $\law_G(Y)$ for $Y$ as in Section~\ref{sec:reverse} and $p'=o(1/n)$. Let $\IM {A'}{e}$ be the marginal influence of $A'$ on $e$ for $\law_G(Y)$.
Letting $\Xpet$ be an independent copy of $\Xpe$, we have
    \[\chi^2(\Xp_{\nine}\|X_{\nine})+1  \le  \E_{\Xpe,\Xpet} \min\Bigg\{ \B(1+\frac{1-p}{p}\sup_{\substack{A'\subset A\cup B
\\ e\in (A\cup B)\backslash A'}} \IM{A'}{e}\B)^{|B|}(1/p)^{|A\cap B|} , (1/p)^{|A|+|B|} \Bigg\}\,,\]
    where $A = E( \Xpet)-e$ and $B = E(\Xpe)-e$.
\end{restatable}



The marginal influence $\IM{A'}{e}$ can be bounded by exploiting the fact that $\mu_G$ has small marginal probabilities under arbitrary conditioning, and we prove a general bound to this effect in Section~\ref{sec:small-marginal}.  

\begin{lemma}\label{lem:edge_influence}
Suppose $G\in \cG$ and consider $\law_G(Y)$ for $Y$ as in Section~\ref{sec:reverse} and $p'=o(1/n)$. Let $\IM A{e}$ be the marginal influence of $A$ on $e$ for $\law_G(Y)$. For an edge set $A\subset E(G)$ with $|A|=O(n)$,
\[\IM A{e}=\tilde{O}(|A|/n)\,.\]
\end{lemma}

With high probability, $|A|$ and $ |B|$ each have size $\tilde O(1)$ and $|A\cap B|=0$, so the right hand of Lemma~\ref{lem:expression-influence} is $1+\tilde O(1/n)$ with high probability. Of course, the tail distribution of $|A|$ and $|B|$ is important, and we will make the bound rigorous in Section~\ref{sec:reverse-preserves} to get that $\chi^2(X_{\nine}\|\Xp_{\nine}) = \tilde O(1/n)$.
By Pinsker's inequality, this shows $\TV(X_{\nine},\Xp_{\nine}) = \tilde O(1/\sqrt{n})$.





\subsubsection{Triangles Including $e$ and TV Between Projections of Distributions}\label{sec:projection}
Recalling the decomposition \eqref{eq:invariance-decomposition}, we now aim to bound the second term,
\begin{equation}\label{eq:adjacent-triangle}
    \TV\b(Y_{\sim e}, \E_{X'\sim X_{\nine}}
    \law(\Yp_{\sim e}|\Xp_{\nine}=X')
    \b)
    \,.
\end{equation}
This can be rewritten as the more symmetric expression \[\TV\b(\E_{X'\sim X_{\nine}}
    \law(Y_{\sim e}|X_{\nine}=X'), \E_{X'\sim X_{\nine}}
    \law(\Yp_{\sim e}|\Xp_{\nine}=X')
    \b)\,.\]
Since $Y_{\sim e}|X$ and $\Yp_{\sim e}|\Xp$ are projections of triangle variables $X$ and $\Xp$ onto the edge space, the most natural approach to establish their closeness is to show that $\law(X_{\ine}|X_{\nine})$ and $\law(\Xp_{\ine}|{\Xp_{\nine}})$ are close and appeal to data processing inequality. 
However, this is not true: 
As mentioned earlier, $X_{\ine}=\Vec{0}$ since $e\notin G$, and in contrast $\Xp_{\ine}$ is non-zero with non-negligible probability, since there are $\Theta(n)$ triangles containing to $e$ and each is selected with probability approximately $p'=\tilde \Theta(1/n)$.

Nevertheless, the fact that we are proving a statement about projection onto the edge space, $Y_{\sim e}$ and $\Yp_{\sim e}$, allows us to carry out manipulations in the triangle space before projection. 
We will design an auxiliary distribution $\aux$ over the same support as $X_{\nine}$ that when added to $\Xp_{\nine}$ results in the \emph{identical} edge projection as $\Xp_{e}$, i.e.,  
\[E(\aux\vee \Xp_{\nine})-e = E(\Xp_{\ine}\vee \Xp_{\nine})-e\,.\]
This means the edge indicator vector of $\aux\vee \Xp_{\nine}$, which we denote by $\tilde Y_{\sim e}$, is the same as $\Yp_{\sim e}$. Thus, instead of comparing $Y_{\sim e}$ and $\Yp_{\sim e}$ in \eqref{eq:adjacent-triangle}, it suffices to compare $Y_{\sim e}$ and $\tilde Y_{\sim e}$, and by data processing inequality it in turn suffices to compare
$X_{\nine}$ and $\Xp_{\nine} \vee \aux$ (see Fig.~\ref{fig:projection}):
\begin{equation}\label{e:aux}
    \TV\b(Y_{\sim e}, \Yp_{\sim e}
    \b)
    \le \TV\b(X_{\nine}, \Xp_{\nine} \vee \aux\b)
\end{equation}

We now define $\aux$. Each triangle in $\Xp_{\ine}$ adds at most two edges to $\Yp_{\sim e}$. To simulate this change without using any triangles containing $e$, we add a triangle to $\aux$ for each new edge introduced by $\Xp_{\ine}$. Crucially, this is done without adding other new edges, see Fig.~\ref{fig:aux}. Avoiding adding new edges is possible with high probability as long as the graph $G$ is sufficiently dense (here $c$-uniformly 2-star dense plays a role) and $p' = \tilde\Theta(1/n)$, which implies that the edge set of $\Xp_{\nine}$ has sufficient coverage of relevant triangles that we might potentially add to $\aux$. 
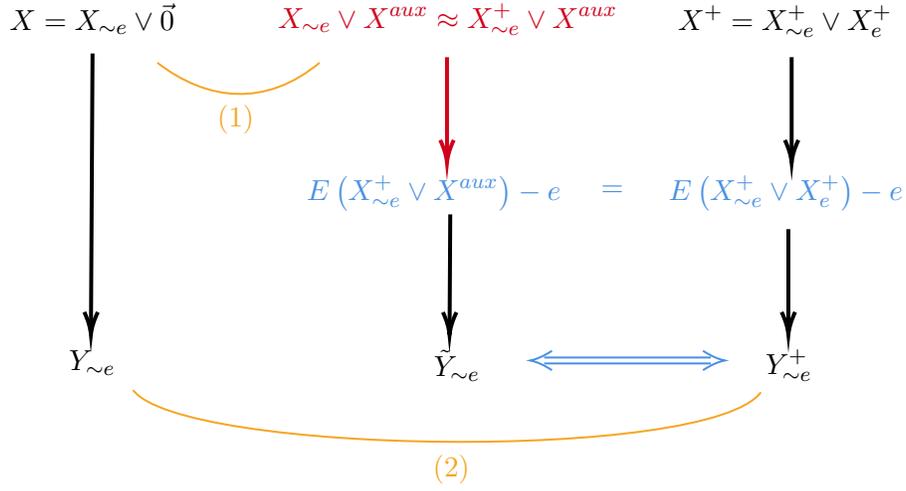
\begin{figure}
    \centering
    \tikzset{every picture/.style={line width=0.75pt}} 

\begin{tikzpicture}[x=0.75pt,y=0.75pt,yscale=-1,xscale=1]

\draw [color={rgb, 255:red, 208; green, 2; blue, 27 }  ,draw opacity=1 ][line width=1.5]    (237.71,36.87) -- (237.71,87.08) ;
\draw [shift={(237.71,90.08)}, rotate = 270] [color={rgb, 255:red, 208; green, 2; blue, 27 }  ,draw opacity=1 ][line width=1.5]    (11.37,-3.42) .. controls (7.23,-1.45) and (3.44,-0.31) .. (0,0) .. controls (3.44,0.31) and (7.23,1.45) .. (11.37,3.42)   ;
\draw [line width=1.5]    (58.95,35) -- (58.15,173.89) ;
\draw [shift={(58.13,176.89)}, rotate = 270.33] [color={rgb, 255:red, 0; green, 0; blue, 0 }  ][line width=1.5]    (11.37,-3.42) .. controls (7.23,-1.45) and (3.44,-0.31) .. (0,0) .. controls (3.44,0.31) and (7.23,1.45) .. (11.37,3.42)   ;
\draw [line width=1.5]    (411.53,36.87) -- (411.53,87.08) ;
\draw [shift={(411.53,90.08)}, rotate = 270] [color={rgb, 255:red, 0; green, 0; blue, 0 }  ][line width=1.5]    (11.37,-3.42) .. controls (7.23,-1.45) and (3.44,-0.31) .. (0,0) .. controls (3.44,0.31) and (7.23,1.45) .. (11.37,3.42)   ;
\draw [line width=1.5]    (239.36,116.21) -- (239.02,174) ;
\draw [shift={(239,177)}, rotate = 270.34] [color={rgb, 255:red, 0; green, 0; blue, 0 }  ][line width=1.5]    (11.37,-3.42) .. controls (7.23,-1.45) and (3.44,-0.31) .. (0,0) .. controls (3.44,0.31) and (7.23,1.45) .. (11.37,3.42)   ;
\draw [line width=1.5]    (409.89,123.68) -- (409.89,173.89) ;
\draw [shift={(409.89,176.89)}, rotate = 270] [color={rgb, 255:red, 0; green, 0; blue, 0 }  ][line width=1.5]    (11.37,-3.42) .. controls (7.23,-1.45) and (3.44,-0.31) .. (0,0) .. controls (3.44,0.31) and (7.23,1.45) .. (11.37,3.42)   ;
\draw [color={rgb, 255:red, 245; green, 166; blue, 35 }  ,draw opacity=1 ]   (91.66,39.67) .. controls (117.02,59.27) and (146.46,62.07) .. (173.45,39.67) ;
\draw [color={rgb, 255:red, 245; green, 166; blue, 35 }  ,draw opacity=1 ]   (79.39,204.89) .. controls (103.93,237.56) and (369.83,241.47) .. (396,206) ;
\draw [color={rgb, 255:red, 74; green, 144; blue, 226 }  ,draw opacity=1 ][line width=0.75]    (286.73,188.46) -- (370.76,188.46)(286.73,191.46) -- (370.76,191.46) ;
\draw [shift={(377.76,189.96)}, rotate = 180] [color={rgb, 255:red, 74; green, 144; blue, 226 }  ,draw opacity=1 ][line width=0.75]    (10.93,-4.9) .. controls (6.95,-2.3) and (3.31,-0.67) .. (0,0) .. controls (3.31,0.67) and (6.95,2.3) .. (10.93,4.9)   ;
\draw [shift={(279.73,189.96)}, rotate = 0] [color={rgb, 255:red, 74; green, 144; blue, 226 }  ,draw opacity=1 ][line width=0.75]    (10.93,-4.9) .. controls (6.95,-2.3) and (3.31,-0.67) .. (0,0) .. controls (3.31,0.67) and (6.95,2.3) .. (10.93,4.9)   ;

\draw (15.67,8.07) node [anchor=north west][inner sep=0.75pt]   [align=left] {$\displaystyle X=X_{\sim e} \lor \vec{0}$};
\draw (150.98,8.17) node [anchor=north west][inner sep=0.75pt]  [color={rgb, 255:red, 208; green, 2; blue, 27 }  ,opacity=1 ] [align=left] {$\displaystyle X_{\sim e} \lor X^{aux} \approx X_{\sim e}^{+} \lor X^{aux}$};
\draw (352.74,9.1) node [anchor=north west][inner sep=0.75pt]   [align=left] {$\displaystyle X^{+} =X_{\sim e}^{+} \lor X_{e}^{+}$};
\draw (165.62,95.01) node [anchor=north west][inner sep=0.75pt]  [color={rgb, 255:red, 74; green, 144; blue, 226 }  ,opacity=1 ] [align=left] {$\displaystyle E\left( X_{\sim e}^{+} \lor X^{aux}\right) -e$};
\draw (348.15,95.01) node [anchor=north west][inner sep=0.75pt]  [color={rgb, 255:red, 74; green, 144; blue, 226 }  ,opacity=1 ] [align=left] {$\displaystyle E\left( X_{\sim e}^{+} \lor X_{e}^{+}\right) -e$};
\draw (229.57,182.99) node [anchor=north west][inner sep=0.75pt]   [align=left] {$\displaystyle \tilde{Y}_{\sim e}$};
\draw (396.97,182.66) node [anchor=north west][inner sep=0.75pt]   [align=center] {$ $$\displaystyle Y_{\sim e}^{+}$};
\draw (45.66,183.39) node [anchor=north west][inner sep=0.75pt]   [align=left] {$ $$\displaystyle Y_{\sim e}$};
\draw (120.44,59.61) node [anchor=north west][inner sep=0.75pt]  [color={rgb, 255:red, 245; green, 166; blue, 35 }  ,opacity=1 ] [align=left] {$\displaystyle ( 1)$};
\draw (229.18,236.03) node [anchor=north west][inner sep=0.75pt]  [color={rgb, 255:red, 245; green, 166; blue, 35 }  ,opacity=1 ] [align=left] {$\displaystyle ( 2)$};
\draw (314,100) node [anchor=north west][inner sep=0.75pt]  [color={rgb, 255:red, 74; green, 144; blue, 226 }  ,opacity=1 ] [align=left] {$\displaystyle =$};

\end{tikzpicture}
    \caption{Relation between variables. By the data processing inequality, which implies that total variation shrinks under projection, closeness of (1) implies closeness of (2).}
    \label{fig:projection}
\end{figure}
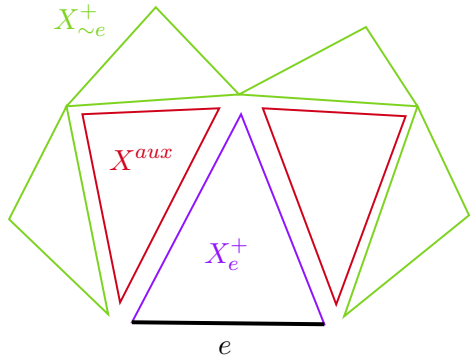
\begin{figure}
    \centering
    \tikzset{every picture/.style={line width=0.75pt}} 

\begin{tikzpicture}[x=0.75pt,y=0.75pt,yscale=-1,xscale=1]

\draw [color={rgb, 255:red, 144; green, 19; blue, 254 }  ,draw opacity=1 ]   (76,177) -- (131,72) -- (173,178) -- (97.98,177.23) -- cycle ;
\draw [color={rgb, 255:red, 208; green, 2; blue, 27 }  ,draw opacity=1 ]   (70,167) -- (51,72) -- (120,69) -- cycle ;
\draw [color={rgb, 255:red, 208; green, 2; blue, 27 }  ,draw opacity=1 ]   (179,168) -- (142,69) -- (214,73) -- cycle ;
\draw [color={rgb, 255:red, 126; green, 211; blue, 33 }  ,draw opacity=1 ]   (88,18) -- (130,62) -- (43,68) -- cycle ;
\draw [color={rgb, 255:red, 126; green, 211; blue, 33 }  ,draw opacity=1 ]   (43,68) -- (64,173) -- (14,125) -- cycle ;
\draw [color={rgb, 255:red, 126; green, 211; blue, 33 }  ,draw opacity=1 ]   (220,68) -- (246,123) -- (183,174) -- cycle ;
\draw [color={rgb, 255:red, 126; green, 211; blue, 33 }  ,draw opacity=1 ]   (194,28) -- (130,62) -- (220,68) -- cycle ;
\draw [line width=1.5]    (76,177) -- (173,178) ;

\draw (35,15) node [anchor=north west][inner sep=0.75pt]  [color={rgb, 255:red, 126; green, 211; blue, 33 }  ,opacity=1 ] [align=left] {$\displaystyle X_{\sim e}^{+}$};
\draw (111,133) node [anchor=north west][inner sep=0.75pt]  [color={rgb, 255:red, 144; green, 19; blue, 254 }  ,opacity=1 ] [align=left] {$\displaystyle X_{e}^{+}$};
\draw (63,88) node [anchor=north west][inner sep=0.75pt]  [color={rgb, 255:red, 208; green, 2; blue, 27 }  ,opacity=1 ] [align=left] {$\displaystyle X^{aux}$};
\draw (118,185) node [anchor=north west][inner sep=0.75pt]   [align=left] {$\displaystyle e$};

\end{tikzpicture}
    \caption{Design of $\aux$. For each edge in $E(\Xp_{\ine})-E(\Xp_{\nine})-e$ (purple edges), $\aux$ includes a triangle (in red) containing the edge, without adding other edges not already in $E(\Xp_{\nine})$ (in green).}
    \label{fig:aux}
\end{figure}

From the previous section, we have $X_{\nine}\approx_{\tilde O(1/\sqrt{n})} \Xp_{\nine}$, so to show \eqref{e:aux} it remains to prove that $X_{\nine}\approx X_{\nine}\vee \aux$, i.e., the addition of $\aux$ must be undetectable. We show this via concentration of the likelihood ratio between $X_{\nine}\vee \aux$ and $X_{\nine}$. In Section~\ref{sec:small-marginal}, we prove that Lipschitz functions over $\mu_G$ concentrate. 
%
The likelihood ratio under consideration is not quite Lipschitz -- it is Lipschitz only over a high probability subset -- but this turns out to suffice for it to concentrate.





\section{Strong Equivalence Between Parameterized Problems}
\label{sec:equiv}
Each problem we consider is actually an entire \emph{family} of problems. 
Our results turn out to imply a strong form of computational equivalence between families of problems. This form of equivalence is new, so we carefully introduce the relevant definitions.


\paragraph{Average-Case Reduction in Total Variation}

Let $\cP = \{( P_{0,n}, P_{1,n}),n\in \mathbb{N}\}$
be a hypothesis testing problem between 
\[H_0: X\sim P_{0,n}\quadand H_1: X\sim P_{1,n}\,,\]
and $\cP' = \{( P_{0,n}', P_{1,n}'),n\in \mathbb{N}\}$ be another hypothesis testing problem between
\[H_0: X\sim P_{0,n}'\quadand H_1: X\sim P_{1,n}'\,.\]
An \emph{average-case reduction} from $\cP$ to $\cP'$ is a randomized mapping $\cA:\Omega\rightarrow \Omega$, computable in polynomial time, such that $\TV(\cA( P_{0,n}),  P_{0,n}')=o_n(1)$ and $\TV(\cA( P_{1,n}),  P_{1,n}')=o_n(1)$.\footnote{This notion of reduction is based on \emph{Le Cam deficiency} \cite{le2012asymptotic} and has been (sometimes implicitly) used in the recent line of work on average-case reductions for statistical problems, e.g., \cite{berthet2013complexity,ma2015computational,hajek2015computational,cai2016optimal,brennan2018reducibility,brennan2019optimal,brennan2020reducibility}. Levin's \cite{levin1986average}
notion of average-case reduction is similar but tailored for dist-NP completeness results, and is more stringent in the error probability and less stringent in the distance between distributions.}

It is clear that if there exists an average case reduction $\cA$ from $\cP$ to $\cP'$, then any polynomial-time algorithm $\Phi$ that has vanishing error on $\cP'$ yields a polynomial-time algorithm $\Phi\circ \cA$ with vanishing error on $\cP$. Therefore, if $\cP$ is computationally hard (or information-theoretically impossible), then $\cP'$ is also computationally hard (or information-theoretically impossible).

\paragraph{Parameterized Problems and Reductions}
Most statistical problems have several parameters. In our case of planted RGT, for a given graph size $n$ we have parameters for the dense-subgraph size $k$, ambient edge probability $p$,  triangle probability $p'$, and edge density $q$ in the dense subgraph. 
Each tuple of parameters describes a distinct problem instance from within the larger family, and a reduction maps between a problem at a specific parameter to another problem at some parameter. 
We use $\cP(\vec{\alpha}_n) = \{( P_{0,n}(\vec{\alpha}_n),  P_{1,n}(\vec{\alpha}_n)), \vec{\alpha}_n\in S, n\in \mathbb{N} \}$ to denote a family of parameterized problems. Here  $S\subset \mathbb{R}^a$ is assumed to be compact, which can be achieved for all problems we know of by using an appropriate parameterization. For instance, for PDS with constant $p$ and $q$, one can take $\vec{\alpha}_n = (\log_n k, p,q)$. For the regime $p = c q$ and $q=n^{-b}$ for $b>0$ as studied in \cite{hajek2015computational} one can take $\vec{\alpha}_n = (\log_n k, \log_n (1/q), c)$.

\paragraph{Phase Diagram} 
In the literature on the computational complexity of statistics problems, feasibility of solving a problem is usually described in the language of phase diagrams (see, e.g., \cite{ma2015computational,hajek2015computational,brennan2018reducibility,schramm2022computational}). These diagrams divide the parameter space into regions corresponding to efficiently solvable, computationally hard, or information-theoretically impossible. 
In order to describe the relation between families of problems, we need to introduce a precise definition that captures the meaning of existing phase diagrams.

For a parameterized family of hypothesis testing problems $\cP  = \{( P_{0,n}(\vec{\alpha}_n),  P_{1,n}(\vec{\alpha}_n)), \vec{\alpha}_n\in S, n\in \mathbb{N} \}$, the \emph{Phase Diagram} is a mapping $\pd_{\cP}:S\rightarrow \{\easy,\hard,\impo,\perp\}$: 
\begin{itemize}
    \item $\pd_{\cP}(\vec{\alpha}) = \easy$ if and only if for any sequence $\vec{\alpha}_n$ converging to $\vec{\alpha}$, $\cP(\vec{\alpha}_n)$ can be solved efficiently with vanishing error as $n\to \infty$.
    \item $\pd_{\cP}(\vec{\alpha}) = \hard$ if and only if for any sequence $\vec{\alpha}_n$ converging to $\vec{\alpha}$, $\cP(\vec{\alpha}_n)$ can only be solved by super-polynomial algorithms.
    \item  $\pd_{\cP}(\vec{\alpha}) = \impo$ if and only if for any sequence $\vec{\alpha}_n$ converging to $\vec{\alpha}$, solving $\cP(\vec{\alpha}_n)$ is information-theoretically impossible.
    \item  $\pd_{\cP}(\vec{\alpha}) = \perp$, otherwise.
    
\end{itemize}

\begin{remark}\label{rem:phase-diagram-paramter}
It is important to note that the phase diagram of a problem depends on the parameterization: Some parameterizations lead to meaningless phase diagrams. 

For example, for the PDS problem with constants $p\not=q$, if the parameters are chosen to be $\vec{\alpha}_n =  (k/n,p,q)$, then the phase diagram becomes trivial. For any $0<c\le 1$, $\pd_{\cP} (c,p,q)=\easy$ as it suffices to count the number of edges and tell the difference between the two hypothesises. And $\pd_{\cP} (0,p,q)=\perp$, because we know the problem is information-theoretically impossible for constant $k$ and easy for $k=\Theta(\sqrt{n})$, and $k/n$ converges to 0 in both regimes. On the other hand, if we use the parameterization $\vec{\alpha}_n = (\log_n k,p,q)$, the conjectured non-trivial phase transition will appear when $\log_n k=1/2$. If we believe the computational threshold happens at $k=\tilde \Theta(\sqrt{n} )$, then 
\[\pd_{\cP}(c,p,q) = 
\begin{cases}
\hard &\text{ if } 0<c<1/2\\
\easy &\text{ if } 1/2<c\le 1\\
\perp &\text{ if } c=1/2\,.
\end{cases}\]
One can also study the more refined behavior near the phase transition by introducing extra parameters. For instance, in the case of PDS, one can use the parameterization $\vec{\alpha}_n = (a,b,p,q)$ where $k=n^a\log^b n$. This allows the study of more fine-grained phase transition at $k=\sqrt{n}\log^b n$.
\end{remark}

\begin{remark}
Although $\vec{\alpha}_n$ belongs to $S$ for every $n$, it is crucial that $S$ remains a compact set, independent of $n$. In the case of PDS, the size of the planted subgraph, $k$, cannot be selected directly as a parameter, as it would diverge with $n$. To ensure that the asymptotic behavior of the problem is represented by a single point on the phase diagram, $\vec{\alpha}_n$ should remain constant or converge to a constant as $n$ increases.
\end{remark}



\paragraph{Strong Equivalence Between Problems}

To describe the relation between the phase diagrams of two parameterized families of problems, $\cP(\vec{\alpha}_n)$ where $\vec{\alpha}_n\in S$ and $\cP'(\vec{\beta}_n)$ where $\vec{\beta}\in S'$, we say that $\vec{\alpha}\in S$ \emph{maps} to $\vec{\beta}\in S'$ if 
\begin{itemize}
    \item for any sequence $\vec{\alpha}_n$ converging to $\vec{\alpha}$, there exists a sequence $\vec{\beta}_n$ converging to $\vec{\beta}$ with an average case reduction from $\cP(\vec{\alpha}_n)$ to $\cP'(\vec{\beta}_n)$, and
    \item for any sequence $\vec{\beta}_n$ converging to $\vec{\beta}$, there exists a sequence $\vec{\alpha}_n$ converging to $\vec{\alpha}$ with an average case reduction from $\cP(\vec{\alpha}_n)$ to $\cP'(\vec{\beta}_n)$.
\end{itemize}
This definition yields the following correspondence of phase diagrams.

\begin{lemma}
Let $\cP(\vec{\alpha}_n)$, where $\vec{\alpha}_n\in S$, be a parameterized family of problems with phase diagram $\pd_\cP$. And $\cP'(\vec{\beta}_n)$, where $\vec{\beta}_n\in S'$, is a parameterized family of problems with phase diagram $\pd_{\cP'}$. If $\vec{\alpha}\in S$ maps to and can be mapped from $\vec{\beta}\in S'$, then \[\pd_{\cP}(\vec{\alpha}) = \pd_{\cP'}(\vec{\beta})\,.\]
\end{lemma}
\begin{proof}
In this proof, we use the convention that information-theoretically impossible is harder than computationally hard, which is harder than efficiently solvable. If two problem are in the same category, they are considered to have same hardness.

Let $\cP(\vec{\alpha}_n^h)$ (or $\cP(\vec{\alpha}_n^e)$) be one of the hardest (or easiest) problems where $\vec{\alpha}_n$ converges to $\vec{\alpha}$. $\cP'(\vec{\beta}_n^h)$ and $\cP'(\vec{\beta}_n^e)$ are defined in the same way. 
Since $\vec{\alpha}$ maps to $\vec{\beta}$, there exists a problem $\cP'(\vec{\beta}_n)$ that can be reduced from $\cP(\vec{\alpha}_n^h)$ with $\vec{\beta}_n$ converging to $\vec{\beta}$. So $\cP'(\vec{\beta}_n^h)$ is at least as hard as $\cP'(\vec{\beta}_n)$, and therefore, at least as hard as $\cP(\vec{\alpha}_n^h)$. Similarly, we can conclude that $\cP'(\vec{\beta}_n^h)$ 
 is at least as hard as $\cP(\vec{\alpha}_n^h)$, which means they have the same hardness. 

 On the other hand, since $\vec{\alpha}$ maps to $\vec{\beta}$, there exists a problem $\cP(\vec{\alpha}_n)$ where $\vec{\alpha}_n$ converges to $\vec{\alpha}$ that $\cP(\vec{\alpha}_n)$ reduces to $\cP'(\vec{\beta}_n^e)$. So $\cP'(\vec{\beta}_n^e)$ is at least as hard as $\cP(\vec{\alpha}_n)$, and therefore, at least as hard as $\cP(\vec{\alpha}_n^e)$. We can also conclude that $\cP(\vec{\alpha}_n^e)$ and $\cP'(\vec{\beta}_n^e)$ have the same hardness. It is not hard to see from the definition of $\pd$ that $\pd_{\cP}(\vec{\alpha}) = \pd_{\cP'}(\vec{\beta})$.
\end{proof}

We say two families of problems $\cP$ and $\cP'$ are \emph{strongly equivalent} if for any point $\vec{\alpha}\in S$, there exists $\vec{\beta}\in S'$ such that $\vec{\alpha}$ maps to and can be mapped from $\vec{\beta}$, and similarly, for any point $\vec{\beta}\in S'$, there exists $\vec{\alpha}\in S$ such that $\vec{\beta}$ maps to and can be mapped from $\vec{\alpha}$. 

\begin{fact}
If two problems are strongly equivalent, then the phase diagram of either one of the problems determines the phase diagram of the other problem.
\end{fact}

\begin{remark}
    The phase diagrams of the same problem under different parameterizations may not be strongly equivalent. One such example was discussed in Remark~\ref{rem:phase-diagram-paramter}.
\end{remark}

The main result of the paper, Theorem~\ref{thm:main}, can be re-stated as the following.


\begin{theorem}
Consider the PDS problem, i.e., hypothesis testing between $G(n,p)$ and $G(n,p,k,q)$, parameterized by $\vec{\alpha}_n = (\log_n k,p,q)$, restricted to the parameter region 
$\log_n k<   1/4$,
$0<p,q<1$. 
The Planted RGT problem has hypotheses $\rgt(n,p,p')$ and $\rgt(n,p,p',k,q)$ and is parameterized by $\vec{\beta} = (p,\log_n p', \log_n k,q)$ where the parameter region is $\log_n p'< -1$, $\log_n k< 1/4$, $0<p,q<1$. The PDS and Planted RGT problems are strongly equivalent in the given parameter regimes.
\end{theorem}

\section{Preliminaries}

\subsection{Probability Notation}

The probability law of a random variable $X$ is denoted by $\cL(X)$. It will often be necessary to condition on a graph $G$ for a jointly distributed $X$ and $G$, and we use the shorthand $\cL_G(X)$ in place of $\cL(X|G)$. 

Any randomized algorithm $\cA$ defines a Markov kernel where the distribution of the output for given input $x$ follows the law of $\cA(x)$.
When an algorithm $\cA$ is applied to a distribution $\mu$, $\cA(\mu)$ stands for the law of $\cA(X)$ where $X\sim \mu$. 
For random variables $X,Y$ we use $\TV(X,Y)$ in place of $\TV(\law(X), \law(Y))$.
For the ease of presentation, we use $\TV(P,Q)\le \epsilon$ and $P\approx_{\epsilon} Q$ interchangeably.

For two vectors $X$ and $Y$, $X\vee Y$ is the coordinate-wise max.

We say that an event occurs with high probability if its probability is at least $1-\frac{1}{n^c}$ for a large enough constant $c$, i.e., large enough for subsequent use (which in our case entails losing some numerical constant in $c$ and still having the probability tend to one).

\subsection{Graph Theory Notation}\label{sec:graph}
$G$ is always used to denote a graph. In general we call a pair of vertices $e=\{i,j\}\in {[n]\choose2}=K_n$ an \emph{edge}, regardless of whether $e$ is actually in the particular graph $G$ under consideration. If it is, then we say that edge $e$ is present, or included, in $G$. A graph $G$ is considered a subset of edges in $K_n$, or a set of indicators in $\{0,1\}^{K_n}$ for whether an edge is included in $G$. We also use the convention that for any $E\subset K_n$, $G|_{E}$ stands for the intersection $G\cap E$.  

Let $N_k(v)$ stands for the set all vertices in $G$ with distance at most $k$ to $v$ without $v$ itself. Two triangles in the graph are considered adjacent or having distance 1 if they share an edge. The distance between two triangles are defined by the length of a shortest path between them. Similarly, two edges in a graph are adjacent or having distance 1 if they share a node. The distance between two edges are defined by the length of a shortest path between them.

For a set of triangles $T$ (as a binary vector with 1 indicating that a triangle is present), $E(T)$ denotes the edge set of the union of the triangles in $T$. 

For a graph $G$ and a pair of vertices $u,v$, the \emph{set of wedge vertices} $\wed[G] (u,v)$ is the set of vertices that are adjacent to both $u$ and $v$ (excluding $u$ and $v$ themselves):
\begin{equation}\label{eq:def-wedge-set}
  \wed[G](u,v) = \{w\in [n]| (w,u),(w,v)\in G\}  \,.
\end{equation}



\subsection{Total Variation Distance}\label{sec:TV}

The \emph{total variation distance} between two distributions $P$ and $Q$ on the same space $\Omega$ is defined as $\TV(P,Q)=\sup_{A\subseteq \Omega} P(A) - Q(A)$.  
The total variation is convex as function of the pair of distributions $(P,Q)$. 
We state several lemmas that will be used to bound total variation distance.

\begin{lemma}[Data Processing Inequality (DPI)] 
For any two random variables $X,Y$ on space $\Omega$ and any randomized algorithm (Markov transition kernel) $R:\Omega\to \Omega'$ for arbitrary space $\Omega'$, we have
\[\TV\b(R(X),R(Y)\b)\le \TV(X,Y)\,.\]
\end{lemma}

\begin{lemma}[Convexity and Mixtures]\label{lem:coupling_TV}
Let $P$ and $Q$ be two distributions over the same space $\Omega$. If there is a joint distribution $(X,X',Y)$ where the marginal distributions of $X$ and $X'$ are $P$ and $Q$, respectively, we have
\[\TV(P,Q)\le \E_{Y} \TV(P_{X|Y},Q_{X'|Y})\,.\]
\end{lemma}
\begin{proof}
Since $\E_YP_{X|Y}=P$ and $\E_{Y}P_{X'|Y}=Q$, the statement follows from convexity of $\TV$.
\end{proof}

\begin{lemma}[TV Decomposition]\label{lem:TVdecomp}
Consider two pairs of random variables, $(X,Y)$ and $(X',Y')$. We have
\[\TV(Y,Y')\le \TV(X,X') +  \TV\b(Y,\E_{Z\sim \law(X)} \law(Y'|X' = Z) \b)\,.\]
\end{lemma}
\begin{proof}
Consider an intermediate distribution $\E_{Z\sim \law(X)} \law(Y'|X' = Z) $.
By the triangle inequality,
\begin{align*}
    \TV\b(Y,Y'\b)&\le \TV\b(Y\,,\E_{Z\sim \law(X)} \law(Y'|X' = Z)\b)\\
    &\qquad + \TV\b(\E_{Z\sim \law(X')} \law(Y'|X' = Z)\,,\E_{Z\sim \law(X)} \law(Y'|X' = Z)\b)\,.
\end{align*}
By the data processing inequality, the second term is bounded by 
$\TV(X,X')$. 
\end{proof}

\begin{lemma}\label{lem:bad_event}
Let $P$ and $Q$ be two distributions over the same space $\Omega$, and suppose that $Q$ is absolutely continuous with respect to $P$. For any event $A\subseteq \Omega$,
\[\TV(P,Q)\le 2(P(A)+Q(A))+(1-Q(A))\TV(P|_{ A^c},Q|_{ A^c})\,.\]
\end{lemma}
\begin{proof} Several applications of the triangle inequality yield
\begin{align*}
    \TV(P,Q)
    &= \E_P\B|\frac{dQ}{dP}-1\B|\\
    &= \E_P\left[ \1_A\B|\frac{dQ}{dP}-1\B|\right] +\E_P\left[\1_{ A^c}\B|\frac{dQ}{dP}-1\B|\right]\\
    &=  \E_P\left[ \1_A\B|\frac{dQ}{dP}-1\B|\right]+\E_{P}\left[\1_{A^c} \B|\frac{Q(A^c)}{P(A^c)} \cdot \frac{dQ|_{A^c}}{dP|_{A^c}}-1\B|\right]\\
   & \le   Q(A)+P(A) +\E_{P}\left[\1_{A^c} \frac{Q(A^c)}{P(A^c)}\cdot \bigg|\frac{dQ|_{A^c}}{dP|_{A^c}}-1\bigg| + \1_{A^c}\bigg|\frac{Q(A^c)}{P(A^c)}-1\bigg| \right] \\
    &\le    Q(A)+P(A)+ Q(A^c) \cdot \E_{P|_{A^c}} \bigg|\frac{dQ|_{A^c}}{dP|_{A^c}}-1\bigg| + |P(A)-Q(A)|\\
    &\le  2(P(A)+Q(A))+Q(A^c) \cdot\TV(P|_{ A^c},Q|_{ A^c})\qedhere
\end{align*}
\end{proof}

\begin{lemma}[Bounding TV via Concentration of Likelihood Ratio]\label{lem:concentration-TV}
Let $P$ and $Q$ be two distributions over the same space $\Omega$, with $Q$ absolutely continuous with respect to $P$. Let $\frac{dQ}{dP}$ be the likelihood ratio (or Radon-Nikodym derivative) between the two distributions. If $\b|\frac{dQ}{dP}(X)-1\b|\le \epsilon$ with probability at least $1-\delta$ for $X\sim P$, then 
\[\TV(P,Q)\le 2\epsilon+2\delta\,.\]
\end{lemma}
\begin{proof}
Let $A$ be the event that $\b|\frac{dQ}{dP}(X)-1\b|>\epsilon$, so $P(A)\le\delta$ from the assumption. Since $\frac{dQ}{dP}(X)\ge 1-\epsilon$ on $A^c$, $Q(A^c)\ge (1-\epsilon)P(A^c)\ge 1-\epsilon-\delta$. So $Q(A)\le \epsilon+\delta$.  We have 
\begin{align*}
     \TV(P,Q)
    = \E \B|\frac{dQ}{dP}(X)-1\B|
    &= \E \left[\1_{A}\B|\frac{dQ}{dP}(X)-1\B|\right] +\E \left[\1_ {A^c}\B|\frac{dQ}{dP}(X)-1\B|\right]\\
    &\le  P(A)+Q(A)+\epsilon\le 2\epsilon+2\delta\,.\qedhere
\end{align*}
\end{proof}

\begin{lemma}[Many Steps of Markov Kernel in Terms of One Step]\label{lem:kernel-tv}
Let $X$ be a random variable over space $\Omega$, $K$ be a Markov kernel from $\Omega$ to $\Omega$. Then for any  $m\in \mathbb{Z}^+$,
\[\TV(X,K^m(X))\le m\cdot \TV(X,K(X))\,.\]
\end{lemma}
\begin{proof}
We do an induction on $m$. The case of $m=1$ is trivial. For $m>1$, by triangle inequality,
\begin{align*}
    \TV(X,K^m(X)) &\le \TV(X,K^{m-1}(X))+\TV(K^{m-1}(X),K^m(X))\\
    &\le (m-1)\TV(X,K(X)) + \TV(K^{m-1}(X),K^m(X))\,.
\end{align*}
From data processing inequality, the second term can also be bounded by \[\TV(K^{m-1}(X),K^m(X))\le \TV(X,K(X))\,.\qedhere\]
\end{proof}

\begin{lemma}\label{lem:distance-expectation}
Let $X$ be a real-valued random variable satisfying $0\le X\le M$ almost surely. If there exists $a\in \mathbb{R}$ that $|X-a|\le \sigma$ with probability at least $1-\delta$, then 
\[|X-\E X|\le 2\sigma+\delta M\]
with probability at least $1-\delta$.
\end{lemma}
\begin{proof}
We have 
\[a-\sigma \le \E X\le (a+\sigma)(1-\delta)+\delta M\le a+\sigma+\delta M\,.\]
Combining this with $|X-a|\le \sigma$, we get 
$-2\sigma-\delta M\le X-\E X\le 2\sigma$.
\end{proof}

\subsection{Chernoff Bound and Mcdiarmid's Inequality}
\begin{lemma}[Chernoff Bound for Independent Binary Variables]\label{lem:chernoff}
Let $X_1,\cdots, X_n$ be random binary variables. Let $X$ denote their sum and $\mu=\E[X]$. We have for any $\delta\ge 0$,
\[\Pr(X\ge (1+\delta)\mu)\le e^{-\delta^2\mu/(2+\delta)}\quadand\Pr(X\le (1-\delta) \mu)\le e^{-\delta^2\mu/2}\,.\]
\end{lemma}

\begin{lemma}[Mcdiarmid's Inequality \cite{mcdiarmid1998concentration}]\label{lem:mcdiarmid}
Consider a real-valued martingale $X_0,X_1,\dots X_n$ adapted to filtration $\mathcal{F}$. 
If $|X_k-X_{k-1}|\le R$  and $\var(X_k|\mathcal{F}_{k-1})\le \sigma_i$ almost surely for any $k$, then 
\[\Pr(X_n- \E X_n\ge t)\le \exp\B(\frac{-t^2/2}{\sum\sigma_i^2+Rt/3}\B)\,. \]
\end{lemma}

\begin{lemma}[Alternative form of McDiarmid's Inequality]\label{lem:mcdiarmid-extension}
Consider $f:\cX^n\rightarrow \mathbb{R}$.
Let $X=(X_1,\dots,X_n)$ with $X_i\sim P_i$ sampled independently, where $P_i$ has probability mass at least $p_i$ on one element.
If there exists $\mathcal{Y}$ such that for any $x,x'\in \mathcal{Y}$, 
\[|f(x)-f(x')|\le \sum_{i:x_i\not=x_i'} c_i\]
and $\Pr(X\in \mathcal{Y}) = 1-p$, then for any $t\ge p\sum_ic_i$,
\[\Pr(f(X)-\E f(X)>t)\le p+\exp\B( \frac{-(t-p\sum_ic_i)^2/2}{\sigma^2+\max_i c_it/3} \B)\,,\]
where $\sigma^2=\sum_ip_i(1-p_i)c_i^2$.
\end{lemma}
\begin{proof}
The proof is the same as \cite{combes2015extension} but instead of using the version of Mcdiarmid's inequality in \cite{combes2015extension}, we use the more general version of Mcdiarmid's inequality stated in Lemma~\ref{lem:mcdiarmid}. 
\end{proof}






\section{Forward Map Preserves Planted Signal}\label{sec:forward}

This section aims to show the first part of Theorem~\ref{thm:main} by applying Theorem~\ref{thm:general} on the forward process $\Af$ described in Defn.~\ref{def:forward}.

In this section only, we use $X\in\{0,1\}^{\binom{[n]}{3}}$ to denote the indicator vector of the set of triangles added by $\Af$ and use $Z\in\{0,1\}^{\binom{[n]}{2}}$ to denote the indicator for the set of edges in $\Af(G)$, i.e., the indicator of $G\cup E(X)$. From the definition of $\Af$, $X\sim \bern(p')^{\otimes \binom{n}{3}}$.



\begin{theorem}[$\Af$ Increases Triangle Density in Planted RGT]\label{thm:PDSforward} Let
    $p'_1<p'_2\le 1/n$, $k:=|S|=O(n^{1/4}\log^{-5/4} n)$,  and $0<p,q<1$ constant. 
If $\Af$ includes each triangle with probability $p' = \frac{p_2'-p_1'}{1-p_1'}$, then
    \[\TV\b(\Af(\rgt(n,p,p'_1,S,q)),\rgt(n,p,p'_2,S,q')\b)=o_n(1)\,,\]
    where $q' =q+(1-q)\left(1-(1-p')^{n-2}\right)$.
\end{theorem}
By choosing $p_1'=0$ we get the reduction from planted dense subgraph to planted RGT. 
\begin{corollary}
    For $p'\le 1/n$, $k:=|S|=O(n^{1/4}\log^{-5/4} n)$  and $p,q$ being constant,
    \[\TV\b(\Af(G(n,p,S,q)),\rgt(n,p,p'_2,S,q')\b)=o_n(1)\,,\]
    where $q' =q+(1-q)\left(1-(1-p')^{n-2}\right)$.
\end{corollary}

\begin{remark}
The statement of the corollary is actually true all the way up to $k=O(\sqrt{n})$. This can be proved by calculating the $\chi^2$-divergence induced by adding a single triangle to an \ER\ random graph, with or without resampling edges in $S$. (A similar calculation appears in \cite{brennan2019phase}.) We instead use our general result, Theorem~\ref{thm:general}, as this gives a unified view on how to deal with correlation in planted models.
\end{remark}


\begin{remark}\label{rem:forward-q}
In Section~\ref{sec:results}, Theorem~\ref{thm:main} is stated without explicitly giving the mapping function $f(q)$. This corollary shows that $f(q)=q+(1-q)\left(1-(1-p')^{n-2}\right)$.
Because we assume $p'\ll 1/n$, we have $f(q)=q+O(np')=q+o_n(1)$. The parameter mapping $g$ for the reverse map also satisfies $g(q)=q+o_n(1)$ (see Cor.~\ref{cor:true-reverse} in Section~\ref{sec:reverse-preserves}). So as claimed in Theorem~\ref{thm:main}, $f\circ g(q) =q+o_n(1)$.
\end{remark}

\begin{proof}[Proof of Theorem~\ref{thm:PDSforward}]
We proof follows by checking the conditions of Theorem~\ref{thm:general}, which we do in the next three subsections. Here the two pure-noise models are $\rgt(n,p,p'_1)$ and $\rgt(n,p,p'_2)$, and the transformation between them is $\Af$. 
We also need to describe the partially planted version of the model, which we do similarly to Section~\ref{sec:ideas-inter-planted-model}. Let $e_1,e_2,\cdots, e_{\binom{k}{2}}$ be an enumeration of edges with both endpoints in $S$. Let $\RGTi{i}(n,p,p_1',S,q)$ be a random graph generated by starting with $\rgt(n,p,p_1')$ and resampling edges $e_1,e_2,\cdots, e_i$ to be each included with probability $q$.


\subsection{Class of Graphs $\cG$}
Define the class of graphs $\cG$ to be the set of graphs that are $p^2/2$-uniformly 2-star dense. 

\begin{lemma}\label{lem:probability-of-good}
Assume that $k=o(\sqrt{n})$. Both $G(n,p)$ and all $\rgt_i(n,p,p')$ for $1\leq i\leq {k\choose 2}$ are $p^2/2$-uniformly 2-star dense with probability at least $1-e^{-\Omega(n)}$.
\end{lemma}

\begin{proof}
Let $i,j$ be an arbitrary pair of nodes. For any $k\not\in \{i,j\}$, the probability that both $(i,k)$ and $(j,k)$ are in the graph is $p^2$. 
Recalling the wedge vertices $
\wed[G](i,j) = \{w\in [n]| (w,i),(w,j)\in G\}
$ from Section~\ref{sec:graph}, 
 $\E|\wed[G](i,j)| = p^2(n-2)$. Since the event that both $(i,k)$ and $(j,k)$ are in the graph is  independent for each $k$, we can use
 a Chernoff bound (Lemma~\ref{lem:chernoff}) to get
 \[\Pr\b(|\wed[G](i,j)|<2p^2(n-2)/3\b)\le e^{-p^2(n-2)/18}\,.\]
By the union bound over all pairs, we get that the probability that there is exists $i,j$ with $|\wed[G](i,j)|<2p^2(n-2)/3$ is at most 
$\binom{n}{2}e^{-p^2(n-2)/18}=e^{-\Omega(n)}$.
The same holds for $\rgt(n,p,p')$ because the property of being $c$-uniformly 2-star dense is preserved under addition of edges. 
Finally, note that changing $o(n)$ edges preserves the property, so the conclusion holds for $\rgt_i(n,p,p')$. 
\end{proof}

\subsection{Checking Conditions in Theorem~\ref{thm:general}}
\ref{cond-gen1} of Theorem~\ref{thm:general} is immediate, as $\Af(G-e)=G-e+E(X)$ and $\Af(G+e)=G+e+E(X)$ are the same on $\binom{[n]}{2}-e$. Condition \ref{cond-gen3}
 is also easy to check and holds with $\epsilon=0$: $\Pr(e\in \Af(G-e))=\Pr(e\in E(X)) = 1-(1-p')^{n-2}=:p_-^e$. Since $\Af$ only adds edges, $\Pr(e\in \Af(G+e))=1=:p_+^e$.
 
Verifying condition \ref{cond-gen2} is more technical. For $G+e$, $e\in \Af(G+e)$ always holds so $\Af(G+e) = \Af(G+e)|_{\sim e}\times \bern(1).$ For $G-e$, we will show condition \ref{cond-gen2}
assuming the following lemma. Let $Z\in\{0,1\}^{\binom{[n]}{2}}$ be the indicator vector for the set of edges in $\Af(G-e)$.

\begin{lemma}\label{lem:forward-influence-on-other-edges}
Fix $0<c<1$. For a $c$-uniformly 2-star-dense graph $G$ with $e\notin G$,
\[
\TV\b(\law(Z_{\sim e}), \law(Z_{\sim e}|Z_e=0) \b) = O(n^{-1/2} \log^{5/2} n)
\,.
\]
\end{lemma}

The proof is deferred until the next subsection, and we proceed with checking condition \ref{cond-gen2}.
Note that 
\[\law(Z_{\sim e}) = p_-^e \law(Z_{\sim e}|Z_e=1) + (1-p_-^e)\law(Z_{\sim e}|Z_e=0)\,.\]
Manipulating the relevant p.m.f.s, it follows that
$$
\TV\b(\law(Z_{\sim e}), \law(Z_{\sim e}|Z_e=1) \b) 
\le \frac1{p_-^e}\cdot \TV\b(\law(Z_{\sim e}), \law(Z_{\sim e}|Z_e=0) \b) = 
\tilde O\B(\frac1{p_-^e\sqrt{n}} \B)\,.$$ 
Combining with Lemma~\ref{lem:forward-influence-on-other-edges},
\begin{align*}
    \law(Z) &= p_-^e \law(Z|Z_e=1) + (1-p_-^e)\law(Z|Z_e=0)\\
    &\approx_{O(n^{-1/2} \log^{5/2} n)} p_-^e \law(Z_{\sim e})\times \bern(1) + p_-^e \law(Z_{\sim e})\times \bern(0)\\
    &= \law(Z_{\sim e})\times \bern(p_-^e)\,.
\end{align*}
This completes the proof for condition \ref{cond-gen2}.

Therefore, Theorem~\ref{thm:general} applies with $\epsilon = O(n^{-1/2} \log^{5/2} n)$ and $\delta=O(e^{-\Omega(n)})$. When $k= O(n^{1/4}\log^{-5/4} n)$, there are at most $ O(n^{1/2} \log^{-5/2} n)$ planted edges and the theorem follows from Theorem~\ref{thm:general} with $q' = qp_+^e+(1-q)p_-^e=q+(1-q)(1-(1-p')^{n-2})$.
\end{proof}

\subsection{Perturbation Insensitivity of $\Af$ and Proof of Lemma~\ref{lem:forward-influence-on-other-edges}}
It only remains to prove the lemma. 
The idea of the proof is similar to Section~\ref{sec:projection}. In the forward mapping, we do not need to worry about the influence of non-incident triangles in Section~\ref{sec:non-incident-triangles}, as all the triangles are independent. This section thus serves as a warm up for the proof of Lemma~\ref{lem:edge_on_other_edge}.

In this proof, we will apply the projection idea as in Section~\ref{sec:projection} and design an auxiliary triangle distribution $\aux$ that will allow us to compare
 $Z_{\sim e}$ and $\law(Z_{\sim e}|Z_e=0)$ indirectly, by comparing certain triangle distributions. Finally, these triangle distributions will be bounded by showing that their likelihood ratio concentrates.

\begin{proof}[Proof of Lemma~\ref{lem:forward-influence-on-other-edges}]

Let $T_e$ be the set of triangles in $G$ that contain $e$. Define $Z = \Af(G) = G + E(X)$ to be the output graph, so that
$\law(Z_{\sim e}|Z_e=0)$ is the conditional distribution of the output graph given that all triangles in $T_e$ are not selected in $X$. 
Let $X'$ be the same as $X$ on $\binom{[n]}{3}\backslash T_e$ and 0 on $T_e$.
We have 
\begin{equation}\label{eq:edge-of-Y}
   E(Z_{\sim e})\disteq G+E(X)-\{e\}\quadand E(Z_{\sim e}|Z_e=0)\disteq G+E(X')-\{e\}
\,. 
\end{equation}

\paragraph{Design of Auxiliary Triangle Distribution $\aux$}

To prove the lemma, it would be convenient if $X$ and $X'$ were close in distribution, but this is unfortunately not true as $X$ includes a triangle in $T_e$ with constant probability.
Luckily, it is the edge set $E(X)-e$ and $E(X')$ that we care about. Our approach is to construct another set of triangles $\aux$ that includes the extra edges in $E(X)-e$ that are missing in $E(X')$, \emph{by adding in triangles that could plausibly be in $X'$}.  The relation between $X,X'$ and $\aux$ is shown in Figure~\ref{fig:projection-forward}. In order to define $\aux$ we first require a lemma.
\begin{figure}
    \centering
    \tikzset{every picture/.style={line width=0.75pt}} 

\begin{tikzpicture}[x=0.75pt,y=0.75pt,yscale=-1,xscale=1]

\draw [color={rgb, 255:red, 208; green, 2; blue, 27 }  ,draw opacity=1 ][line width=1.5]    (231.71,36.87) -- (231.71,87.08) ;
\draw [shift={(231.71,90.08)}, rotate = 270] [color={rgb, 255:red, 208; green, 2; blue, 27 }  ,draw opacity=1 ][line width=1.5]    (11.37,-3.42) .. controls (7.23,-1.45) and (3.44,-0.31) .. (0,0) .. controls (3.44,0.31) and (7.23,1.45) .. (11.37,3.42)   ;
\draw [line width=1.5]    (109.95,35) -- (109.02,156) ;
\draw [shift={(109,160)}, rotate = 270.43] [color={rgb, 255:red, 0; green, 0; blue, 0 }  ][line width=1.5]    (11.37,-3.42) .. controls (7.23,-1.45) and (3.44,-0.31) .. (0,0) .. controls (3.44,0.31) and (7.23,1.45) .. (11.37,3.42)   ;
\draw [line width=1.5]    (350.53,35.87) -- (350.53,86.08) ;
\draw [shift={(350.53,89.08)}, rotate = 270] [color={rgb, 255:red, 0; green, 0; blue, 0 }  ][line width=1.5]    (11.37,-3.42) .. controls (7.23,-1.45) and (3.44,-0.31) .. (0,0) .. controls (3.44,0.31) and (7.23,1.45) .. (11.37,3.42)   ;
\draw [line width=1.5]    (353.89,121.68) -- (353.99,156) ;
\draw [shift={(354,159)}, rotate = 269.82] [color={rgb, 255:red, 0; green, 0; blue, 0 }  ][line width=1.5]    (11.37,-3.42) .. controls (7.23,-1.45) and (3.44,-0.31) .. (0,0) .. controls (3.44,0.31) and (7.23,1.45) .. (11.37,3.42)   ;
\draw [color={rgb, 255:red, 245; green, 166; blue, 35 }  ,draw opacity=1 ]   (122,38) .. controls (147.35,57.6) and (177.01,61.4) .. (204,39) ;
\draw [color={rgb, 255:red, 245; green, 166; blue, 35 }  ,draw opacity=1 ]   (108,189) .. controls (132.54,221.67) and (326.83,224.47) .. (353,189) ;

\draw (102,10.) node [anchor=north west][inner sep=0.75pt]   [align=left] {$\displaystyle X'$};
\draw (203,10) node [anchor=north west][inner sep=0.75pt]  [color={rgb, 255:red, 208; green, 2; blue, 27 }  ,opacity=1 ] [align=left] {$\displaystyle X'\lor X^{aux}$};
\draw (342.74,12) node [anchor=north west][inner sep=0.75pt]   [align=left] {$\displaystyle X$};
\draw (128.62,92.01) node [anchor=north west][inner sep=0.75pt]  [color={rgb, 255:red, 74; green, 144; blue, 226 }  ,opacity=1 ] [align=left] {$\displaystyle G+E\left( X'\lor X^{aux}\right) -e$};
\draw (302.15,94.88) node [anchor=north west][inner sep=0.75pt]  [color={rgb, 255:red, 74; green, 144; blue, 226 }  ,opacity=1 ] [align=left] {$\displaystyle G+E( X) -e$};
\draw (308.97,163.66) node [anchor=north west][inner sep=0.75pt]   [align=left] {$ $$\displaystyle Z_{\sim e} |Z_{e} =0$};
\draw (92.66,163.39) node [anchor=north west][inner sep=0.75pt]   [align=left] {$ $$\displaystyle Z_{\sim e}$};
\draw (154.44,60.61) node [anchor=north west][inner sep=0.75pt]  [color={rgb, 255:red, 245; green, 166; blue, 35 }  ,opacity=1 ] [align=left] {$\displaystyle ( 1)$};
\draw (220.18,218.03) node [anchor=north west][inner sep=0.75pt]  [color={rgb, 255:red, 245; green, 166; blue, 35 }  ,opacity=1 ] [align=left] {$\displaystyle ( 2)$};
\draw (282,99.01) node [anchor=north west][inner sep=0.75pt]  [color={rgb, 255:red, 74; green, 144; blue, 226 }  ,opacity=1 ] [align=left] {$\displaystyle =$};

\end{tikzpicture}
    \caption{Relation between $X,X', \aux$ and the corresponding edge sets used in proving Lemma~\ref{lem:forward-influence-on-other-edges}. Closeness of (1) implies closeness of (2). }
    \label{fig:projection-forward}
\end{figure}
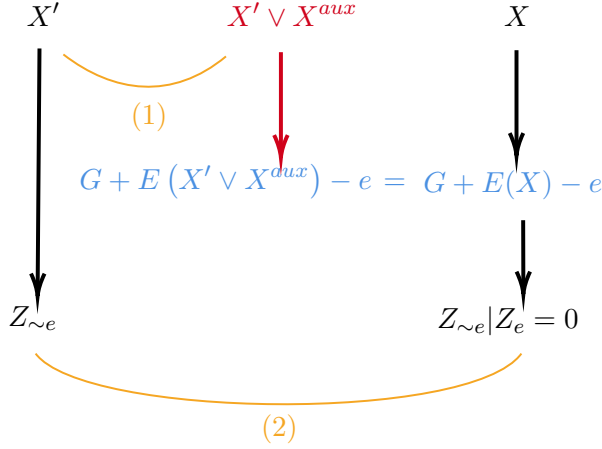

 Recall that $G$ is $c$-uniformly $2$-star-dense and $e=(u,v)\notin G$.  We can choose for each $w\notin\{u,v\}$ a subset $\bwed[G](u,w)$ of $\wed[G](u,w)$ such that the collection satisfies two conditions. 

\begin{restatable}{lemma}{regularsubgraph}\label{lem:regular-subgraph}
Let $G$ be a $c$-uniformly two-star dense graph with $n$ vertices, $u,v$ are vertices in $G$ and $(u,v)\notin G$. There exists $n_0>0$ such that for any $n>n_0$, we can choose for every $w\notin \{u,v\}$
a set $\bwed[G](u,w)\subset \wed[G](u,w)$ 
such that 
\begin{itemize}
    \item For every $ w \not\in \{u,v\}$, $|\bwed[G](u,w)|=cn/3$, and
    \item For every pair $ w_1,w_2\not\in \{u,v\}$, either $w_1\notin \bwed[G](u,w_2)$ or $w_2\notin \bwed[G](u,w_1)$.
\end{itemize}
\end{restatable}
The proof of the lemma can be found in Appendix~\ref{sec:technical-lemmas}.
We apply the lemma with $v$ in place of $u$ to also find sets $\bwed[G](v,w)$ satisfying the conditions. 


We can now define $\aux$:
\begin{enumerate}
    \item Start with $\aux=\vec{0}$.
    \item For each $ w\notin \{u,v\}$, with probability $p'$ do the following: 
    \begin{itemize}
        \item[\tiny $\bullet$] Sample $w_1\sim \unif(\bwed[G](u,w))$ and $w_2\sim \unif(\bwed[G](v,w))$. 
        \item[\tiny $\bullet$] Set $\aux_{(u,w,w_1)} = \aux_{(v,w,w_2)}=1$.
    \end{itemize}
\end{enumerate}
\begin{figure}
    \centering
    \tikzset{every picture/.style={line width=0.75pt}} 

\begin{tikzpicture}[x=0.75pt,y=0.75pt,yscale=-1,xscale=1]

\draw [color={rgb, 255:red, 144; green, 19; blue, 254 }  ,draw opacity=1 ]   (76,177) -- (131,72) -- (173,178) -- (97.98,177.23) -- cycle ;
\draw [color={rgb, 255:red, 208; green, 2; blue, 27 }  ,draw opacity=1 ]   (70,167) -- (51,72) -- (120,69) -- cycle ;
\draw [color={rgb, 255:red, 208; green, 2; blue, 27 }  ,draw opacity=1 ]   (179,168) -- (142,69) -- (214,73) -- cycle ;
\draw [color={rgb, 255:red, 126; green, 211; blue, 33 }  ,draw opacity=1 ]   (124,62) -- (43,66) ;
\draw [color={rgb, 255:red, 126; green, 211; blue, 33 }  ,draw opacity=1 ]   (43,66) -- (64,173) -- cycle ;
\draw [color={rgb, 255:red, 126; green, 211; blue, 33 }  ,draw opacity=1 ]   (223,67) -- (140,63) ;
\draw [line width=1.5]    (76,177) -- (173,178) ;
\draw [color={rgb, 255:red, 126; green, 211; blue, 33 }  ,draw opacity=1 ]   (223,67) -- (182,176) ;

\draw (60,41) node [anchor=north west][inner sep=0.75pt]  [color={rgb, 255:red, 126; green, 211; blue, 33 }  ,opacity=1 ] [align=left] {$\displaystyle G+E( X')$};
\draw (111,133) node [anchor=north west][inner sep=0.75pt]  [color={rgb, 255:red, 144; green, 19; blue, 254 }  ,opacity=1 ] [align=left] {$\displaystyle X$};
\draw (63,88) node [anchor=north west][inner sep=0.75pt]  [color={rgb, 255:red, 208; green, 2; blue, 27 }  ,opacity=1 ] [align=left] {$\displaystyle X^{aux}$};
\draw (118,185) node [anchor=north west][inner sep=0.75pt]   [align=left] {$\displaystyle e$};

\end{tikzpicture}
    \caption{Illustration of $\aux$. It includes edges that are in $G+E(X)-\{e\}$ but not in $G+E(X')$.}
    \label{fig:aux-forward}
\end{figure}
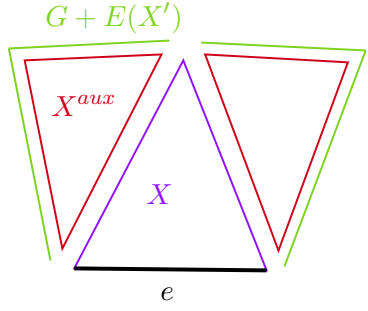
Since $w_1\in \wed[G](u,w)$ and $w_2\in \wed[G](v,w)$, the only possible new edges introduced by $\aux$ are $(u,w)$ and $(v,w)$, which has the same effect on $\binom{[n]}{2}\backslash\{e\}$ as adding triangle $(u,v,w)$ to the graph. Therefore, we have
\begin{equation}
\label{eq:x-prime-with-aux}
    E(Z_{\sim e})=G+E(X)-\{e\} \disteq G+E(X'\vee \aux)-\{e\} 
\,.
\end{equation}
An illustration of $\aux$ is given in Figure~\ref{fig:aux-forward}.
 Comparing \eqref{eq:x-prime-with-aux} with \eqref{eq:edge-of-Y}, 
we have by data processing inequality that
\begin{equation}\label{eq:Z-projection}
    \TV\b(\law(Z_{\sim e}), \law(Z_{\sim e}|Z_e=0) \b)
    \leq \TV\b(\law(X') , \law(X'\vee \aux)\b)\,,
\end{equation}
and we now turn to bounding the latter quantity by adding triangles in $\aux$ to $X'$ one at a time. 

\paragraph{Adding one triangle at a time}

Let $K$ be the following channel acting on any input $X\in \{0,1\}^{\binom{[n]}{3}}$. Sample $w\sim \unif\b([n]\backslash\{u,v\}\b)$, $w_1\sim \unif\b(\bwed[G](u,w)\b)$ and $w_2\sim \unif\b(\bwed[G](v,w)\b)$. Set $\aux_{(v,w,w_1)} = \aux_{(v,w,w_2)}=1$. It turns out that $X'\vee\aux$ is closely approximated by repeating $K$ on $X'$ a $ \bino(n-2,p')$ number of times: we show in Lemma~\ref{lem:binom-2-sample-with-replacement} in the appendix that
\begin{equation}\label{eq:aux-kernel}
   X'\vee \aux \approx_{O(p'^2n)} K^{m}(X')\,,\quad\text{ for } m \sim \bino(n-2,p')\,. 
\end{equation}
The only difference between applying $K^m$ and adding triangles in $\aux$ is that $K^m$ may sample the same triangle twice, which gives rise to the $O(p'^2n)$ error.

Let us use $Q$ and $P$ to denote the p.m.f. of $K(X')$ and $X'$, respectively. Let $A(x)$ denote the set of possible triples $w,w_1,w_2$ that could have been chosen by $K$ if $x=K(X')$ for some $X'$, i.e., 
\[A(x) = \{(w,w_1,w_2):w\in [n]\backslash \{u,v\}, w_1\in \bwed[G](u,w), w_2\in \bwed[G](v,w), x_{(u,w,w_1)} = x_{(v,w,w_2)}=1\}\,.\]
Note that
\begin{align}\label{eq:expectA}
    \E |A(X')|
    = \sum_{
    \substack{
    w\in [n]\backslash\{u,v\},\\
    w_1\in \bwed[G](u,w),\\ w_2\in \bwed[G](v,w) }} \Pr(X_{(u,w,w_1)}=X_{(v,w,w_2)}=1)
    =  p'^2(n-2)(c/3)^2n^2
    \,.
\end{align}
Using this, we can write the likelihood ratio between $Q$ and $P$ as 
\begin{align}\label{eq:LQP}
\frac{dQ}{dP}(X') = \frac{\sum\limits_{(w,w_1,w_2)\in A(X')} \frac{p'^{|X'|-2}(1-p')^{\binom{n}{3}-|X'|}}{(n-2)(c/3)^2n^2}  }{p'^{|X'|}(1-p')^{\binom{n}{3}-|X'|}} = \frac{|A(X')|}{ p'^2(n-2)(c/3)^2n^2} = \frac{|A(X')|}{\E |A(X')|} \,.
\end{align}

\paragraph{Concentration of the Likelihood Ratio}
The idea is to use Lemma~\ref{lem:concentration-TV} to bound the total variation between $P$ and $Q$ by showing that $|A(X')|$ (and hence $dQ/dP$) concentrates.

Let $\Twed(u,w) = \{(u,w,w_1):w_1\in \bwed[G](u,w)\}$ and $\Twed(v,w) = \{(u,w,w_2):w_2\in \bwed[G](v,w)\}$.
Note that 
$$|A(X')| = \sum_{w\in [n]\backslash\{u,v\}}|X'_{\Twed(u,w)}|\cdot |X'_{\Twed(v,w)}|:=
\sum_{w\in [n]\backslash\{u,v\}} H_w
\,.$$ 
By construction of $\bwed[G](u,w)$ and $\bwed[G](v,w)$ in Lemma~\ref{lem:regular-subgraph}, $\Twed(u,w)$ and $\Twed(v,w)$ are disjoint sets of triples, so 
the $H_w=|X'_{\Twed(u,w)}|\cdot |X'_{\Twed(v,w)}|$ variables are independent.
Let $H$ be the vector of all $H_w$. 

To show concentration of $|A(X')|$ one natural approach is to apply McDiarmid's inequality to $|A(X')|= f(H)$. This does not work, because $X'_{\Twed(u,w)}$ and $X'_{\Twed(v,w)}$ can have size $cn/3$, which leads to a $\Theta(n^2)$ Lipschitz constant. The key is to use the fact that these sets are \emph{typically} much smaller, since each coordinate is 1 only with probability $p'$. 
In fact, by a Chernoff bound, $X'_{\Twed(u,w)}$ and $X'_{\Twed(v,w)}$ are with high probability of size on the order of $\log n$. Each coordinate of $X'_{\Twed(u,w)}$ is independent $\bern(p')$, and there are $cn/3$ coordinates, so we have 
\[
\Pr(|X'_{\Twed(u,w)}|>10\log n)\le \exp\B(- \frac{\delta^2cnp'}{3(2+\delta)} \B)\,,
\]
where $\delta = \frac{30\log n}{cnp'}-1$. Since $\delta>2$, we have $\delta^2/(2+\delta)\ge \delta/2$ and therefore,
\[\exp\B(- \frac{\delta^2cnp'}{3(2+\delta)}\B)\le e^{-\delta cnp'/6} = e^{-5\log n+cnp'/6}\le n^{-4}\,.\]
By the union bound, $|X'_{\Twed(u,w)}|\le 10\log n$  and $|X'_{\Twed(v,w)}|\le 10\log n$ for all pairs $u,w$ and $v,w$ with probability at least $1-2/n^2$. Hence, with high probability, $H_w\le 100\log^2 n$ for all $w$.

Now we can use the modified McDiarmid's inequality (Lemma~\ref{lem:mcdiarmid-extension}). 
Let $\cY = \{h:h_w\le 100\log^2 n \text{ for all }w\}$. By the calculation above, $\Pr(H\in \cY)\ge 1-2/n^2$. On the other hand, for any $H,H'\in\cY$,  $|A(X')|= \sum_w H_w:=f(H)$ satisfies $|f(H) - f(H')|\leq \sum_{w:H_w\not=H'_w} 100\log ^2 n$. 
By Markov's inequality, $\Pr(|X'_{\Twed(u,w)}|>0)\le \E |X'_{\Twed(u,w)}| = cnp'/3$, so $\Pr(H_w\not= 0)\le (cnp'/3)^2$.
Applying Lemma~\ref{lem:mcdiarmid-extension} now yields
\[
\Pr\b(\b|\,|A(X')|-\E  | A(X')|\, \b|\ge t\b)\le 4/n^2+2 \exp\B(\frac{-(t-200(n-2)\log^2 n/n^2)^2 /2}{ 100^2(n-2)(cnp'/3)^2\log ^4 n+ 100t\log^2 n/3}\B)
\,.
\]

Choosing $t = O(n^{3/2}p'\log^{5/2} n)$ makes this probability at most $1/n$. Combining this with $\E |A(X')|$ in \eqref{eq:expectA}, we get that with probability at least $1-1/n$ the likelihood ratio $dQ/dP$ in \eqref{eq:LQP} is
\begin{equation}\label{eq:QPconc}
   \frac{dQ}{dP}(X') =\frac{|A(X')|}{\E |A(X')|} = 1+O\B(\frac{n^{3/2}p'\log^{5/2} n}{n^3p'^2}\B)
 = 1+O\B(\frac{\log^{5/2} n}{ n^{3/2}p'}\B)\,.
\end{equation}

\paragraph{Putting the Pieces Together}
By Lemma~\ref{lem:concentration-TV}, \eqref{eq:QPconc} implies that $$\TV\b(K(X'),X'\b)=\TV(Q,P) =O\B( \frac 1{n}\B)+O\B(\frac{\log^{5/2} n}{ n^{3/2}p'}\B) 
\,.$$ 
By Lemma~\ref{lem:kernel-tv} and convexity of total variation distance, $$\TV(K^m(X'),X') = O\B(\frac{\E[m]\log^{5/2} n}{n^{3/2}p'} \B)=  O(n^{-1/2} \log^{5/2} n)\,,$$ where in the last step we used that $m \sim \bino(n-2,p')$. 
This further implies by \eqref{eq:aux-kernel} that $\TV(X'\vee \aux, X')=O(n^{-1/2} \log^{5/2} n)$ and hence by \eqref{eq:Z-projection}, $\TV(\law(Z_{\sim e}), \law(Z_{\sim e}|Z_e=0))= O(n^{-1/2} \log^{5/2} n) $.
\end{proof}

\section{Properties of Triangle Distribution $\mu_G$}

Recall the triangle distribution $\mu_G$ from Defn.~\ref{def:triangle_dist}. To prepare for proving Theorem~\ref{thm:PDSreverse} in the next section, we introduce some notation associated to $\mu_G$ and show some basic properties.

In the remainder of the paper, for any events $A,B$ over $X$, we will use the notation
\[\mu_G(A):=\Pr(X\in A)\quad\text{and}\quad\mu_G(A|B) := \Pr(X\in A|B)\]
to distinguish distributions defined on different graphs.

\begin{remark}[Graphical Model of $\mu_G$]
The graphical model of $\mu_G$ is the graph over the set of all triangles in $\binom{[n]}{3}$ with two triangles being adjacent if and only if they share an edge. We assign variables also for triangles that are not in $G$ for convenience in the proofs.
\end{remark}

\begin{remark}[Conditioning Property of $\mu_G$]
For any two graphs $G$ and $G'$ such that $E(G')\supset E(G)$,  $$\mu_G=\mu_{G'}(\dotspace| Y_{E(G')-E(G)}=0)$$ and hence also $\cL_G(Y) = \cL_{G'}(Y|Y_{E(G')-E(G)}=0)$.
\end{remark}

Below are a few properties of $\mu_G$ that will be used later in the proof. 
We next define two properties of a distribution and show that $\mu_G$ satisfies these properties.

\begin{definition}\label{def:marg-small}
A distribution $P$ over $\{0,1\}^N$ is said to be \emph{$q$-marginally small (or large)} if the marginal at any coordinate $i$ of $P$ is bounded by $q$ under arbitrary conditioning $x_{\sim i}$:
    \begin{equation*}
        P(x_i=1|x_{\sim i})\le q\quad(\text{or }P(x_i=1|x_{\sim i})\ge q)\,.
    \end{equation*}
\end{definition}

\begin{lemma}[$\mu_G$ has bounded marginals]\label{lem:mug_marginally_small}
$\mu_G$ is $O(p')$-marginally small and $\Omega(p')$-marginally large.
\end{lemma}
\begin{proof}
Let $t$ be a triangle in $T_G$ and let $X_{\sim t}$ be an arbitrary $\{0,1\}$ configuration on the other triangles. Let $X^{+t}$ be defined by setting $(X^{+t})_t=1$, $(X^{+t})_{\sim t}=X_{\sim t}$, and let $X^{-t}$ be defined analogously with $(X^{-t})_t = 0$. Then
\[\mu_G(X_t=1|X_{\sim t}) = \frac{\b(\frac{p'}{1-p'}\b)^{|X^{+t}|} p^{-\e(X^{+t})} }{\b(\frac{p'}{1-p'}\b)^{|X^{-t}|} p^{-\e(X^{-t})} + \b(\frac{p'}{1-p'}\b)^{|X^{+t}|} p^{-\e(X^{+t})}} \,. \]
Note that $|X^{+t}|=|X^{-t}|+1$ and $|\e(X^{+t})-\e(X^{-t})|\le 3$. It follows that
\[ p' \le  \mu_G(X_t=1|X_{\sim t})\le\frac{p'p^{-3}}{1-p'+p'p^{-3}}  \,.\]
Here we used that $p$ is a constant with respect to $n$.
\end{proof}

We restate the definition of uniformly 2-star dense graphs.
\begin{definition}[Uniformly 2-star Dense]\label{def:uniformly-2-star-dense}
A graph $G$ is $c$-uniformly 2-star dense if for any pair $\{i,j\}\in \binom{[n]}{2}$, the wedge set $\wed[G](i,j)=\{w\in [n]| (w,u),(w,v)\in G\}$ has size $|\wed[G](i,j)|\ge c(n-2)$.
\end{definition}

\begin{corollary}\label{cor:muge-marginal}
For any $c$-uniformly 2-star dense graph $G$ and for any $e\in G$, $\mug{G}(Y_e=1)=\Theta(np')$. 
\end{corollary}
\begin{proof}
For any edge $e\in G$, let $T_e=\{t_1,t_2,\cdots,t_k\}$ be the set of triangles in $G$ that contain $e$. Since $G$ is $c$-uniformly 2-star dense, $k=\Theta(n)$.
First,
$\mug{G}(Y_e=1)=\mu_G(X_{T_e}\not= \vec{0})$, and
\[\mu_G(X_{T_e}\not= \vec{0}) = \mu_G(X_{t_1}=1)+\mu_G(X_{t_2}=1|X_{t_1}=0)+\cdots+\mu_G(X_{t_k}=1|X_{t_1}=\cdots =X_{t_{k-1}}=0)\,.\]
From Lemma~\ref{lem:mug_marginally_small}, each term is $\Theta(p')$. So the probability is $\Theta(kp')=\Theta(np')$.
\end{proof}

As a consequence of the marginal smallness of $\mu_G$, it has concentration of measure and also the edge distribution $\law_G(Y)$ has low influence. The proofs are deferred to Sections~\ref{sec:small-marginal} and~\ref{sec:small-marginal-concentration}. 

\begin{definition}[Marginal Influence]\label{def:marginal-influence}
For a binary distribution over $\{0,1\}^\cX$, the \emph{marginal influence} of $S\subset \cX$ on $S'\subset \cX$ is defined by
$$\IM{S}{S'}=\sup_{x_S^a,x_S^b\in \{0,1\}^S} \TV(P_{x_{S'}|x_S^a}, P_{x_{S'}|x_S^b} )\,.$$
For $x,y\in \cX$, use $\IM{x}{y}$ to denote  $\IM{\{x\}}{\{y\}}$.
\end{definition}

\begin{restatable}{lemma}{lowInfLemma}\label{lem:edge-influence-adjacent}
Suppose $G$ is a $c$-uniformly 2-star dense graph for a constant $c>0$.
Consider distribution $\law_G(Y)$ as in Definition~\ref{def:triangle_dist}. For $A\subset E(G)$ with $|A|\le cn/4$ and any edge $e\notin A$, 
\[\IM A{e}=O(|A|/(n^3p'^2))=\tilde{O}(|A|/n)\,.\]
For any pair of edges $e,  e'\in E(G)$ that do not share common nodes,
 \[\IM {e'}{e}=O(1/(n^3p'))=\tilde{O}(1/n^2)\,.\]
\end{restatable}

\begin{restatable}{corollary}{influencepe}\label{cor:influence_on_pe}
For any $c$-uniformly 2-star dense graph $G$ and any $e\not=e'$, 
\[\left|
\mug{G}(Y_e=1)-\mug{G+e'}(Y_e=1)\right| = \tilde{O}(1/n)\,.\]
\end{restatable}


\begin{restatable}{corollary}{concentrationmuG}\label{cor:concentration_muG}
Let $f:\{0,1\}^{\binom{[n]}{3}}\rightarrow \mathbb{R}$ be an $L$-Lipschitz function.
Let $T$ be a subset of $\binom{[n]}{3}$ with size $\binom{n}{3}-m$ and $x_T\in \{0,1\}^{T}$ be a configuration on $T$. If $X\sim \mu_G(\ \cdot\ |X_T = x_T)$,  then
\[|f(X)-\E f(X)| = O(\sqrt{mp'L^2}\log m)\]
holds with high probability.
\end{restatable}

\section{Reverse Map Preserves Planted Signal}\label{sec:reverse-preserves}


In this section we prove our main result, restated here for convenience.


\begin{mythm}{\ref{thm:PDSreverse}}[Triangle Removal in Planted Case]
    For $k=o(n^{1/4}\log^{-17/4}n)$, $p'=\pval$, and $0<p<q<1$ being constant,  
    \[\TV\b(\Ab(\rgt(n,p,p',S,q)),G(n,p,S,q\cdot p_e)\b)=o_n(1)\,,\]
    where $p_e=\E_{G\sim \rgt(n,p,p')}\Pr_{\Ab}(e\in \Ab(G+e))$ for an arbitrary edge $e$.    
\end{mythm}

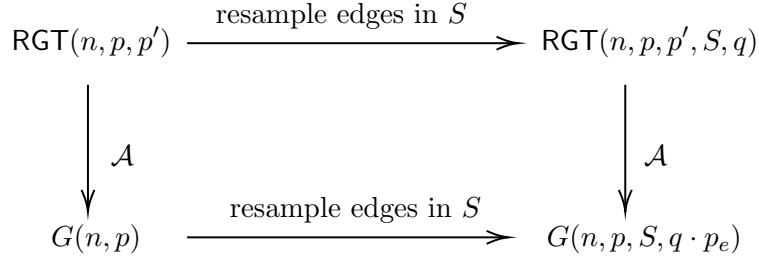
\begin{figure}
    \centering
\tikzset{every picture/.style={line width=0.75pt}} 

\begin{tikzpicture}[x=0.75pt,y=0.75pt,yscale=-1,xscale=1]

\draw    (160,81) -- (160,145) ;
\draw [shift={(160,145)}, rotate = 270] [color={rgb, 255:red, 0; green, 0; blue, 0 }  ][line width=0.75]    (10.93,-3.29) .. controls (6.95,-1.4) and (3.31,-0.3) .. (0,0) .. controls (3.31,0.3) and (6.95,1.4) .. (10.93,3.29)   ;
\draw    (431,81) -- (431,145) ;
\draw [shift={(431,145)}, rotate = 270] [color={rgb, 255:red, 0; green, 0; blue, 0 }  ][line width=0.75]    (10.93,-3.29) .. controls (6.95,-1.4) and (3.31,-0.3) .. (0,0) .. controls (3.31,0.3) and (6.95,1.4) .. (10.93,3.29)   ;
\draw    (210,60) -- (375,60) ;
\draw [shift={(377,60)}, rotate = 180] [color={rgb, 255:red, 0; green, 0; blue, 0 }  ][line width=0.75]    (10.93,-3.29) .. controls (6.95,-1.4) and (3.31,-0.3) .. (0,0) .. controls (3.31,0.3) and (6.95,1.4) .. (10.93,3.29)   ;
\draw    (210,158) -- (370,158) ;
\draw [shift={(372,158)}, rotate = 180] [color={rgb, 255:red, 0; green, 0; blue, 0 }  ][line width=0.75]    (10.93,-3.29) .. controls (6.95,-1.4) and (3.31,-0.3) .. (0,0) .. controls (3.31,0.3) and (6.95,1.4) .. (10.93,3.29)   ;

\draw (120,50) node [anchor=north west][inner sep=0.75pt]   [align=left] {$\displaystyle \rgt( n,p,p')$};
\draw (387,50) node [anchor=north west][inner sep=0.75pt]   [align=left] {$\displaystyle \rgt( n,p,p',S,q)$};
\draw (140,150) node [anchor=north west][inner sep=0.75pt]   [align=left] {$\displaystyle G(n,p)$};
\draw (390,150) node [anchor=north west][inner sep=0.75pt]   [align=left] {$\displaystyle G( n,p,S,q\cdot p_e)$};
\draw (170,110) node [anchor=north west][inner sep=0.75pt]   [align=left] {$\displaystyle \Ab$};
\draw (439,110) node [anchor=north west][inner sep=0.75pt]   [align=left] {$\displaystyle \Ab$};
\draw (223,38) node [anchor=north west][inner sep=0.75pt]   [align=left] {resample edges in $S$};
\draw (230,136) node [anchor=north west][inner sep=0.75pt]   [align=left] {resample edges in $S$};
\end{tikzpicture}\vspace{-1cm}
    \caption{Exchange Diagram. Theorem~\ref{thm:PDSreverse} states that $\cA(\rgt(n,p,p',S,q))$ is close to $G(n,p,S,q\cdot p_e)$. An equivalent way of looking at the result is that the operations of resampling edges and mapping $\cA$ (approximately) commute when applied to $\rgt(n,p,p')$. 
    }
    \label{fig:exchange_diagram}
\end{figure}


\begin{remark} 
In Theorem~\ref{thm:main}, we assumed $p'=\tilde O(1/n)$.
In order to simplify the proof, we only consider $p'=1/(n\log n)=\Tilde{\Theta}(1/n)$, which is the most challenging case. For smaller $p'$, we map from $\rgt(n,p,p',S,q)$ to $G(n,p,S,qp_e)$ by first adding triangles independently via $\Af$ to increase $p'$ up to $\Tilde{\Theta}(1/n)$
and then apply the reverse transition $\Ab$. This is the content of the following corollary.
\end{remark}

\begin{corollary}\label{cor:true-reverse}
For $k=o(n^{1/4}\log^{-17/4}n)$, $p' = O(1/(n\log n))$ and $0<p<q<1$ being constant, there exists an algorithm $\Ar$ (taking a graph and both $p$ and $p'$ as inputs) such that
\[\TV\left(\Ar(\rgt(n,p,p',S,q),G(n,p,S,q')\right)=o_n(1)\,,\]
where we let $\pprime = 1/(n\log n)$ and
\[q' = \B(q+(1-q)\B(1-\B(1-\frac{\pprime-p'}{1-p'}\B)^{n-2}\B)\B)\cdot \E_{G\sim \rgt(n,p,\pprime)}\Pr_{\Ab}\b(e\in \Ab(G+e)\b) = q+o(1)\,. \]
\end{corollary}
\begin{proof}
We can first apply $\Af$ as stated in Theorem~\ref{thm:PDSforward} with $p'_1=p'$ and $p'_2=\pprime$, and then apply $\Ab$ stated in Theorem~\ref{thm:PDSreverse}. The resulting edge density in the planted dense subgraph after $\Af$  would be $(q+(1-q)(1-(1-p_{\Delta}')^{n-2})$, where $p_{\Delta}' = \frac{\pprime-p'}{1-p'}$. So after $\Ab$, the density becomes
\[q' = \b(q+(1-q)(1-(1-p_{\Delta}')^{n-2})\b)\cdot \E_{G\sim \rgt(n,p,\pprime)}\Pr_{\Ab}\b(e\in \Ab(G+e)\b)\,.\qedhere\]
\end{proof}

We verify in Appendix~\ref{sec:postponed-proofs} that the resulting edge density $q'=q\cdot p_e$ is still $q+o(1)$.

\begin{lemma}\label{lem:reverse-q}
The parameter $p_e$ in Theorem~\ref{thm:PDSreverse} satisfies \[p_e=1-O(np') = 1+o_n(1)\,.\]
\end{lemma}

In the rest of the section, we will prove Theorem~\ref{thm:PDSreverse}. Recall that in Section~\ref{sec:apply-thm-to-reverse} we proved Theorem~\ref{thm:PDSreverse} given Lemma~\ref{lem:classG} (that the graphs $\cG$ have high probability) and Lemma~\ref{lem:edge_on_other_edge} (that $\Ab$ satisfies perturbation insensitivity). What remains is to prove both lemmas. 
We will prove Lemma~\ref{lem:classG} in 
 Section~\ref{sec:graph-class} via a concentration argument. We start with Lemma~\ref{lem:edge_on_other_edge}.


\begin{mylem}{\ref{lem:edge_on_other_edge}}
    If $G\in \cG_1$,
then
$
\cL_{G-e}(Y_{\sim e}) = 
\cL_{G+e}(Y_{\sim e}|Y_e=0)\approx_{O(\log^{17/2}n/\sqrt{n})} \cL_{G+e}(Y_{\sim e}|Y_e=1)
$.
\end{mylem}

In this section we always focus on $G$ and $G+e$ where $e=(u,v) \notin G$. 
Let $T_e=T_{G+e}\backslash T_G$ denote the set of triangles in $G+e$ containing edge $e$. For an indicator $x$ of triangles, we will always use $x_{e}$ to denote $x_{T_e}$ and $x_{\nine}$ to denote $x_{T_G}$.
Let $W$ denote the set of edges in $G$ that are incident to $e$. 

Fix $G$ and let $X\sim \mu_{G+e} (\spacedot| Y_e=0) $ and $\Xp\sim \mu_{G+ e}(\,\cdot\, |Y_e=1)$, with $Y$ and $\Yp$ the corresponding edge distributions. Lemma~\ref{lem:edge_on_other_edge} is equivalent to $Y_{\sim e}$ and $\Yp_{\sim e}$ being close in distribution, i.e. we want to prove that
\begin{equation}\label{e:pertY}
    \TV(Y_{\sim e},\Yp_{\sim e})=O(\log^{17/2}n/\sqrt{n})
    \,.
\end{equation}

\begin{proof}[Outline of Proof of Lemma~\ref{lem:edge_on_other_edge}]
    We actually do not work with $X^+$ and $Y^+$, instead replacing these by $X^*$ and $Y^*$, which we define later in 
Section~\ref{sec:starQuantities}
and will turn out to be easier to work with.
We will show that
\begin{equation}\label{e:star}
\begin{split}
\TV(Y_{\sim e},\Yp_{\sim e})&=      \Theta(\TV(Y_{\sim e},Y^*_{\sim e})/\log n)\\
 \TV\b(X_{\nine},X^+_{\nine})& =      \Theta(\TV\b(X_{\nine},X^*_{\nine}) /\log n\b)\,.
\end{split}
\end{equation}

Now,
similar to the decomposition \eqref{eq:invariance-decomposition}, we decompose $\TV(Y^*_{\sim e}, Y_{\sim e})$ using Lemma~\ref{lem:TVdecomp} to obtain
\begin{equation}\label{eq:invariance-decomposition-real}
\begin{split}
    \TV(Y_{\sim e},Y^*_{\sim e})
    &\leq \TV\b((Y_{\sim e}, X_{\nine}),(Y^*_{\sim e},X^*_{\nine})\b)
    \\
    &\le \TV\b(X_{\nine},X^*_{\nine}) +\TV\b(Y_{\sim e} ,  \E_{X'\sim X_{\nine}} \law(Y^*_{\sim e}|X^* = X')\b)\,.
\end{split}
\end{equation}
Section~\ref{sec:edge_on_other_triangles} follows the sketch in Section~\ref{sec:non-incident-triangles} to bound 
 $\TV\b(X_{\nine},X^+_{\nine})$ which by \eqref{e:star}
bounds
the first term in \eqref{eq:invariance-decomposition-real} by $O(\log n/\sqrt{n})$. 
Section~\ref{sec:edge-on-other-edge} uses the ideas in
Section~\ref{sec:projection} 
to bound the second term in \eqref{eq:invariance-decomposition-real} by $O(\log^{15/2}n/\sqrt{n})$.

\end{proof}

\subsection{Influence of an Edge on Non-Containing Triangles }\label{sec:edge_on_other_triangles}

We will now bound $\TV\b(X_{\nine},X^+_{\nine})$.
In Section~\ref{sec:edge_on_other_triangles} only, we will overload notation and
take $\Xp\sim \mu_{G+e}$ rather than $\mu_{G+e}(\spacedot|Y_e=1)$. This is justified by the following lemma, which is immediate from the definitions.

\begin{lemma} \label{lem:XplusMixture}
Let $X\sim \mu_G$, $X^{+,1}\sim \mu_{G+e}(\spacedot|Y_e=1)$, and $X^{+,2}\sim\mu_{G+e} $. Then,
\[
\TV\b(X_{\nine},X^{+,1}_{\nine}) =\frac{1}{\Pr(Y_e=1)}
\TV\b(X_{\nine},X^{+,2}_{\nine}) \,.
\]
\end{lemma}

\begin{lemma}\label{lem:influence_edge_triangles}
Let $G$ be a $c$-uniformly 2-star dense graph for a constant $c>0$. Let $e\in {[n]\choose 2}$ be any edge and let $X\sim \mu_{G+e} (\spacedot| Y_e=0)$
and $\Xp\sim \mu_{G+ e}$. Emphasizing that $X_{\nine} = X$, we have
\[\TV\b(X_{\nine},  \Xp_{\nine}\b)=O({\log n}/{\sqrt{n}})\,.\]
\end{lemma}

This rest of this section is structured as follows: First, in Section~\ref{sec:chi-square-influence}, we establish the relationship between $\chi^2(\Xpe\| X_{\nine})$ and the influence of the distribution $\law_G(Y)$, as presented in Lemma~\ref{lem:expression-influence}. Subsequently, in Section~\ref{sec:calculating-chi-square}, we combine Lemma~\ref{lem:expression-influence} with the bound on influence obtained in Lemma~\ref{lem:edge-influence-adjacent} to compute the desired upper bound in Lemma~\ref{lem:influence_edge_triangles}.

\subsubsection{Bounding $\chi^2$-divergence In Terms of Influence}\label{sec:chi-square-influence}

The entirety of Section~\ref{sec:chi-square-influence} is dedicated to showing Lemma~\ref{lem:expression-influence}, which states that the $\chi^2$-divergence between $X_{\nine}$ and $\Xpe$ is bounded in terms of influence.
\expressioninfluence*

We restate the lemma that bounds chi-square divergence for a mixture of Gibbs distributions.

\mixturegibbs*

\begin{proof}
By definition of chi-square divergence,
\begin{align*}
   \chi^2(Q\|P)+1 =\ \E_{X\sim P}\B(\frac{dQ}{dP}(X)\B)^2
   =&\ \E_{X\sim P} \frac{\E_{U}[f_U(X)/Z_U] \E_{U'}[f_{U'}(X)/Z_{U'}]}{ (f(X)/Z)^2}\\
   =&\ \E_{U,U'}\B[\frac{Z^2}{Z_UZ_{U'}} \cdot \E_{X\sim P}\frac{f_U(X)f_{U'}(X)}{(f(X))^2}\B]
   \,.
\end{align*}                               
Notice that \[\frac{Z_U}{Z} = \sum_X\frac{f_U(X)}{f(X)}\frac{f(X)}{Z} = \E_{X\sim P}\rho_U(X)\,.\]
Taking this into the previous expression finishes the proof.
\end{proof}

Applying the above lemma to $X_{\nine}$ and $\Xp_{\nine}$ gives us the following upper bound on their chi-square distance. 
\begin{corollary}\label{cor:chi-square-mixture}
Let $A = E(\Xpet)-e$, and $B = E(\Xpe)-e$. Recall that $Y$ is the indicator of $E(X)$, we have
    \begin{equation}\label{eq:GibbsChi}
    \chi^2\b( \Xp_{\nine}\| X_{\nine}\b)+1\le 
\E_{\Xpe,\Xpet}\frac{\E \B[p^{|Y_A|} \E\b[p^{|Y_B|}|Y_A\b]\B]}{\E \b[p^{|Y_A|}\b] \E \b[p^{|Y_B|}\b]}\,.
\end{equation}
\end{corollary}

\begin{proof}
Let $P=\cL(X_{\nine}) = \cL(X)=\mu_G$ and $Q=\cL(\Xp_{\nine})$. Define $f(x):=\b(\frac{p'}{1-p'}\b)^{|x|}p^{-\e(x)}$ so that $P(x)=f(x)/Z_G$ by Defn.~\ref{def:triangle_dist} of $\mu_G$.

The distribution of $\Xp_{\nine}$ can be viewed as a mixture of Gibbs distribution given $\Xpe$.
We have
\begin{align*}
    \mu_{G+e}(\Xpne|\Xpe) &\propto  p'^{{|\Xpne|}+|\Xpe|}p^{-\e(\Xpne)-\e(\Xpe)}p^{|E(\Xpne)\cap E(\Xpe)|} 
\\&\propto 
p'^{|\Xpne|}p^{-\e(\Xpne)}p^{|E(\Xpne)\cap E(\Xpe)|}\,.
\end{align*}
We can define 
\[f_{\Xpe}(X) = p'^{n_{X}}p^{-e(X)}p^{|E(X)\cap E(\Xpe)|}=f(X)p^{|E(X)\cap E(\Xpe)|}\] and \[Z_{\Xpe} = \sum_{X\in \{0,1\}^{T_G}}f_{\Xpe}(X)\] so that $$\mu_{G+e}(\Xpne|\Xpe) =f_{\Xpe}(\Xpne)/Z_{\Xpe}\,.$$
So $\mu_{G+e}$ can be viewed as the mixture of Gibbs distributions, i.e., 
\[\mu_{G+e}(X) = \E_{\Xpe}[f_{\Xpe}(\Xpne)/Z_{\Xpe}] \,.\]

Note that $f_{\Xpe}(X)/f(X)=p^{|E(T)\cap E(\Xpe)|}$. We can let $\Xpet $ be an independent copy of $\Xpe$,
and apply Lemma~\ref{lem:chi-square-gibbs} to obtain
\begin{align*}
    \chi^2(Q\| P)+1 &= \E_{\Xpe,\Xpet}\frac{\E_{X\sim P}p^{|E(X)\cap E(\Xpe)|} p^{|E(X)\cap E(\Xpet)|}}{\E_{X\sim P}p^{|E(X)\cap E(\Xpe)|} \E_{X\sim P}p^{|E(X)\cap E(\Xpet)|}}\\
    &=  \E_{\Xpe,\Xpet} \frac{\E \B[p^{|Y_{E(\Xpe)}|} p^{|Y_{E(\Xpet)}|} \B]}{ \E p^{|Y_{E(\Xpe)}|} \E p^{|Y_{E(\Xpet)}|}}
    \,.
\end{align*}
For simplicity of notation, let $A=E(\Xpet)-e$, $B=E(\Xpe)-e$ ($e\notin E(X)$ so $Y_e=0$). The above expression can be written as 
\begin{equation*}
\E_{\Xpe,\Xpet}\frac{\E \B[p^{|Y_A|} \E\b[p^{|Y_B|}|Y_A\b]\B]}{\E \b[p^{|Y_A|}\b] \E \b[p^{|Y_B|}\b]}\,.\qedhere
\end{equation*}
\end{proof}

From Corollary~\ref{cor:chi-square-mixture}, we can derive Lemma~\ref{lem:expression-influence} by showing two upper bounds for the expression within $\E_{\Xpe,\Xpet}$ in \eqref{eq:GibbsChi}.

\begin{lemma}
Let $\Inf^{M}$ denote the marginal influence of distribution $\law_G(Y)$. For any $A,B\subset E(G)$, we have
\[
\frac{\E \B[p^{|Y_A|} \E\b[p^{|Y_B|}|Y_A\b]\B]}{\E \b[p^{|Y_A|}\b] \E \b[p^{|Y_B|}\b]}\le (1/p)^{|A|+|B|}
\]
and
    \[
    \frac{\E \B[p^{|Y_A|} \E\b[p^{|Y_B|}|Y_A\b]\B]}{\E \b[p^{|Y_A|}\b] \E \b[p^{|Y_B|}\b]}\le \B(1+\frac{1-p}{p}\sup_{\substack{A'\subset A\cup B
\\ e\in (A\cup B)\backslash A'}} \IM{A'}{e}\B)^{|B|}(1/p)^{|A\cap B|}\,.
    \]
\end{lemma}
\begin{proof}
The first upper bound is straightforward. On the one hand,
$\E \B[p^{|Y_A|} \E\b[p^{|Y_B|}|Y_A\b]\B]\le1 $, and on the other hand, $\E \b[p^{|Y_A|}\b] \E \b[p^{|Y_B|}\b]\ge p^{|A|+|B|}$, so we have 
\[
\frac{\E \B[p^{|Y_A|} \E\b[p^{|Y_B|}|Y_A\b]\B]}{\E \b[p^{|Y_A|}\b] \E \b[p^{|Y_B|}\b]} \le (1/p)^{|A|+|B|}\,.
\]

For the second upper bound, we compare the difference between $\E\b[p^{|Y_B|}|Y_A\b]$ and $\E \b[p^{|Y_B|}\b]$ by constructing a coupling between $Y_B$ and $Y_B|Y_A=y_A$. The reason that influence is useful to us, is that a good coupling can be constructed for a distribution with low influence, as stated in the following lemma.
\begin{lemma}\label{lem:coupling-general}
Let $X$ be a random variable on $\{0,1\}^N$. $A,B\subset [N]$ are disjoint. If for any subset $A'\subset A\cup B$ and $i\in (A\cup B)\backslash A'$, 
\[\IM{A'}{i}\le q,\]
then for any $x_A,x_A'\in \{0,1\}^B$, there exists coupling $(X_B,X'_B)$ between $\law(X_B|X_A=x_A)$ and $\law(X_B|X_A=x_A')$  such that
$|X_B-X'_B|$ is stochastically dominated by $\bino(|B|,q)$.
\end{lemma}

We prove the lemma later in the section.
Now use $\maxI$ to denote $\sup_{\substack{A'\subset A\cup B
\\ e\in (A\cup B)\backslash A'}} \IM{A'}{e}$ and apply Lemma~\ref{lem:coupling-general} on $A$ and $B\backslash A$. 
There exists a coupling $(Y_{B\backslash A},Y_{B\backslash A}')$ between $\cL_G(Y_{B\backslash A}|Y_A=y_A)$ and $\cL_{G}(Y_{B\backslash A})$ such that 
$|Y_{B\backslash A}-Y_{B\backslash A}'|$ is stochastically dominated by $\bino\b(|B\backslash A|,\maxI\b)$. Therefore, we can couple $Y_{A\cap B}$ arbitrarily and obtain a coupling $(Y_B,Y_B')$ between $\cL_G(Y_{B}|Y_A=y_A)$ and $\cL_{G}(Y_{B})$ where $|Y_B-Y_B'|$ is stochastically dominated by $\bino\b(|B\backslash A|,\maxI\b)+|A\cap B|$.

Let $M(\theta) = \E[e^{\theta Z}]$ be the moment generating function (MGF) of $Z\sim \bino(|B|,q^*(|A|+|B|))+|A\cap B|$. By the MGF of binomial distribution, $M(\theta) = (1-\maxI+e^\theta \maxI)^{|B|}e^{\theta |A\cap B|}$.
We have
\[
\E\b[ p^{|Y_B'|-|Y_B|} \b]\le \E\b[ (1/p)^{|Y_B-Y_B'|}\b] \le
M(\log(1/p))
=\B(1+\frac{1-p}{p}\maxI \B)^{|B|}(1/p)^{|A\cap B|}\,.
\]
So for any $y_A$,
\[
\E\b[p^{|Y_B|}|Y_A=y_A\b] = \E\b[p^{|Y_B'|}\b]\le \E\b[p^{Y_B}\b] \B(1+\frac{1-p}{p}\maxI\B)^{|B|}(1/p)^{|A\cap B|}\,.
\]
Hence
\[
\frac{\E \B[p^{|Y_A|} \E\b[p^{|Y_B|}|Y_A\b]\B]}{\E \b[p^{|Y_A|}\b] \E \b[p^{|Y_B|}\b]} \le \B(1+\frac{1-p}{p}\maxI\B)^{|B|}(1/p)^{|A\cap B|}\,. \qedhere
\]
\end{proof}

\begin{proof}[Proof of Lemma~\ref{lem:coupling-general}]
Suppose $B=\{i_1,\cdots, i_k\}$. We couple one coordinate at a time. Since $\IM{A}{i_1}\le q$, 
\[
\b|\Pr(X_{i_1}=1|X_A=x_A)- \Pr(X_{i_1}=1|X_A=x_A')\b|\le q\,.
\]
Let $(X_{i_1},X_{i_1}')$ be the optimal coupling between the two distributions. $\Pr(X_{i_1}\not= X_{i_1}')\le q$. The coupling for other coordinates can be defined similarly one by one. Again from the assumption,
\[\IM{A\cup\{i_1,\cdots,i_{j-1}\}}{i_j} \le q\,.\]
So 
\[
\b|\Pr\b(X_{i_j}=1|X_A=x_A,X_{\{i_1,\cdots,i_{j-1}\}}\b)- \Pr\b(X_{i_j}=1|X_A=x_A',X_{\{i_1,\cdots,i_{j-1}\}}'\b)\b| \le q\,,
\]
so we can set $(X_{i_j}, X'_{i_j})$ conditioned on $X_{\{i_1,\cdots,i_{j-1}\}},X'_{\{i_1,\cdots,i_{j-1}\}}$ to be the optimal coupling between the two distributions and have
\[
\Pr(X_{i_j}\not= X'_{i_j}|X_{\{i_1,\cdots,i_{j-1}\}},X'_{\{i_1,\cdots,i_{j-1}\}})\le q\,.
\]
Equivalently, $|X_{i_j}- X'_{i_j}|$ conditioned on $X_{\{i_1,\cdots,i_{j-1}\}},X'_{\{i_1,\cdots,i_{j-1}\}}$ is stochastically dominated by $\bern(q)$. So $|X_B-X_B'|$ is stochastically dominated by $\bino(|B|,q)$.
\end{proof}

\subsubsection{Plugging Influence Bound into Lemma~\ref{lem:expression-influence} to Obtain Lemma~\ref{lem:influence_edge_triangles}}
\label{sec:calculating-chi-square}

\begin{proof}[Proof of Lemma~\ref{lem:influence_edge_triangles}]
To evaluate the bound in Lemma~\ref{lem:expression-influence}, we use the bound on the marginal influence of $\law_G(Y)$ in Lemma~\ref{lem:edge-influence-adjacent}. Let $q^*=O(1/(n^3p'^2)) = \tilde O(1/n)$ be the factor in the bound in Lemma~\ref{lem:edge-influence-adjacent}, i.e., 
$\IM A{e}\leq q^*|A|/n$ for any $A$ such that $|A|\le cn/4$. It then follows that
for $|A|+|B|\le cn/4$, \[\sup_{\substack{A'\subset A\cup B
\\ e\in (A\cup B)\backslash A'}} \IM{A'}{e}\le q^*(|A|+|B|)\,.\]
We 
now evaluate the bound in Lemma~\ref{lem:expression-influence} in terms of $q^*$. 

Noting that the part within $\E_{\Xpe,\Xpet}$ in the right-hand side is monotone in $|A|$ and $|B|$, we begin by showing an upper bound on their distributions. For any triangle $t$ in $T_e$, $\Pr(t\in \Xpe)\leq c_0\cdot p'$ for some $c_0>0$ since $\mu_{G+e}$ is $O(p')$ marginally small. It follows that the distribution of $|\Xpe|$ (and $|\Xpet|$) is stochastically dominated by $\bino(|T_e|,c_0p')$. As $|T_e|\le n$, they are also stochastically dominated by $\bino(n,c_0p')$. 

Note that $A$ and $B$ do not include edge $e,$ $|A|= 2|\Xpe|$, and $|B|= 2|\Xpet|$, so $|A\cap B| = 2|\Xpe\cap \Xpet|$. By using the first bound in the minimum in the conclusion of Lemma~\ref{lem:expression-influence} when $|A|,|B|$ are not larger than $2\log n$ and using the second bound otherwise, we obtain
\begin{align*}
&\chi^2( \Xp_{\nine}\|X_{\nine}) \\
&\le  \sum_{1\le i,j\le 2\log n}\sum_{k=0}^{ \min\{i,j\}} \Pr\b(|\Xpet|=i,|\Xpe|=j,|\Xpe\cap \Xpet|=k\b)\\
&\quad \cdot\B(\B(1+\frac{1-p}{p}q^*(2i+2j)\B)^{2j}(1/p)^{2k}-1\B) \\
&\quad + 
\sum_{\substack{i\ge 2\log n\text{ or } \\j\ge 2\log n} } \Pr\b(|\Xpet|=i,|\Xpe|=j\b) \B( (1/p)^{2i+2j} -1\B)\\
&\le  \sum_{1\le i,j\le 2\log n} \sum_{k=0}^{ \min\{i,j\}} \binom{n}{k} \binom{n-k}{i-k}\binom{n-i-k}{j-k}(c_0p')^{i+j}\B(\B(1+\frac{1-p}{p}q^*(2i+2j)\B)^{2j}(1/p)^{2k} -1\B) \\
&\quad + \sum_{\substack{i\ge 2\log n\text{ or } \\j\ge 2\log n} } \binom{n}{i}\binom{n}{j} (c_0p')^{i+j}(1/p)^{2i+2j}\,.
\end{align*}

Because $\binom{n}{i}\le n^i$ and $\binom{n}{j}\le n^j$, the second summation is bounded by \[\sum_{\substack{i\ge 2\log n\text{ or } \\j\ge 2\log n} } \b(\frac{c_0np'}{p}\b)^{i+j}\,.\]
Since $np'\ll 1$, this sum is $o(1/n^2)$. 

For the first summation, we can use $\binom{n}{k} \binom{n-k}{i-k}\binom{n-i-k}{j-k}\le n^{i+j-k}$ and separate the $k=0$ summand to bound it by
\begin{equation}\label{e:bigsum}
    \sum_{1\le i,j\le 2\log n} \bigg\{ (c_0np')^{i+j}\B(\B(1+\frac{1-p}{p}q^*(2i+2j)\B)^{2j} -1\B) + \B(1+\frac{1-p}{p}q^*(2i+2j)\B)^{2j} \sum_{k\ge 1}(np^2)^{-k}  \bigg\}\,.
\end{equation}
Since $q^*=\tilde{O}(1/n)$, we have $2j\cdot \frac{1-p}{p}q^*(2i+2j)\ll1$, so
\[
\B(1+\frac{1-p}{p}q^*(2i+2j)\B)^{2j}\le 1+ \frac{8(1-p)}{p}q^* j(i+j)\le 2\,.
\]
Therefore, for large enough $n$, the quantity in \eqref{e:bigsum} can be upper bounded by 
\[
\sum_{i,j\ge 1}\frac{8(1-p)}{p}q^* (c_0np')^{i+j}j(i+j) + \frac{8\log^2 n}{np^2}\,.
\]
As $np'\ll1$, we have $\sum_{i\ge1}(c_0np')^ii=O(np')$ and $\sum_{i\ge1}(c_0np')^ii^2=O(np')$. So \eqref{e:bigsum} is bounded by $O(np'q^*)+O(\log^2 n/n)=O(\log^2 n/n)$. Now that we proved $\chi^2(Q\|P)=O(\log^2 n/n)$, by Pinsker's inequality, 
\[\TV(P,Q)\le \sqrt{2\chi^2(Q\| P)}=O(n^{-1/2}\log n)\,. \]
This finishes the proof for Lemma~\ref{lem:influence_edge_triangles}.
\end{proof}

\subsection{Replacing $X^+$ and $Y^+$ by $X^*$ and $Y^*$}
\label{sec:starQuantities}

As compared to the sketch in Section~\ref{sec:projection}, there are a number of additional steps. 
Instead of working with $X$ and $\Xp$, we replace $\Xp$ by a new random variable $X^*$ constructed as a mixture of the distributions of $X$ and $\Xp$ so that the conditional distribution on $X^*_{e}$ given $X^*_{\nine}$
is a product distribution. As shown in the proof, it will suffice to bound the total variation between $Y_{\sim e}$ and $Y^*_{\sim e}$.
We first define $X^*$.
\paragraph{$X^*$ as a mixture of $\law(X)$ and $\law(\Xp)$ to make $\law(X^*_{e}|X^*_{\nine})$ a product distribution}
The conditioning on edge $e$ introduces dependence between triangles in $T_e$. 
We modify Definition~\ref{def:triangle_dist} to yield a new distribution over sets of triangles in $G+e$. Let
\[g^*_{G+e}(x) \defeq \B(\frac{p'}{1-p'}\B)^{|x|}p^{-|E(x)-e|}\quadand \mu^*_{G+e}(X^*=x) = \frac{g^*_{G+e}(x)}{Z^*(G+e)}\,.\] 
Here $Z^*(G+e) = \sum_{x:E(x)\subset G+e}g^*(x)$ is the normalizing factor. The benefit of defining $\mu^*$ is that now 
\[\mu^*_{G+e}(X_{e}|X_{\nine}) \propto \B(\frac{p'}{1-p'}\B)^{n_{X_{e}}}p^{-|E(X_{e})-E(X_{\nine})-e|}\,.\]
Because triangles in $T_e$ only intersect at $e$, $|E(X_{e})-E(X_{\nine})-e |=\sum_{t\in T_e}X_t |E(t)- E(X_{\nine})|$. So 
\begin{equation}\label{eq:conditional_triangle_dist}
    \mu^*_{G+e}(X_{e}|X_{\nine}) \propto \prod_{t\in T_e}\B(\frac{p'p^{-|E(t)- E(X_{\nine})|}}{1-p'}\B)^{X_t}
\end{equation}
is now a product distribution on triangles in $T_e$.

Note by definition of $g^*$ that 
\[g^*_{G+e}(X) = \begin{cases}
pg_{G+e}(X)&\text{if } e\in E(X)\\
g_{G+e}(X)&\text{otherwise}\,.
\end{cases}\]
Thus, $\mu^*$ re-weights by $p$ those $X$ with $e\in E(X)$, which is the same as $Y_e=1$. It follows
 that $\mu^*_{G+e}$ is equal to
 the mixture
\begin{equation}\label{eq:mustar_decomposition}
   \mu^*_{G+e}(\spacedot)  = \frac{p\mug{G+e}(Y_e=1)\mu_{G+e}(\spacedot|Y_e=1) +\mug{G+e}(Y_e=0)\mu_{G+e}(\spacedot|Y_e=0) }{p\mug{G+e}(Y_e=1)+ \mug{G+e}(Y_e=0)}\,. 
\end{equation}

Let $X^*\sim \mu_{G+e}^*$, and $Y^*$ be the corresponding edge distribution of $X^*$ on $\{0,1\}^{E(G+e)}$. Recalling that $Y^+$ has law $\law_{G+e}(Y|Y_e=1)$ and $\law(Y)=\law_{G+e}(Y|Y_e=0)$, we also have 
\begin{equation}\label{e:Ystar}
    \law(Y^*) = \frac{p\mu_{G+e}(Y_e=1)\law(\Yp) +\mu_{G+e}(Y_e=0)\law(Y) }{p\mu_{G+e}(Y_e=1)+ \mu_{G+e}(Y_e=0)}\,.
\end{equation}

\paragraph{Why we can replace $X^+$ and $Y^+$ with $X^*$ and $Y^*$}
The utility of introducing $X^*$ and $Y^*$ is that they are easier to work with, due the product structure in \eqref{eq:conditional_triangle_dist}, and we now show that we can replace $\Xp$ and $\Yp$ by $X^*$ and $Y^*$ in the arguments that follow, as promised in
\eqref{e:star}.
By Corollary~\ref{cor:muge-marginal}, we have that $\mug{G+e}(Y_e=1)=\tilde{\Theta}(1)$, and it follows from \eqref{e:Ystar} that
\[\begin{split}
\TV(Y_{\sim e},\Yp_{\sim e}) &= \frac{p\mug{G+e}(Y_e=1)+ \mug{G+e}(Y_e=0)}{p\mug{G+e}(Y_e=1)} \TV(Y_{\sim e},Y^*_{\sim e}) \\
&= 
\Theta(\TV(Y_{\sim e},Y^*_{\sim e})/(np')) =\tilde      \Theta(\TV(Y_{\sim e},Y^*_{\sim e}))\,.
\end{split}\]




As far as replacing $X^+$ by $X^*$, the identical argument as above gives from \eqref{eq:mustar_decomposition} that 
\[\TV(X^*_{\nine}, X_{\nine}) = \Theta(np'\cdot \TV(\Xp_{\nine}, X_{\nine})) = \tilde \Theta\b( \TV(\Xp_{\nine}, X_{\nine})\b)
\,.\]

It remains is to bound the second term in \eqref{eq:invariance-decomposition-real} using the ideas from Section~\ref{sec:projection}.

\subsection{Perturbation Insensitivity and Proof of Lemma~\ref{lem:edge_on_other_edge}}\label{sec:edge-on-other-edge}

In this section bound the second term of \eqref{eq:invariance-decomposition-real}. The ideas sketched in Section~\ref{sec:projection} are implemented, transforming the comparison between edge variables back into a comparison between triangle distributions.  
Then we compare the triangle distributions by accounting for one triangle addition at a time, and finally, we bound the total variation induced by a single triangle via concentration of the likelihood ratio.

For the simplicity of discussion, let us create a random variable that follows distribution $\E_{X'\sim X_{\nine}} \law(Y^*_{\sim e}|X^* = X')$. Specifically, let $X''$ be generated from $X_{\nine}$ by channel $\mu^*_{G+e}(\spacedot| X_{\nine})$, and $Y''$ be the corresponding edge distribution, so $\law(Y) =\E_{X'\sim X_{\nine}} \law(Y^*_{\sim e}|X^* = X') $. This way, the second term of \eqref{eq:invariance-decomposition-real} is equal to 
\[
\TV\b(Y_{\sim e} ,  \E_{X'\sim X_{\nine}} \law(Y^*_{\sim e}|X^* = X')\b) = \TV(Y_{\sim e}, Y''_{\sim e})\,.
\]

The goal of this subsection is to bound the latter quantity.
\begin{lemma}\label{lem:second-term}
   $ \TV(Y_{\sim e}, Y''_{\sim e}) =O\B(\frac{\log^3  n}{n^5p'^{9/2}}\B)= O(\log^{15/2}/\sqrt{n})$.
\end{lemma}


\subsubsection{Step 1: Projection of Distributions and Defining $\aux$}
Although we will show that $Y_{\sim e}$ and $Y''_{\sim e}$ are close, $X$ and $X''$ are actually very different distributions, in fact, they have $\tilde{\Omega}(1)$ TV-distance in general.
Indeed, $X_{e}=\vec{0}$, but $X''_{e}|X_{\nine}$ follows \eqref{eq:conditional_triangle_dist}. $X''_{e}$ might bring new edges that makes $Y''$ different from $Y$. 
We will construct an auxiliary triangle distribution $\Tilde{X}$ with $Y''_{\sim e}= \tilde{Y}_{\sim e}$, and by data-processing inequality it suffices to bound $\TV(X,\tilde{X})$
(see Figure~\ref{fig:projection}):

\begin{lemma} If $\tilde X$ is such that $Y''_{\sim e}= \tilde{Y}_{\sim e}$, then 
   $ \TV(Y_{\sim e}, Y''_{\sim e}) = \TV(Y_{\sim e}, \tilde Y_{\sim e}) \leq \TV(X,\tilde X)$.
\end{lemma}

Specifically, $\aux$ will be constructed such that when added to $X$, it results in the same edge projection as $X''$, i.e., 
\begin{equation}\label{eq:aux-edge}
   E(X'')-\{e\} = E(\aux\lor X)\,, 
\end{equation}
We can set $\tilde{X}=\aux\lor X$. Note that $Y''_{\sim e}$ is the indicator for edge set $E(X'')-\{e\}$, this means $Y''_{\sim e}= \tilde{Y}_{\sim e}$.


We now formally define $\aux$ by its distribution conditioned on $X''=X\lor X''_{e}$ (see Figure~\ref{fig:aux}). To satisfy \eqref{eq:aux-edge}, $\aux$ needs to include the set of new edges that are in $E(X'')$ but not in $E(X)$,
$$\enew := E(X'')-E(X)-\{e\}=E(X''_{e})-E(X)-\{e\}\subset W\,.$$ 
Start with $\aux$ being 0. Recall the definition of wedge set in \eqref{eq:def-wedge-set}. For each edge $(u,w)$ (or $(v,w)$) in $\enew$, pick a node $w'$ uniformly at random from $\wed[E(X)](u,w)$ (or $\wed[E(X)](v,w)$) and let $\aux_{(u,w,w')}=1$ (or $\aux_{(v,w,w')}=1$).

\begin{lemma}\label{lem:xaux-edge}
$\enew= E(\aux)-E(X)$.
\end{lemma}

\begin{proof}
Let $(u,w)$ be an edge in $E(X'')-E(X)-\{e\}$, and $w'$ be the pick from $\wed[E(X)](u,w)$ that $\aux_{(u,w,w')}=1$. We have $(u,w)\in E(X'')$. By \eqref{eq:def-wedge-set},   $(u,w'),(w,w')\in E(X)$. So $E(\aux)\subset E(X'')$. 
As $e\notin E(X)$, $u$ or $v$ are not in $\wed[E(X)](u,w)$, we have $E(\aux)\subset E(X'')-\{e\}$, so 
\[E(\aux)-E(X)\subset E(X'')-E(X)-\{e\}\,.\]
Since each edge in $\enew$ is included in $\aux$, we know \[E(\aux)-E(X)\supset \enew\,.\qedhere\]
\end{proof}
Using this lemma we have $E(\tilde{Y}_{\sim e})=E(\aux\lor X)=E(X'')-\{e\}=E(Y''_{\sim e}).$

The remainder of this section shows that $X$ and $\tilde{X}=X\lor \aux$ are close in total variation. 
\begin{lemma}\label{lem:xaux-triangle} $X$ and $\tilde{X}=X\lor \aux$ satisfy the total variation bound
\[
\TV(X,\tilde{X})=O\B(\frac{\log^3  n}{n^5p'^{9/2}}\B)=\tilde{O}\B(\frac{1}{\sqrt{n}}\B)\,.
\]
\end{lemma}

\subsubsection{Step 2: Representation of $\aux$}

To prove the lemma, we need to know the distribution of $\aux|X$. Because $\aux$ is defined through $\enew$, we first consider its distribution. By definition, upon conditioning on $X_{\nine}$, $X''_{e}$ has distribution
 \[\mu^*_{G+e}(X_{e}|X_{\nine})\propto \prod_{t\in T_e}\B(\frac{p'p^{-|E(t)- E(X_{\nine})|}}{1-p'}\B)^{X_t}\,.\]
 So for $t=(w,u,v)\in T_e$, 
\[X''_t|X_{\nine}\sim \begin{cases}
\bern(p_0) \text{ if none of $(w,u)$ and $(w,v)$ are in $E(X)$}\\
\bern(p_1) \text{ if one of $(w,u)$ and $(w,v)$ are in $E(X)$}\\
\bern(p_2) \text{ if both $(w,u)$ and $(w,v)$ are in $E(X)$}\,.
\end{cases}\]
Here $p_i= \frac{p'p^{-3+i}}{1-p'+p'p^{-3+i}}$ is $\Theta(p')$, $\forall i=0,1,2$. 
Define
\begin{align*}
    V_v(X) &\defeq \{w:(w,u)\in E(X), (w,v)\in G\backslash E(X)\}\\
V_u(X) &\defeq \{w:(w,v)\in E(X), (w,u)\in G\backslash E(X)\}\\
V_{uv}(X) &\defeq \{w:(w,u), (w,v)\in G\backslash E(X)\}\,.
\end{align*}
So $\enew$ has the following distribution conditioned on $X$: 
\begin{align*}
    &\text{For $w\in V_v(X)$, }Y_{(v,w)}\sim \bern(p_1);\\
    &\text{For $w\in V_u(X)$, }Y_{(v,w)}\sim \bern(p_1);\\
    &\text{For $w\in V_{uv}(X)$, }Y_{(v,w)}=Y_{(u,w)}\sim \bern(p_0)\,.
\end{align*}
Recalling the definition of $\aux$, its conditional distribution on $X$ is therefore generated by
\begin{align*}
    \text{For $w\in V_v(X)$, sample }& w'\sim \unif(\wed[E(X)](v,w)),\  \aux_{(v,w,w')}\sim \bern(p_1);\\
    \text{For $w\in V_u(X)$, sample }& w'\sim \unif(\wed[E(X)](u,w)),\  \aux_{(u,w,w')}\sim \bern(p_1);\\
    \text{For $w\in V_{uv}(X)$, sample }&w'\sim \unif(\wed[E(X)](v,w)),\  \aux_{(v,w,w')}\sim \bern(p_0), \\
    \text{sample }& w''\sim \unif(\wed[E(X)](u,w)),\ \aux_{(u,w,w'')}\sim \bern(p_0) \,.
\end{align*}

\subsubsection{Step 3: Adding one triangle at a time}
\begin{figure}
    \centering
    \tikzset{every picture/.style={line width=0.75pt}} 

\begin{tikzpicture}[x=0.75pt,y=0.75pt,yscale=-1,xscale=1]

\draw [color={rgb, 255:red, 126; green, 211; blue, 33 }  ,draw opacity=1 ]   (62,210) -- (115,112) ;
\draw [color={rgb, 255:red, 126; green, 211; blue, 33 }  ,draw opacity=1 ]   (197,134) -- (115,112) ;
\draw [color={rgb, 255:red, 126; green, 211; blue, 33 }  ,draw opacity=1 ]   (155,213) -- (197,134) ;
\draw [color={rgb, 255:red, 126; green, 211; blue, 33 }  ,draw opacity=1 ]   (241,210) -- (294,112) ;
\draw [color={rgb, 255:red, 126; green, 211; blue, 33 }  ,draw opacity=1 ]   (376,134) -- (294,112) ;
\draw [color={rgb, 255:red, 126; green, 211; blue, 33 }  ,draw opacity=1 ]   (334,213) -- (376,134) ;
\draw [color={rgb, 255:red, 208; green, 2; blue, 27 }  ,draw opacity=1 ]   (296,119) -- (364,138) ;
\draw [color={rgb, 255:red, 208; green, 2; blue, 27 }  ,draw opacity=1 ]   (330,207) -- (364,138) ;
\draw [color={rgb, 255:red, 208; green, 2; blue, 27 }  ,draw opacity=1 ]   (330,207) -- (296,119) ;
\draw  [color={rgb, 255:red, 74; green, 144; blue, 226 }  ,draw opacity=1 ] (211,143) -- (228.4,143) -- (228.4,133) -- (240,153) -- (228.4,173) -- (228.4,163) -- (211,163) -- cycle ;
\draw [color={rgb, 255:red, 208; green, 2; blue, 27 }  ,draw opacity=1 ]   (247,210) -- (296,119) ;

\draw (93,86) node [anchor=north west][inner sep=0.75pt]  [color={rgb, 255:red, 126; green, 211; blue, 33 }  ,opacity=1 ] [align=left] {$\displaystyle E( X)$};
\draw (285,87) node [anchor=north west][inner sep=0.75pt]  [color={rgb, 255:red, 208; green, 2; blue, 27 }  ,opacity=1 ] [align=left] {$\displaystyle E( K_{v}( X))$};
\draw (64,213) node [anchor=north west][inner sep=0.75pt]  [font=\large] [align=left] {$\displaystyle u$};
\draw (243,213) node [anchor=north west][inner sep=0.75pt]  [font=\large] [align=left] {$\displaystyle u$};
\draw (157,216) node [anchor=north west][inner sep=0.75pt]  [font=\large] [align=left] {$\displaystyle v$};
\draw (336,216) node [anchor=north west][inner sep=0.75pt]  [font=\large] [align=left] {$\displaystyle v$};
\draw (117,120) node [anchor=north west][inner sep=0.75pt]  [font=\large] [align=left] {$\displaystyle w$};
\draw (195,117) node [anchor=north west][inner sep=0.75pt]  [font=\large] [align=left] {$\displaystyle w'$};
\draw (307,130) node [anchor=north west][inner sep=0.75pt]  [font=\large] [align=left] {$\displaystyle w$};
\draw (376,116) node [anchor=north west][inner sep=0.75pt]  [font=\large] [align=left] {$\displaystyle w'$};

\end{tikzpicture}
    \caption{Illustration of $K_v$}
    \label{fig:kv}
\end{figure}
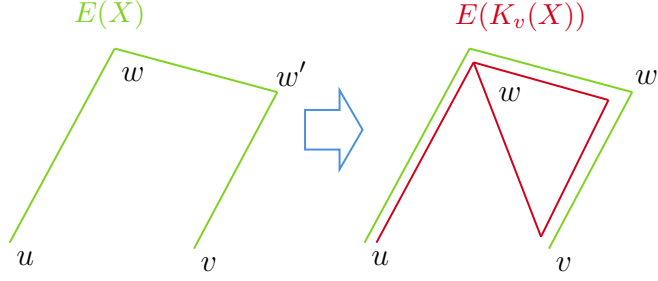

We next aim to describe $\aux$ in terms of single triangle additions. To this end we define a few kernels from $\{0,1\}^{T_G}$ to the same space. First,
\[K_v(X):\text{Pick $w\sim \unif(V_v(X))$, $w'\sim \unif(\wed[E(X)](v,w))$, set $X_{(v,w,w')}=1$}\,,\]
as depicted in Figure~\ref{fig:kv}.
Similarly, 
\[K_u(X):\text{Pick $w\sim \unif(V_u(X))$, $w'\sim \unif(\wed[E(X)](u,w))$, set $X_{(u,w,w')}=1$}\,,\]
and
\begin{align*}
    K_{uv}(X)&:\text{Pick $w\sim \unif(V_{uv}(X))$, $w'\sim \unif(\wed[E(X)](v,w))$, set $X_{(v,w,w')}=1$}\\
    &\;\;\; \text{Pick $w'\sim \unif(\wed[E(X)](u,w))$, set $X_{(u,w,w')}=1$}\,.
\end{align*}
Then 
\[X\lor \aux \disteq (K_{uv})^{m_3}(K_u)^{m_2}(K_v)^{m_1}(X)\,,\]
where $m_1\sim \bino(|V_v(X)|,p_1)$, $m_2\sim \bino(|V_u(X)|,p_1)$ and $m_3\sim \bino(|V_{uv}(X)|,p_0)$. We will show that $K_u(X)$, $K_v(X)$ and $K_{uv}(X)$ are all $\tilde{O}(1/\sqrt{n})$-close to $X$ in total variation distance in Lemmas~\ref{lem:Kvx-vs-x} and \ref{lem:Kuvx-vs-x} just below.

\begin{lemma}\label{lem:Kvx-vs-x} We have the bounds
$\TV(K_u(X), X)=O(\log n/(n^{4}p'^{7/2}))$ and $\TV(K_v(X), X)=O(\log n/(n^{4}p'^{7/2}))$.
\end{lemma}

\begin{lemma}\label{lem:Kuvx-vs-x} We have that
$\TV(K_{uv}(X), X)=O(\log^3 n/(n^6p'^{11/2}))$.
\end{lemma}
We prove Lemma~\ref{lem:Kvx-vs-x} in the next subsection and defer the proof of Lemma~\ref{lem:Kuvx-vs-x} to Appendix~\ref{sec:postponed-proofs}, as the proofs are similar. With these lemmas, we can prove Lemma~\ref{lem:second-term}.

\begin{proof}[Proof of Lemma~\ref{lem:second-term}]
From Lemma~\ref{lem:kernel-tv}, Lemmas~\ref{lem:Kvx-vs-x} and \ref{lem:Kuvx-vs-x} imply that
\[
\TV((K_{uv})^{m_3}(K_u)^{m_2}(K_v)^{m_1}(X), X) = O\B(\frac{(m_1+m_2)\log n}{n^4p'^{7/2}}\B) + O\B(\frac{m_3\log^3  n}{n^6p'^{11/2}}\B)\,.
\]

Note that $|V_v(X)|, |V_u(X)|, |V_{uv}(X)|\le n$ and $p_i=\Theta(p')$, so we have $\E[m_i] =O(np')$. Hence from convexity of total variation,
\[
\TV(X,\tilde{X})=O(np')\cdot O\B(\frac{\log^3  n}{n^6p'^{11/2}}\B) = O\B(\frac{\log^3  n}{n^5p'^{9/2}}\B)\,.
\]
This completes the proof of Lemma~\ref{lem:edge_on_other_edge}. 
\end{proof}
It remains to prove Lemmas~\ref{lem:Kvx-vs-x} and \ref{lem:Kuvx-vs-x}.

\subsection{Proof of Lemma~\ref{lem:Kvx-vs-x} via Likelihood Concentration}

\label{sec:likeconc}

The definitions of $K_u$ and $K_v$ are symmetric with respect to swapping $u$ and $v$, so it suffices to prove the claim for one of them, e.g., $K_v$.

The intuition behind Lemma~\ref{lem:Kvx-vs-x} is that $K_v$ adds a triangle to $X$ on $v,w,w'$ in order to add edge $(v,w)$ without introducing other new edges (see Figure~\ref{fig:kv}). Here $(u,w)$ is already in $E(X)$ and $(v,w)\notin E(X)$ (but is in $G$), and moreover, each of the edges $(v,w')$ and $(w,w')$ are in $E(X)$. Now, if (i) there are $\tilde \Omega(n)$ such $(v,w)$ with triangle $(v,w,w')$ already present in $X$, and (ii) there are $\tilde \Omega(n)$ possible positions to add an additional one, then we expect adding the new triangle to induce a total variation $\tilde{O}(1/\sqrt{n})$\footnote{This is by the heuristic of computing the shift in mean and comparing to the standard deviation, which for Gaussians and other well-behaved distributions is well-known to give the correct total variation.}. Item (i) is shown later in Lemma~\ref{lem:concentration-number-of-terms-Kv} and
Item (ii) is shown shortly in Lemma~\ref{lem:good_xtg0}. The conclusion, the $\tilde{O}(1/\sqrt{n})$ total variation bound, is stated as Lemma~\ref{lem:kvs-vs-x-conditioning}.

\subsubsection{Conditioning on Non-Incident Triangles}

Define a partition of $T_G$ as follows: 
\[T_G^0 \defeq \{t\in T_G:t\text{ does not contain $u$ or $v$}\}\,,\]
\[T_G^v \defeq \{t\in T_G:t\text{ contains  $v$}\}\,,\quadand T_G^u \defeq \{t\in T_G:t\text{ contains  $u$}\}\,.\]

Before bounding the total variation difference, we will condition on triangles in $T_G^0$. This will be helpful as the concentration results we use (Theorem~\ref{thm:small_marginal_concentration}) depend on the dimension of the distribution. By conditioning on $X_{T_G^0}$, the dimension of the distribution changes from $|T_G|$, which typically has size $\Theta(n^3)$, to $|T_G^u \cup T_G^v| $, which has size at most $O(n^2)$.

\begin{restatable}{lemma}{goodxtg}\label{lem:good_xtg0}
Let $G$ be $c$-uniformly 2-star dense. With high probability over the randomness of $X_{T_G^0}$, for any $w\notin\{u,v\}$, there are $\Theta(n^2p') $ different $w'$ such that $(w',v)\in G$ and $(w',w)\in E(X_{T_G^0})$.
\end{restatable}

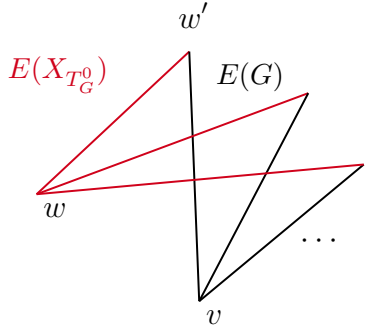
\begin{figure}
    \centering
    \tikzset{every picture/.style={line width=0.75pt}} 

\begin{tikzpicture}[x=0.75pt,y=0.75pt,yscale=-1,xscale=1]

\draw [color={rgb, 255:red, 208; green, 2; blue, 27 }  ,draw opacity=1 ]   (99,158) -- (176,86) ;
\draw    (181,212) -- (176,86) ;
\draw [color={rgb, 255:red, 208; green, 2; blue, 27 }  ,draw opacity=1 ]   (99,158) -- (236,107) ;
\draw    (181,212) -- (236,107) ;
\draw    (181,212) -- (264,143) ;
\draw [color={rgb, 255:red, 208; green, 2; blue, 27 }  ,draw opacity=1 ]   (99,158) -- (264,143) ;

\draw (101,161) node [anchor=north west][inner sep=0.75pt]  [font=\large] [align=left] {$\displaystyle w$};
\draw (183,215) node [anchor=north west][inner sep=0.75pt]  [font=\large] [align=left] {$\displaystyle v$};
\draw (169,60) node [anchor=north west][inner sep=0.75pt]  [font=\large] [align=left] {$\displaystyle w'$};
\draw (83,87) node [anchor=north west][inner sep=0.75pt]  [color={rgb, 255:red, 208; green, 2; blue, 27 }  ,opacity=1 ] [align=left] {$\displaystyle E( X_{T_{G}^{0}})$};
\draw (188,90) node [anchor=north west][inner sep=0.75pt]   [align=left] {$\displaystyle E( G)$};
\draw (230.5,176.5) node [anchor=north west][inner sep=0.75pt]  [font=\large] [align=left] {$\displaystyle \cdots $};

\end{tikzpicture}
    \caption{Lemma~\ref{lem:good_xtg0} states that with high probability, $X_{T_G^0}$ satisfies that for any $w\notin\{u,v\}$, there are $\tilde{\Omega}(n)$ different $w'$ such that $(w',v)\in G$ and $(w',w)\in E(X_{T_G^0})$.}
    \label{fig:xtg0}
\end{figure}

\begin{proof}
Fix $w\notin\{u,v\}$.
By definition of $c$-uniformly 2-star density, there are $\Theta(n)$ different $w'$ that both $(w',w),(w',v)\in G$. Each edge $(w',w)$ is contained in $\Omega(n)$ triangles in $T_G^0$. Since $\mug{G}$ is $\Omega(p')$-marginally large, the number of $w'$ that $(w',w)\in E(X_{T_G^0})$ stochastically dominates $\bino(\Omega(n),\Omega(np'))$. So it is $\Omega (n^2p')$ with high probability by the Chernoff bound. $\mug{G}$ is also $\Omega(p')$-marginally small, so by similar argument, the number of $w'$ that $(w',w)\in E(X_{T_G^0})$ is also $O(n^2p')$ with high probability.
\end{proof}

We will prove the following lemma, which implies Lemma~\ref{lem:Kvx-vs-x}. 

\begin{lemma}\label{lem:kvs-vs-x-conditioning}
For any $X_{T_G^0}$ satisfying the conclusion of Lemma~\ref{lem:good_xtg0}, 
\[\TV\b(K_v(X)|X_{T_G^0}, X|X_{T_G^0}\b)=O\B(\frac{\log n}{n^{4}p'^{7/2}}\B)=\tilde{O}\B(\frac{1}{\sqrt{n}}\B)\,.\] 
\end{lemma}

\begin{proof}[Proof of Lemma~\ref{lem:Kvx-vs-x}]
The $X_{T_G^0}$ satisfying the conclusion of Lemma~\ref{lem:good_xtg0} occur with high probability, hence the total variation bound in Lemma~\ref{lem:kvs-vs-x-conditioning} occurs with high probability, and the result follows by applying Lemma~\ref{lem:bad_event}.
\end{proof}

In the remainder of Section~\ref{sec:likeconc} we prove Lemma~\ref{lem:kvs-vs-x-conditioning}.

\subsubsection{Lemma~\ref{lem:kvs-vs-x-conditioning} via Likelihood Ratio Concentration}


To use Lemma~\ref{lem:concentration-TV} to prove Lemma~\ref{lem:kvs-vs-x-conditioning}, we will show that the likelihood ratio between the two distributions concentrates. We start by determining the likelihood ratio. 

\paragraph{Likelihood Ratio}
 Note that the distribution of $X|X_{T_G^0}$ is just $\mu_G(X|X_{T_G^0})$. We can write the density of both distributions.
\begin{equation}\label{eq:density-kvx}
\begin{split}
   &\Pr(K_v(X)=x|X_{T_G^0})\\&\qquad= \sum_{\substack{t=(v,w,w'):t\in x\\ w\in V_v(x-t),w'\in \wed[E(x-t)](v,w)}} \frac{1}{|V_v(x-t)|}\frac{1}{|\wed[E(x-t)](v,w)|}\mu_G(X=x-t|X_{T_G^0})\,, 
\end{split}
\end{equation}
and 
\[\Pr(X=x|X_{T_G^0}) = \mu_G(X=x|X_{T_G^0})\,.\]
Here $x-t$ stands for the vector that is the same as $x$ except $x_t$ is set to 0.
So the likelihood radio is 
\begin{align}\label{eq:derivative}
 &   L(x) := \frac{\Pr(K_v(X)=x|X_{T_G^0})}{\Pr(X=x|X_{T_G^0})} \nonumber
    \\&\qquad\qquad= \sum_{\substack{t=(v,w,w'):t\in x\\ w\in V_v(x-t),w'\in \wed[E(x-t)](v,w)}} \frac{1}{|V_v(x-t)|}\frac{1}{|\wed[E(x-t)](v,w)|} \frac{\mu_G(X=x-t)}{\mu_G(X=x)}\,.
\end{align}

\begin{lemma}[Concentration of Likelihood Ratio]\label{lem:concentration-likelihood}
For any $X_{T_G^0}$ that satisfies the conclusion of Lemma~\ref{lem:good_xtg0}, we have that
\[\left| L(X)- 1\right|=O\B(\frac{\log n}{n^{4}p'^{7/2}}\B)=\tilde{O}(1/\sqrt{n})\,,\]
with high probability.
\end{lemma}



With Lemma~\ref{lem:concentration-TV}, Lemma~\ref{lem:concentration-likelihood} immediately implies
Lemma~\ref{lem:kvs-vs-x-conditioning}. 
In the remainder of the section, our focus will be on proving Lemma~\ref{lem:concentration-likelihood}.

\subsubsection{Concentration of Components of Likelihood Ratio}
The proof will follow from three concentration statements for factors in the likelihood ratio~\eqref{eq:derivative}:
\begin{enumerate}
    \item[(a)] $\frac{\mu_G(X=x-t)}{\mu_G(X=x)}$ is constant with respect to $x$,
    \item[(b)] $|V_v(X)|$  and $|\wed[E(X)](v,w)|$ concentrate around their mean,
    \item[(c)] the total number of terms in \eqref{eq:derivative} also concentrates.
\end{enumerate}

\paragraph{(a) $\frac{\mu_G(X=x-t)}{\mu_G(X=x)}$ is constant with respect to $x$}
To start, note that if $t$ is chosen by $K_v$ to be added to $X$, it always increases the number of edges by 1, i.e., $e(x)-e(x-t)=1$ if $t$ satisfies the summation in \eqref{eq:derivative}. So
\[\frac{\mu_G(X=x-t)}{\mu_G(X=x)} = \frac{p'^{|x-t|}p^{-e(x-t)}}{p'^{|x|}p^{-e(x0)}} = \frac{p}{p'}\,,\]
which is constant as claimed.

\paragraph{(b) Concentration of $|V_v(X)|$ and $|\wed[E(X)](v,w)|$}
By Corollary~\ref{cor:concentration_muG}, any Lipschitz function of a marginally small random variable concentrates. From the definitions, $|V_v(X)|$ and $|\wed[E(X)](v,w)|$ are $O(1)$-Lipschitz, yielding the following lemma.

\begin{lemma}\label{lem:var-Vv-Wvw}
Conditioned on any $X_{T_G^0}$ satisfying the condition in Lemma~\ref{lem:good_xtg0},
with high probability, \[\b|\,|V_v(X)|-\E|V_v(X)|\,\b|= O(\sqrt{n^2p'}\log n)\] and \[\b|\,|\wed[E(X)](v,w)|-\E|\wed[E(X)](v,w)|\,\b|= O(\sqrt{n^2p'}\log n)\,.\]
\end{lemma}
\begin{proof}
$\law(X|X_{T_G^0})$ is a conditional distribution of $\mu_G$ over $O(n^2)$ triangles. So we can apply Corollary~\ref{cor:concentration_muG}. With high probability, any $O(1)$-Lipschitz function is within distance  $O(\sqrt{n^2p'}\log n)$ from its mean with high probability.
\end{proof}

To use the concentration result, we also need to show that the means of $|V_v(X)|$ and $|\wed[E(X)](v,w)|$ are large relative to their variances.
\begin{lemma}\label{lem:e-Vv}
For any $X_{T_G^0}$ satisfying the condition in Lemma~\ref{lem:good_xtg0},
\[\E \b[|V_v(X)|\,\b|X_{T_G^0}\b]=\Theta(n^2p')
= \tilde{\Theta}(n)
\,,\]
 and for any $w\not=u,v$
\[\E \b[|\wed[E(X)](v,w)|\b|X_{T_G^0}\b]=\Theta(n^3p'^2)
=\tilde{\Theta}(n)
\,.\]
\end{lemma}

\begin{proof}
Recall that $V_v(X)$ consists of vertices $w$ that $(w,u)\in E(X), (w,v)\notin E(X)$. Fix any $w\in W_{u,v}(G)$. From Lemma~\ref{cor:muge-marginal},
\[\Pr((w,v)\notin E(X)) = \Omega(1-np') = \Theta(1)\,.\]
By Lemma~\ref{lem:margin-under-conditioning}, 
\[\Pr((u,w)\in E(X)|(w,v)\notin E(X)) = \Theta(np')\,.\]
So $\Pr(w\in V_v(X))=\Theta(np').$ We have $\E[|V_v(X)|] = |W_{u,v}(G)|\cdot \Theta(np') = 
\Theta(n^2p')$. 

Fix any $w\not=u,v$. Consider a vertex $w'$ such that $(w,w')\in E(X_{T_G^0}), (v,w')\in G$, since $G$ is $c$-uniformly 2-star dense, $(v,w')$ is included in $\Theta(n)$ triangles, each having probability $\Theta(p')$ of being in $X$. So
\[\Pr((v,w')\in E(X)|X_{T_G^0})=\Theta(np')\,.\]
From the assumption on $X_{T_G^0}$, there are $\Theta(n^2p')$ such $w'$, so 
\[\E \b[|\wed[E(X)](v,w)|\b|X_{T_G^0}\b] = \Theta(n^2p')\cdot \Theta(np') = \Theta(n^3p'^2) \,.\qedhere\]
\end{proof}

Note that $\wed[E(X)](v,w)\ge |\wed[E(X_{T_G^0})](v,w)|=\Theta(n)$, so $\E \b[|\wed[E(X)](v,w)|\b|X_{T_G^0}\b]=\Theta(n)$. By Lemma~\ref{lem:var-Vv-Wvw} and Lemma~\ref{lem:e-Vv}, for any $v$ and $w$, conditioned on any $X_{T_G^0}$,
\[\B|\frac{|V_v(X)|-\E |V_v(X)|}{\E |V_v(X)|}\B|=O\B(\frac{\log n}{np'^{1/2}}\B)\text{ and }\B|\frac{|\wed[E(X)](v,w)|-\E |\wed[E(X)](v,w)|}{\E |\wed[E(X)](v,w)|}\B|=O\B(\frac{\log n}{n^2p'^{3/2}}\B)\]
holds with high probability.

Note $|V_v(X-t)|$ and $|V_v(X)|$ differ by at most one, and the same holds for $|W_{v,w}(X-t)|$ and $|\wed[E(X)](v,w)|$. 
So with high probability, each term in \eqref{eq:derivative} concentrates,
\begin{align}\label{eq:each-term}
\begin{split}
    &\B|\frac{1}{|V_v(X-t)|}\frac{1}{|W_{v,w}(X-t)|} \frac{\mu_G(X-t)}{\mu_G(X)} - \frac{1}{\E |V_v(X)|}\frac{1}{\E |\wed[E(X)](v,w)|} \frac{p}{p'}\B|\\
    &\qquad =  O\B(\frac{\log n}{n^2p'^{3/2}}\B)\cdot \frac{1}{\E |V_v(X)|}\frac{1}{\E |\wed[E(X)](v,w)|} \frac{p}{p'}=O(\frac{\log n}{n^7p'^{11/2}})=\tilde{O}(n^{-3/2})\,.\\
\end{split}
\end{align}

\paragraph{(c) Concentration of the Number of Terms in \eqref{eq:derivative}.}
It remains to show that the number of terms
also concentrates. Let $S$ be the set of ordered vertices
\[S(X)= \{(w,w'):t=(v,w,w'),t\in X, w\in V_v(X-t),w'\in W_{v,w}(X-t)\}\,,
\]
as illustrated in Figure~\ref{fig:sx}
Each triangle $(v,w,w')$ in the summation of \eqref{eq:derivative} corresponds a pair of vertices $(w,w')$ in $S(X)$.

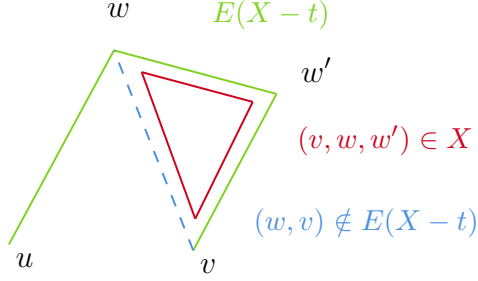
\begin{figure}
    \centering
    \tikzset{every picture/.style={line width=0.75pt}} 

\begin{tikzpicture}[x=0.75pt,y=0.75pt,yscale=-1,xscale=1]

\draw [color={rgb, 255:red, 126; green, 211; blue, 33 }  ,draw opacity=1 ]   (101,200) -- (154,102) ;
\draw [color={rgb, 255:red, 126; green, 211; blue, 33 }  ,draw opacity=1 ]   (236,124) -- (154,102) ;
\draw [color={rgb, 255:red, 126; green, 211; blue, 33 }  ,draw opacity=1 ]   (194,203) -- (236,124) ;
\draw [color={rgb, 255:red, 208; green, 2; blue, 27 }  ,draw opacity=1 ]   (168,113) -- (224,128) ;
\draw [color={rgb, 255:red, 208; green, 2; blue, 27 }  ,draw opacity=1 ]   (195,187) -- (224,128) ;
\draw [color={rgb, 255:red, 208; green, 2; blue, 27 }  ,draw opacity=1 ]   (195,187) -- (168,113) ;
\draw [color={rgb, 255:red, 74; green, 144; blue, 226 }  ,draw opacity=1 ] [dash pattern={on 4.5pt off 4.5pt}]  (194,203) -- (181.76,172.09) -- (154,102) ;

\draw (245,138) node [anchor=north west][inner sep=0.75pt]  [color={rgb, 255:red, 208; green, 2; blue, 27 }  ,opacity=1 ] [align=left] {$\displaystyle ( v,w,w') \in X$};
\draw (103,203) node [anchor=north west][inner sep=0.75pt]  [font=\large] [align=left] {$\displaystyle u$};
\draw (196,206) node [anchor=north west][inner sep=0.75pt]  [font=\large] [align=left] {$\displaystyle v$};
\draw (149,77) node [anchor=north west][inner sep=0.75pt]  [font=\large] [align=left] {$\displaystyle w$};
\draw (247,104) node [anchor=north west][inner sep=0.75pt]  [font=\large] [align=left] {$\displaystyle w'$};
\draw (202,75) node [anchor=north west][inner sep=0.75pt]  [color={rgb, 255:red, 126; green, 211; blue, 33 }  ,opacity=1 ] [align=left] {$\displaystyle E( X-t)$};
\draw (223,180) node [anchor=north west][inner sep=0.75pt]  [color={rgb, 255:red, 74; green, 144; blue, 226 }  ,opacity=1 ] [align=left] {$\displaystyle ( w,v) \notin E( X-t)$};

\end{tikzpicture}
    \caption{Definition of $S(X)$. A pair $w,w'$ is in $S(X)$ if $t=(v,w,w')\in X, (w,u), (w,w'), (v,w')\in E(X-t), (w,v)\notin E(X-t)$.}
    \label{fig:sx}
\end{figure}

\begin{lemma}\label{lem:concentration-number-of-terms-Kv}
$|S(X)|$ is $O(1)$-Lipschitz, and
under $\law(X|X_{T_G^0})$, 
\[\E\b[|S(X)|\,\b|X_{T_G^0}\b]=\Theta(n^5p'^4)=\tilde{\Theta}(n)\,.\]
\end{lemma}

\begin{proof}
From the definition of $V_v$ and $W_{v,w}$, $S$ can be alternatively written as
\[S(X)= \{(w;w'):t=(v,w,w'),t\in X, (w,u), (w,w'), (v,w')\in E(X-t), (w,v)\notin E(X-t) \}.\]
Suppose $X$ and $X'$ differ at one triangle, $X'=X\vee e_{t_0}$, we want to show $S(X')$ and $S(X)$ differ by $O(1)$ elements.

First consider $S(X)\backslash S(X')$. Suppose $(w;w')\in S(X)$ but $(w;w')\notin S(X+t_0)$. Then we must have $(w,v)\notin E(X-t)$ but $(w,v)\in E(X+t_0-t)$. 
Since  $E(X+t_0-t)$ and $E(X-t)$ differ by at most 3 edges, we have $|\{w:\exists w', (w;w')\in S(X)\backslash S(X')\}|\le 3$. Notice that each for each $w$, at most one $w'$ satisfy $(w;w')\in S(X)$, otherwise $(w,v)\in E(X-t)$. So $|S(X)\backslash S(X')|\le 3$.

Then consider $S(X')\backslash S(X)$. If $(w;w')\in S(X')\backslash S(X)$, either $(w,u)\in E(X+t_0-t)-E(X-t)$, $(w,w')\in E(X+t_0-t)-E(X-t)$ or $(v,w')\in E(X+t_0-t)-E(X-t)$. Let \[S_1=\{(w;w'): (w;w')\in S(X'), (w,u)\in E(X+t_0-t)-E(X-t)\}\,,\]
\[S_2=\{(w;w'): (w;w')\in S(X'), (w,w')\in E(X+t_0-t)-E(X-t)\}\,,\]
and
\[S_3=\{(w;w'): (w;w')\in S(X'), (v,w')\in E(X+t_0-t)-E(X-t)\}\,.\]
Then 
\[S(X')\backslash S(X)\subset S_1\cup S_2\cup S_3\,.\]
Still note that for each $w$, at most one $w'$ satisfies $(w;w')\in S(X')$, and $E(X+t_0-t)-E(X-t)$ has at most 3 edges, we have $|S_1|\le 3$, $|S_2|\le 3$. For $S_3$, notice that each $w'$ has at most one $w$ such that $(w;w')\in S(X')$ and $(v,w')\in E(X+t_0-t)-E(X-t)$. Otherwise by definition of $S(X')$, at least two triangles in $X$ contains edge $(v,w')$. So $(v,w')$ would belong to $E(X-t)$ for any $t$, contradicting that $(v,w')\in E(X+t_0-t)-E(X-t)$. Therefore, $|S_3|\le 3$ also holds.
We have $|S(X')\backslash S(X)|\le 9$.


Next we show that $\E |S(X)|=\Theta(n^5p'^4)$. It is sufficient to prove there are $\Theta(n^3p')$ pairs $(w;w')$ that $(w;w')$ is in $S(X)$ with probability $\Theta(n^2p'^3)$. Since $G$ is a $c$-uniformly 2-star dense graph, there are $\Theta(n)$ vertices $w$ such that $(w,u),(w,v)\in G$. From the assumption that $|\wed[E(X_{T_G^0})](v,w)|=\Theta (n^2p')$, there are $ \Theta(n^3p')$ pairs $(w;w')$ that $(w,u),(w,v)\in G$ and $(w,w')\in E(X_{T_G^0})$.

Next we show that for such $(w;w')$, it is in $S(X)$ with probability $\Theta(n^2p'^3)$.
Consider the event that $(w;w')\in S(X)$, i.e., 
\[t=(v,w,w'),t\in X, (w,u), (w,w'), (v,w')\in E(X-t), (w,v)\notin E(X-t)\,.\]
Since $E(X-t)\supset E(X_{T_G^0})$, $(w,w')\in E(X-t)$ always holds.
First, $t\in X$ with probability $\Theta(p')$, next consider the event conditioning on this. $(w,v)$ are included in at most $n-2$ triangles, each being included with probability $\Theta(p')$. $(w,v)\notin E(X-t)$ with probability at least
\[(1-\Theta(p'))^{n-2}=\Theta (1)\,.\]
 Graph $G$ is $c$-uniformly 2-star dense, so there are $\Theta(n)$ triangles contains $(w,u)$ and $(v,w')$ and do not contain $(w,v)$. By Lemma~\ref{lem:mug_marginally_small}, conditioning on the all the events above, each triangle is included in $X$ with probability $\Theta(p')$, so $(w,u),(v,w')\in E(X-t)$ with probability $\Theta(n^2p'^2)$.
Therefore, the event $(w;w')\in S(X)$ happens with probability $\Theta(n^2p'^3).$
\end{proof}

\subsubsection{Proof of Concentration of the Likelihood Ratio}

\begin{proof}[Proof of Lemma~\ref{lem:concentration-likelihood}]
Because $S(X)$ is Lipschitz by Lemma~\ref{lem:concentration-number-of-terms-Kv},   Corollary~\ref{cor:concentration_muG} applies to yield $\b||S(X)|-\E S(X)\b|=\tilde{O}(\sqrt{n})$ with high probability. So $S(X)=\tilde{O}(n)$ with high probability.

Now for any $w$, 
\[\frac{1}{\E |V_v(X)|}\frac{1}{\E |\wed[E(X)](v,w)|} \frac{p}{p'}=\Theta\B(\frac{1}{n^5p'^4}\B)\,,\]
so again because $S(X)$ is Lipschitz we have that the following function is $O(1/(n^5p'^4))$-Lipschitz, as each term in the right hand side is $O(1/(n^5p'^4))$: 
\[f(x) \defeq \sum_{(w;w')\in S(x)}\frac{1}{\E |V_v(X)|}\frac{1}{\E |\wed[E(X)](v,w)|} \frac{p}{p'}\,.\]
Again by Corollary~\ref{cor:concentration_muG}, $f(X)$ concentrates around its mean with distance at most $O(\log n/(n^4p'^{7/2}))$ with high probability. 

Now combining this with \eqref{eq:each-term}, we get for any fixed $X_{T_G^0}$ that satisfies the conclusion of Lemma~\ref{lem:good_xtg0}, the likelihood ratio 
$$
L(x)=\frac{\Pr(K_v(X)=x|X_{T_G^0})}{\Pr(X=x|X_{T_G^0})}
$$
concentrates as
\[\left|L(X)
- \E f(X) \right|=O\B( \frac{\log n}{n^4p'^{7/2}}\B)+ O\B(|S(X)|\frac{\log n}{n^{7}p'^{11/2}}\B)\,.\]
This is $O(\log n/(n^{4}p'^{7/2})) = \tilde{O}(1/\sqrt{n})$ with high probability since $|S(X)|=\Theta(n^5p'^4)$ with high probability. 

We can use Lemma~\ref{lem:distance-expectation} (which states that concentration implies concentration about the mean for bounded random variables), as it is not hard to check that the likelihood ratio is always bounded by $\tilde{O}(n^3)$. Hence, with high probability,
\[\left| L(X)- \E L(X) \right|=O\B(\frac{\log n}{n^{4}p'^{7/2}}\B)=\tilde{O}(1/\sqrt{n})\,.\]
This completes the proof of Lemma~\ref{lem:concentration-likelihood}, using that $\E  L(X) = 1$. 
\end{proof}

\subsection{The class of graphs $\cG$ has high-probability}\label{sec:graph-class}

The goal of this section is to show Lemma~\ref{lem:classG}, i.e., 
\[\Pr(\rgt_i\in \cG)= 1-o(1)\,.\]
Recall the definition of $\cG$ in Section~\ref{sec:apply-thm-to-reverse}, $\cG = \cG_1\cap \cG_2$. Here $\cG_1$ is the set of graphs that are $p^2/3$-uniformly 2-star dense and $\cG_2$ is the set of graphs that satisfy \[\b|\mug{G+e}(Y_e=1)-p_e^+\b| = 
C_p n^{-5/2}p'^{-2}\sqrt{\log n}\,,
\quad \text{where}\quad 
p_e^+ = \E_{G'\sim \rgt(n,p,p')}[\mug{G'+e}(Y_e=1)]\,,
\]
where $C_p$ is a large enough constant depending on $p$.
We can prove that $\cG_2$ is a high-probability event under $\rgt(n,p,p')$ through concentration.
\begin{restatable}{lemma}{concentration}\label{lem:concentration_pe}
If $G\sim \rgt(n,p,p')$, then $\Pr(G\in \cG_2) = 1-o(1)$.
\end{restatable}
Let us first show why  Lemma~\ref{lem:classG} holds given Lemma~\ref{lem:concentration_pe}.

\begin{proof}[Proof of Lemma~\ref{lem:classG}]First, $\RGTi{0}$ is $p^2/2$-uniformly 2-star dense by Lemma~\ref{lem:probability-of-good} with probability $1-e^{-\Omega(n)}$. For any $\RGTi{i}$, $i\in [K]$, and any $j,k\in [n]$, the number of vertices $l$ such that $(j,l)$ or $(k,l)$ are different from $\RGTi{0}$ is at most $K$. So when $K=\tilde{O}(n^{1/4})$, $\RGTi{i}$ is $p^2/3$-uniformly 2-star dense if $\RGTi{0}$ is $p^2/2$-uniformly 2-star dense. Thus, $\RGTi{i}\in \cG_1$ for all $i\in [K]$ with high probability. 

Lemma~\ref{lem:concentration_pe} shows that $\RGTi{0}\in \cG_2$ with high probability. If $\RGTi{i}\in \cG_1$ for any $i$, then by Corollary~\ref{cor:influence_on_pe}, 
\[\b|\mug{\RGTi{i}+e}(Y_e=1)-\mug{\RGTi{0}+e}(Y_e=1)\b|=\tilde{O}(i/n)\,.
\]
Therefore, with high probability,
\[\b|\mug{\RGTi{i}+e}(Y_e=1)-\E_{G'\sim \rgt(n,p,p')}[\mug{G'+e}(Y_e=1)]\b| = \tilde{O}(1/\sqrt{n})+\tilde{O}(k^2/n) = \tilde{O}(1/\sqrt{n})\,.\qedhere
\]
\end{proof}

Next we prove Lemma~\ref{lem:concentration_pe}. The idea is to use Mcdiarmid's inequality and the fact that the marginal probability of $\law(Y)$ is a Lipschitz function due to bounded influence.



\begin{proof}[Proof of Lemma~\ref{lem:concentration_pe}]
Since $\rgt(n,p,p')$ is generated by adding triangles to random graphs, we can view it as a function over the original random graph and the triangles added. Let $x\sim \bern(p)^{\otimes \binom{[n]}{2}}$ be the indicator of edges in $G(n,p)$ and $y\sim \bern(p')^{\otimes \binom{[n]}{3}}$ be the indicator of whether a triangle is added to the graph. 
We use $\rgt(x,y)$ to denote the graph generated by choice $x$ and $y$. Then $\mug{G+e}(Y_e=1)$ can be viewed as a function of $(x,y)$, 
\[f(x,y)\defeq \mug{\rgt(x,y)+e}(Y_e=1)\,.\]

By Lemma~\ref{lem:edge_influence}, $\mug{G+e}(Y_e=1)$ is Lipschitz as a function of a $c$-uniformly 2-star dense graph $G$. Indeed, for any $e'\notin G$,
\begin{align*}
    \mug{G+e+e'}(Y_e=1)
    &= \mug{G+e+e'}(Y_{e'}=1) \mug{G+e+e'}(Y_e=1|Y_{e'}=1)\\
    &\qquad + \mug{G+e+e'}(Y_{e'}=0) \mug{G+e+e'}(Y_e=1|Y_{e'}=0)
    \,.
\end{align*}
As $\mug{G+e+e'}(Y_e=1|Y_{e'}=0)=\mug{G+e}(Y_e=1)$, we have
\begin{align*}
    |\mug{G+e+e'}(Y_e=1)- \mug{G+e}(Y_e=1)|
    &\le 
    |\mug{G+e+e'}(Y_e=1|Y_{e'}=1)-\mug{G+e+e'}(Y_e=1|Y_{e'}=0)|\\&\le \IM{e'}{e}.
\end{align*}
So for any $c$-uniformly 2-star dense graph $G$, and edge $e'$ incident to $e$,
\[|\mug{G+e+e'}(Y_e=1)- \mug{G+e}(Y_e=1)|=O(1/(n^3p'^2))=\tilde{O}(1/n)\,,\]
and for other edge $e'$ not incident to $e$,
\[|\mug{G+e+e'}(Y_e=1)- \mug{G+e}(Y_e=1)|=O(1/(n^3p'))=\tilde{O}(1/n^2)\,.\]
This implies that if $\rgt(x,y)$ is a $c$-uniformly 2-star dense graph, we have
\[\B|\frac{\partial f}{\partial x_{e'}}(x,y)\B|=
\begin{cases}
\tilde{O}(1/n) & \text{if } e'\cap e\neq\varnothing\\
\tilde{O}(1/n^2) & \text{otherwise}
\end{cases}
\,,
\]
and
\[\B|\frac{\partial f}{\partial y_{t}}(x,y)\B|=
\begin{cases}
\tilde{O}(1/n)& \text{if } t\cap e\neq \varnothing\\
\tilde{O}(1/n^2)& \text{otherwise}
\end{cases}
\,.
\]
Denote the quantities on the right hand side by $c_{e'}$ and $c_t$.
To use Lemma~\ref{lem:mcdiarmid-extension}, we need to check that $f$ is Lipschitz. For any $x_1,x_2\in \{0,1\}^{\binom{[n]}{2}}$ and $y_1,y_2\in \{0,1\}^{\binom{[n]}{3}}$ such that 
$\rgt(x_1, y_1)$ and $\rgt(x_2,y_2)$ are $c$-uniformly 2-star dense graphs, 
\begin{align*}
    &|f(x_1,y_1)-f(x_2,y_2)|\\
    &\le |f(x_1,y_1)-f(x_1\vee x_2,y_1\vee y_2)|+ |f(x_2,y_2)-f(x_1\vee x_2,y_1\vee y_2)|\\
    &\le  \sum_{e':(x_{1})_{e'}\not=(x_1\vee x_2)_{e'}}c_{e'} + \sum_{t:(y_{1})_{t}\not=(y_1\vee y_2)_{t}}c_{t} +\sum_{e':(x_{2})_{e'}\not=(x_1\vee x_2)_{e'}}c_{e'} + \sum_{t:(y_{2})_{t}\not=(y_1\vee y_2)_{t}}c_t\\
    &= \sum_{e':(x_{1})_{e'}\not=( x_2)_{e'}}c_{e'} + \sum_{t:(y_{1})_{t}\not=( y_2)_{t}}c_{t}\,.
\end{align*}
Here we used that $\rgt(x_1\vee x_2,y_1\vee y_2)=\rgt(x_1,y_1)\cup \rgt(x_2,y_2)$ is $c$-uniformly 2-star dense.

We can set $\mathcal{Y}$ in Lemma~\ref{lem:mcdiarmid-extension} to be the set of $x,y$ such that $\rgt(x,y)$ is $c$-uniformly 2-star dense.  
The probability of $G\sim \rgt(n,p,p')$ being not $c$-uniformly 2-star dense is bounded by $e^{-\Omega(n)}$ by Lemma~\ref{lem:probability-of-good}. So for $t>0$, we have
\begin{equation}\label{eq:concentration}
\Pr_G\B(\b|\mug{G+e}(Y_e=1)-\E_{G'\sim \rgt(n,p,p')}\mug{G'+e}(Y_e=1)\b| >t\B)\le e^{-\Omega(n)}+ \exp\B(\frac{-(t-e^{-\Omega(n)})^2}{\sigma^2+\tilde{O}(t/n)}\B)
\,.
\end{equation}
Here 
\begin{align*}
\sigma^2 &=   p(1-p)\B(\sum_{e'\in \binom{[n]}{2}:\text{$e'$ shares node with } e} O\B(\frac1{n^6p'^4}\B)+\sum_{e'\in \binom{[n]}{2}:\text{$e'$ does not node with } e} O\B(\frac1{n^6p'^2}\B)\B)
\\
&\quad 
+p'(1-p')\B( \sum_{t\in \binom{[n]}{3}:\text{$t$ shares nodes with } e} \B(\frac1{n^6p'^4}\B) + \sum_{t\in \binom{[n]}{3}:\text{$t$ does not node with } e} O\B(\frac1{n^6p'^2}\B) \B)
\\
&=
p(1-p)\B(n\cdot  O\B(\frac1{n^6p'^4}\B)+ n^2\cdot  O\B(\frac1{n^6p'^2}\B)\B) + p'(1-p') \B(n^2\cdot O\B(\frac1{n^6p'^4}\B) + n^3\cdot O\B(\frac1{n^6p'^2}\B) \B)\\
&=  O(n^{-5}p'^{-4})
\,.
\end{align*}
We can choose $t=\Theta( n^{-5/2}p'^{-2}\sqrt{\log n})$ so that the right hand side of \eqref{eq:concentration} is $1/n^c$ for a large enough constant $c$.
\end{proof}

\section{Low Influences for Marginally Small Distributions}\label{sec:small-marginal}

The distribution $\mu_G$ studied in the paper, as proven in Lemma~\ref{lem:mug_marginally_small}, is marginally small (Defn.~\ref{def:marg-small}).
In this section, 
we carry out a general investigation of  marginally small distributions and show, under an assumption on the underlying graphical model, that they have low influences. 
The influence bound on distribution $\law_G(Y)$ is then derived as a result of low influence in $\mu_G$.

We start by formally defining influence. 
\begin{definition}[Influence]\label{def:influence}
For a binary distribution over $\{0,1\}^\cX$, the \emph{influence} of $S\subset \cX$ on $S'\subset \cX$ is defined by $$\I{S}{S'}=\sup_{\substack{x_S^a,x_S^b\in \{0,1\}^S \\ \omega\in \{0,1\}^A,A\subset \cX-S-S'}}\TV(P_{x_{S'}|x_S^a,\omega}, P_{x_{S'}|x_S^b,\omega} )\,,$$ where $\omega$ represents an arbitrary fixing (``pinning") of variables in $A$. 

For $x,y\in \cX$, use $\I{x}{y}$ to denote $\I{\{x\}}{\{y\}}$. 
\end{definition}
\begin{remark}Definition~\ref{def:influence} is different from the marginal influence defined in Definition~\ref{def:marginal-influence} in the sense that it allows arbitrary pinning of the variables.
\end{remark}

\begin{remark}\label{rem:influence-spectual-independence}
The notion of influence plays a prominent role in the line of work using spectral independence to bound the mixing time of Markov chains, examples include \cite{anari2021spectral,chen2021rapid}.
The influence used in these works are usually the marginal influence, but a separate influence defined for each possible pinning. 
We instead maximize over all possible pinnings, which turns out to be more convenient. 
Another notion of influence appears in Dobrushin uniqueness conditions \cite{dobrushin1970prescribing}, where the pinning is over $\cX-S-S'$, not over an arbitrary subset as in Defn.~\ref{def:influence}.
\end{remark}


\subsection{Low Influences for Marginally Small Distributions}

\begin{theorem}\label{thm:small_marginal_low_influence}
    Let $P$ be a distribution over $\{0,1\}^N$, and $G$ be the graphical model of $P$ with maximum degree $\Delta$. Suppose that $P$ is $q$-marginally small
    where $q<\frac{1}{2\Delta}$ and let $$C=\sup_{k}\sup_{i,j:d(i,j)=k}|N_{k-1}(i)\cap N_1(j)|\,.$$
    Then for any $S\subset [N]$ and $i \in [N]$ such that $d(S,i)\ge d$, we can bound the influence of $S$ on $i$ as
    \begin{equation*}
        \I{S}{i}\le |S|\B(\frac{2Cq}{1-2q\Delta}\B)^d\,.
    \end{equation*}
    In particular, if $C=O(1)$ and $1-2q\Delta=\Omega(1)$, then $\I{j}{i} = O(q^{d(i,j)})$.
\end{theorem}
\begin{proof}
Noticing that all properties of $P$ are maintained under arbitrary conditioning, it is without loss of generality to prove \[\IM{S}{i}\le |S|\B(\frac{2Cq}{1-2q\Delta}\B)^d\,.\]

Let us fix vertex $i$, consider the possible influence on $i$ by a set $S$ that $d(S,i)\ge d$ and also $|S\cap N_d(i)|\le k$.
Define 
\begin{equation*}
        I_d(k) = \sup_{d(S,i)\ge d, |S\cap N_d(i)|\le k}\IM{S}{i}\,.
\end{equation*}
We will show that $I_d(k)\le k\B(\frac{2Cq}{1-2q\Delta}\B)^d$ by induction on $d$.

First, note that $P$ is $q$-marginally bounded, by definition of the influence, $\I{\sim i}{i}\le q$. So $I_1(k)\le q\le \frac{2kCq}{1-2q\Delta}$ for arbitrary $k$.

Now assume $I_{d-1}(k)\le k\B(\frac{2Cq}{1-2q\Delta}\B)^{d-1}$. Suppose the supremum in $I_d(k)$ is given by set $S$ and two marginal conditions $x_S^a$ and $x_S^b$, so that
\[I_d(k) = \TV(P_{x_i|x_S}(\ \cdot \ |x_S^a), P_{x_i|x_S}(\ \cdot \ |x_S^b))\,.\]
Let $X^a\sim P_{x_{N_1(S)}|x_S}(\ \cdot \ |x_S^a)$ and $X^b\sim P_{x_{N_1(S)}|x_S}(\ \cdot \ |x_S^a)$ be independent of each other. Note that $G$ is the graphical model of $P$, so $x_i$ and $x_S$ are independent of each other conditioned on $x_{N_1(S)}$. 
This means 
\[P_{x_i|x_S}(\ \cdot \ |x_S^a) = \E \b[ P_{x_i|x_S,x_{N_1(S)}}(\ \cdot \ |x_S^a,X^a)\b] = \E \b[P_{x_i|x_{N_1(S)}}(\ \cdot\ |X^a)\b]\,.\]
Similarly, $P_{x_i|x_S}(\ \cdot \ |x_S^b)= \E \b[P_{x_i|x_{N_1(S)}}(\ \cdot\ |X^b)\b]$.
Then 
by convexity of total variation (Lemma~\ref{lem:coupling_TV}), we have
\[\TV(P_{x_i|x_S}(\ \cdot \ |x_S^a), P_{x_i|x_S}(\ \cdot \ |x_S^b)) \le \E \b[ \TV\b(P_{x_i|x_{N_1(S)}}(\ \cdot\ |X^a), P_{x_i|x_{N_1(S)}}(\ \cdot\ |X^b)\b) \b]\,.\]
By the induction assumption,
\begin{align*}
    &\E \b[ \TV\b(P_{x_i|x_{N_1(S)}}(\ \cdot\ |X^a), P_{x_i|x_{N_1(S)}}(\ \cdot\ |X^b)\b) \b]\\
    &\le \sum_{j\ge 1}\Pr_{X^a,X^b}\B(\b\|X^a_{N_{d-1}(i)}  - X^b_{N_{d-1}(i)}\b\|_1 = j\B)I_{d-1}(j)\\
    &\quad + \sum_{j\ge 1}\Pr_{X^a,X^b}\B(X^a_{N_{d-1}(i)} = X^b_{N_{d-1}(i)}\,,\, \b\|X^a_{N_{d}(i)} - X^b_{N_{d}(i)}\b\|_1=j\B) I_{d}(j)\\
    &\le \sum_{j\ge 1} \Pr_{X^a,X^b}\B(\b\|X^a_{N_{d-1}(i)}\b\|_1  + \b\|X^b_{N_{d-1}(i)}\b\|_1 = j\B)\cdot j\cdot \B(\frac{2Cq}{1-2q\Delta}\B)^{d-1} \\
    &\quad+ \sum_{j\ge 1}\Pr_{X^a,X^b}\B( \b\|X^a_{N_{d}(i)}\b\|_1+ \b\|X^b_{N_{d}(i)}\b\|_1=j\B) I_{d}(j)\,.
\end{align*}
Because $d(S,i)=d$, $|N_1(S)\cap N_{d-1}(i)|\le |S\cap N_{d}(i)|C\le kC$, and $P$ is $q$-marginally small, the distribution of $\b\|X^a_{N_{d-1}(i)}\b\|_1  + \b\|X^b_{N_{d-1}(i)}\b\|_1$ is stochastically dominated by $\bino(2kC,q)$. Similarly, $|N_1(S)|\le k\Delta$, so $\b\|X^a_{N_{d}(i)}\b\|_1+ \b\|X^b_{N_{d}(i)}\b\|_1\le \|X^a\|_1+\|X^b\|_1$ is stochastically dominated by $\bino(2k\Delta,q)$. We have
\begin{align*}
    I_{d}(k)&\le \B(\frac{2Cq}{1-2q\Delta}\B)^{d-1} \sum_{j=0}^{2ka} \bino(2kC,q,j)\cdot j + \sum_{j=0}^{2k\Delta}\bino(2k\Delta,q,j)I_{d}(j)\\
    &= 2kCq\B(\frac{2Cq}{1-2q\Delta}\B)^{d-1}+ \sum_{j=0}^{2k\Delta}\bino(2k\Delta,q,j)I_{d}(j)\,.
\end{align*}
For simplicity of notation, let $\Tilde{I}_d(k) = \frac{(1-2q\Delta)^{d-1}}{(2Cq)^d}I_d(k)$. We get a recursion \[\Tilde{I}_d(k)\le k+ \sum_{j=0}^{2k\Delta}\bino(2k\Delta,q,j)\tilde{I}_{d}(j)\,.\]

We compare this recursion with percolation on a tree with branching factor $2\Delta$ and branching probability $q$. Let $m_k$ be the total number of nodes in $k$ such independent trees. We have
\[\E m_k = k+ \sum_{j=0}^{2k\Delta}\bino(2k\Delta,q,j)\E m_j\,.\]
By comparing both recursions, we can couple the recursion of $\Tilde{I}_d(k)$ with the branching process and conclude that $\Tilde{I}_d(k)\le \E m_k$. By linearity of expectation, $ \E m_k= \frac{k}{1-2q\Delta}$. So from the definition of $\Tilde{I}_d(k)$, we have
\[I_d(k) \le k\B(\frac{2Cq}{1-2q\Delta}\B)^d.\qedhere\]
\end{proof}

\begin{restatable}{corollary}{influencemuG}\label{cor:influence_muG}
Let $\mu_G$ be the distribution in Definition~\ref{def:triangle_dist}.
For any two triangles $t$ and $t'$ in graph $G$ such that $d(t,t')\ge k$, 
\[\I{t}{t'}=O(p'^k)\,.\]
\end{restatable}
\begin{proof}
We only need to verify that $\mu_G$ satisfies all the properties in Theorem~\ref{thm:small_marginal_low_influence}
In the graphical model of $\mu_G$, any two triangles are connected if and only if they share an edge. So the maximum degree of the graphical model of $\mu_G$ is at most $3(n-2)$. From Lemma~\ref{lem:mug_marginally_small}, $\mu_G$ is $O(p')$-marginally small. Because $p'\ll 1/n$, we have that $p'<1/{2n}$.

By induction one can verify that two triangles have distance $k$ if and only if they share $3-k$ common vertices. So for two triangles $t$, $t'$ that $d(t,t')=2$, $t$ and $t'$ share 1 vertex. $N_{1}(t)\cap N_1(t')$ are the vertices that share 2 vertices with both $t$ and $t'$. Therefore, $N_{1}(t)\cap N_1(t')=2$.

Suppose two triangles $t$, $t'$ satisfy $d(t,t')=3$, i.e., $t$ and $t'$ do not share vertex. $N_{2}(t)\cap N_1(t')$ are the triangles that share 1 vertex with $t$ and 2 vertices with $t'$. There are 3 of them in total.

We have \[\sup_k\sum_{t,t':d(t,t')=k}|N_{k-1}(t)\cap N_1(t')|\le 3\,.\qedhere\]
\end{proof}

\subsection{Influences in $\law_G(Y)$}

In this section, we will apply the conclusions from the previous section about $\mu_G$ to bound the influence of  $\law_G(Y)$.
We start with a technical lemma on the marginals of $\mu_G$ and $\law_G(Y)$.

\begin{lemma}\label{lem:margin-under-conditioning}
Suppose $G$ is a $c$-uniformly 2-star dense graph. Let $A\subset E(G)$ be an edge set that $|A|\le cn/4$ and $T\subset T_G$ be a triangle set that $|T|\le cn/4$. Let $y_A\in \{0,1\}^A$, $x_T\in \{0,1\}^T$ be arbitrary consistent configurations on $A$ and $T$. For any edge $e\notin A$,
\[
\mu_G(Y_e=1|Y_A=y_A,X_T=x_T)=\Omega(np')\,,
\]
If further $T_e\cap T = \emptyset$ or $x_{T_e\cap T} = \vec{0}$, then 
\[
\mu_G(Y_e=1|Y_A=y_A,X_T=x_T)=\Theta(np')\,.
\]
For any triangle $t\notin T$ such that  $(y_A)_{E^*}=\vec{1}$ where $E^*=E(t)\cap A$, we have
\[
\mu_G(X_t=1|Y_A=y_A,X_T=x_T)=O\B(\frac{1}{n^3p'^2}\B)
\,.
\]
\end{lemma}
\begin{proof}
For the first statement, consider all triangles in $G$ that contains $e$, denoted by $T_e$. If $T_e\cap T =\ne \emptyset$ and $x_{T_e\cap T} \ne \vec{0}$, then one of the triangles in $T$ that contain $e$ is already in $x_T$, which means
\[
\mu_G(Y_e=1|Y_A=y_A,X_T=x_T)=1\,.
\]
Now we only consider the case when $T_e\cap T = \emptyset$ or $x_{T_e\cap T} = \vec{0}$. 
As $G$ is $c$-uniformly 2-star dense, there are at least $c(n-2)$ such triangles. Let $T^*\subset T_e$ be the set of triangles that does not contain any edge from $A$ or any triangles in $T$. We have $|T^*|\ge c(n-2)-cn/4-cn/4=\Theta(n)$. Suppose $T^*=\{t_1,\cdots,t_k\}$, $k=|T^*|$. We have
\begin{align*}
    \mu_G(Y_e=1|Y_A=y_A,X_T=x_T) &=  \sum_{i=1}^k\mu_G(X_{t_i}=1|Y_A=y_A,X_T=x_T,X_{t_j}=0,\forall j<i)\\
    &=\Theta(kp')=\Theta(np')\,.
\end{align*}
The first equality is because $\mu_G$ is $\Omega(p')$-marginally large, and the conditioning is an event on triangles other than $X_{t_i}$.

Now let's look at the second statement.
\[
\mu_G(X_t=1|Y_A=y_A,X_T=x_T) = \frac{\mu_G(X_t=1| Y_{A\backslash E^*}=(y_A)_{A\backslash E^*} ,X_T=x_T)}{\mu_G(Y_{E^*}=\vec{1} |Y_{A\backslash E^*}=(y_A)_{A\backslash E^*} ,X_T=x_T)}
\,.
\]
Let $E^*=\{e_1,\cdots, e_k\}$, $k\le 3$.
From the first statement, the denominator can be bounded as follows,
\begin{align*}
& \mu_G(Y_{E^*}=\vec{1} |Y_{A\backslash E^*}=(y_A)_{A\backslash E^*} ,X_T=x_T)\\
&=  \prod_{i=1}^k  \mu_G(Y_{e}=1 |Y_{A\backslash E^*}=(y_A)_{A\backslash E^*} ,X_T=x_T,Y_j=1,\forall j<k)\\
&=  \Omega\b((np')^k\b)=\Omega(n^3p'^3)\,.
\end{align*}
Since $\mu_G$ is $O(p')$-marginally small, the nominator is bounded by $O(p')$. We have 
\[\mu_G(X_t=1|Y_A=y_A,X_T=x_T)=O(1/(n^3p'^2))\,.\qedhere\]
\end{proof}

\lowInfLemma*

\begin{proof}
    The proof idea is to consider edge variable $Y$ as function of the triangle variable $X$, and use Lemma~\ref{lem:influence_partition} to bound influence between one pair edges as the sum of influences between many pairs of triangles.

Let $T_A$ be the set of triangles that contain edges in $A$, $T_e$ be the set of triangles that contain edges in $e$. Let $T_o=T_A\cap T_e$, Since any triangle in $T_o$ have at least one edge in $A$, $|T_o|\le |A|$. 


Suppose $X\sim\mu_G$, $Y$ be the corresponding edge distribution,  $\law(Y)= \mug{G}$. Suppose $\IM{A}{e}$ is given by conditioning $y_A^a$ and $y_A^b$, i.e.,
\[\IM{A}{e} =\TV \b(\mug{G}(Y_e|Y_A=y_A^a), \mug{G}(Y_e|Y_A=y_A^b ) \b)\,.\]
From Lemma~\ref{lem:bad_event},
\begin{align*}
    \IM{A}{e}&\le  2\b(\mu_{G}(X_{T_o}\not= \vec{0}|Y_A=y_A^a)+ \mu_{G}(X_{T_o}\not= \vec{0}|Y_A=y_A^b)\b)\\
    &\qquad + \TV \b( \mug{G}(Y_e|Y_A=y_A^a,  X_{T_o}= \vec{0}), \mug{G}(Y_e|Y_A=y_A^b ,  X_{T_o}= \vec{0}) \b)\,.
\end{align*}
From Lemma~\ref{lem:margin-under-conditioning}, we can bound the first two terms.
\[\mu_{G}(X_{T_o}\not= \vec{0}|Y_A=y_A^a)\le \sum_{t\in T_o}\mu_G(X_t=1|Y_A=y_A^a)=O(|A|/(n^3p'^2))\,,\]
and similarly, $\mu_{G}(X_{T_o}\not= \vec{0}|Y_A=y_A^b)=O(|A|/(n^3p'^2))$.

Let $X^a\sim X|Y_A=y_A^a,X_{T_o}= \vec{0}$, $X^b\sim X|Y_A=y_A^b,X_{T_o}= \vec{0}$ be independent, $T_C=\{t|t\in T_A,X^a_t=1\ \text{or}\  X^b_t=1\}$ we can use Lemma~\ref{lem:coupling_TV}, 
\begin{align*}
    \IM{A}{e}&\le \E_{X_{T_A}^a, X_{T_A}^b}\TV(\mug{G}(Y_e|X_{T_A}^a,  X_{T_o}= \vec{0}), \mug{G}(Y_e|X_{T_A}^b,  X_{T_o}= \vec{0}))+O(|A|/(n^3p'^2))\\
    &\le\E_{X_{T_A}^a, X_{T_A}^b} \TV(\mu_{G}(X_{T_e}|X_{T_A}^a,  X_{T_o}= \vec{0}), \mug{G}(X_{T_e}|X_{T_A}^b,  X_{T_o}= \vec{0}))+O(|A|/(n^3p'^2))\\
    &\le \E_{X_{T_A}^a, X_{T_A}^b}\I{\{t|X^a_t\not= X^b_t\}}{T_e}+O(|A|/(n^3p'^2))\\
    & \le \E_{X_{T_A}^a, X_{T_A}^b} \I{T_C}{T_e}+O(|A|/(n^3p'^2)) \,.
\end{align*}
By Lemma~\ref{lem:influence_partition}, 
$\I{T_C}{T_e}\le \sum_{t\in T_e}\I{T_C}{t}$.

Now we need to bound $\I{T_C}{t}$ for $t\in T_e$.
From Corollary~\ref{cor:influence_muG}, if $T_C$ and $t$ are adjacent then $\I{T_C}{t}= O(p')$, otherwise, $\I{T_C}{t}=O( |T_C|p'^2)$. So 
\[\sum_{t\in T_e}\I{T_C}{t}=O( |N_1(T_C)\cap T_e|p' + |T_e||T_C|p'^2)\,.\]
By the definition of $X^a,X^b$ and $T_C$, $T_C\subset T_A\backslash T_e$. So for any $t\in T_C$, $t$ share at most 1 vertex with $e$.  There are at most 2 triangles that contains edge $e$ and is adjacent to $t$. Therefore, $|N_1(T_C)\cap T_e|\le 2|T_C|$.
We have\[\sum_{t\in T_e}\I{T_C}{t}=O(|T_C|p') \,.\]
Here we used that $|T_e|\le n-2$, $np'=o(1)$. 

Now we have $\IM{A}{e}=O\B(p'\E_{X_{T_A}^a, X_{T_A}^b}|T_C|+|A|/(n^3p'^2)\B)$, and it remains to bound $\E_{X_{T_A}^a, X_{T_A}^b}|T_C|$. By linearity of expectation, 
\[\E_{X_{T_A}^a, X_{T_A}^b}|T_C|\le \sum_{t\in T_A}\b(\mu_G(X_t^a=1|Y_A=y_A^a,X_{T_o}=\vec{0})+\mu_G(X_t^b=1|Y_A=y_A^a,X_{T_o}=\vec{0})\b)\,.\]
By Lemma~\ref{lem:margin-under-conditioning}, each term on the right hand side is $O(1/n^3p'^2)$.
So $$\E_{X_{T_A}^a, X_{T_A}^b}|T_C|=O\B(\frac{|T_A|}{n^3p'^2}\B)=O\B(\frac{|A|}{n^2p'^2}\B)$$ and hence $\IM{A}{e}=O(|A|/(n^3p'^2))$.

 Let $T_e$ and $T_{e'}$ be the set of triangles that contains $e$ and $e'$ respectively. Note that $e$ and $e'$ do not share common nodes, so $T_e$ and $T_{e'}$ are disjoint. Similar to the previous argument, suppose 
 \[\IM{e'}{e} = \TV \b( \mug{G}(Y_e|Y_{e'}=1), \mug{G}(Y_e|Y_{e'}=0 ) \b)\,.\]

Let $X^a\sim X|Y_{e'}=1$, $X^b\sim X|Y_{e'}=0$ be independent, $T_C=\{t|t\in T_{e'},X^a_t=1\}$, we have
\begin{align*}
    \IM{e'}{e}& \le \E_{X^a} \I{T_C}{T_e} \,.
\end{align*}
Again by Lemma~\ref{lem:influence_partition},
$\I{T_C}{T_e}\le \sum_{t_c\in T_C, t_e\in T_e}\I{t_c}{t_e}$. Now let's bound $\I{t_c}{t_e}$ for possible pairs of $t_e$ and $t_c$. From Corollary~\ref{cor:influence_muG}, this is bounded by $p'^{d(t_c,t_e)}$.

Let $T_o$ be the set of triangles
\[T_{o}=\{t:t\in T_{e'},d(t,T_e)=1\}\,.\]
This is equivalent to triangles that contains $e'$ and share a node with $e$. There are at most 2 such triangles. For $t_c\in T_o$, there is at most 1 triangle in $T_e$ with distance 1, so
\[\sum_{t_e\in T_e} \I{t_c}{t_e}\le O(p'+np'^2)=O(p')\,.\]
For $t_c\in T_{e'}\backslash T_o$, there is at most 1 triangle in $T_e$ with distance 2, so
\[\sum_{t_e\in T_e} \I{t_c}{t_e}\le O(p'^2+np'^3)=O(p'^2)\,.\]
Combining the two cases, we get
\[\I{T_C}{T_e}= O\b(\ind{T_C\cap T_o\not=\emptyset}p' + |T_C|p'^2\b)\,.\]

Now we have $\IM{e'}{e}=O\b(p'^2\E |T_C|+p'\Pr(T_C\cap T_o\not=\emptyset)\b)$, where the randomness is from $X^a$. From Lemma~\ref{lem:margin-under-conditioning}, we have
\[
\E|T_C| = \sum_{t\in T_{e'}}\mu_G (X_t=1|Y_{e'}=1)=O(1/(n^2p'^2))\,,
\]
and
\[
\Pr(T_C\cap T_o\not=\emptyset)\le \sum_{t\in T_o}\mu_G (X_t=1|Y_{e'}=1)=O(1/(n^3p'^2))\,.
\]
Combining the above equations, we have $\IM{e'}{e}=O(1/(n^3p'))$.
\end{proof}

\influencepe*
\begin{proof}
If $e'\in G$, this is trivial. So assume $e'\not\in G$.
Because
\[\mug{G+e'}(Y_e=1)  = \mu_{G+e'}(Y_{e'}=1) \mug{G+e'}(Y_e=1|Y_{e'}=1) + \mu_{G+e'}(Y_{e'}=0) \mug{G+e'}(Y_e=1|Y_{e'}=0)\,,\]
and $\mug{G}(Y_e=1) = \mug{G+e'}(Y_e=1|Y_{e'}=0)$,
we have 
\[\left|
\mug{G}(Y_e=1)-\mug{G+e'}(Y_e=1)\right| \le |\mug{G+e'}(Y_e=1|Y_{e'}=1)-\mug{G+e'}(Y_e=1|Y_{e'}=0)|\le \IM{e'}{e}\,.\qedhere\]
\end{proof}

\section{Concentration and Mixing for Marginally Small Distributions}\label{sec:small-marginal-concentration}

In this section we continue our study of marginally small distributions, and show concentration of Lipschitz functions and fast mixing of Glauber dynamics.

\subsection{Glauber Dynamics and Path Coupling}

\label{sec:glauber}
Glauber Dynamics is a simple and commonly used discrete-time Markov chain to sample from distributions with local structures. Consider a distribution $\mu$ over $\Omega = \cX^{[N]}$. Starting with arbitrary initial configuration $X\in \Omega$, at each time step, the Glauber dynamics sample a uniformly random location $i\in [N]$ and resamples $X_i$ according to the conditional distribution given other locations, $\mu(x_i|x_{\sim i} = X_{\sim i})$. Let $\gl(X)$ denote the variable generated by applying one step of Glauber dynamics on $X$.

When applying Glauber dynamics, the crucial property we care about is the mixing time, i.e., the time it takes for the dynamics to converge to the stationary distribution. There are many approaches to prove the mixing speed of Glauber dynamics, in this paper we will be using the path coupling technique introduced in \cite{bubley1997path}. 
\begin{definition}[Weighted Hamming Distance]\label{def:path-distance}
A path distance over $\Omega$ is a weighted hamming distance where 
$d(X,Y)=\sum_{i:X_i\not=Y_i} c_i$.
\end{definition}

\begin{lemma}[Path Coupling, \cite{bubley1997path}]\label{lem:path-coupling}
Let $d(X,Y)$ denote the path distance between a pair of configurations $X,Y\in \Omega$. If for any pair of $X,Y$ that differs at one coordinate, there is a coupling $(\gl(X),\gl(Y))$ such that 
\[\E d(\gl(X),\gl(Y)) \le (1-\alpha)d(X,Y)\,,\]
then the coupling can be extended to any pair of configurations that same inequality holds.
\end{lemma}
We say the Glauber dynamics satisfies path coupling with constant $\alpha$ with respect to distance $d$ if the above holds. Path coupling imply mixing of Glauber dynamics.

\begin{lemma}\label{lem:path-glauber}
If the Glauber dynamics of $\mu$ satisfies path coupling with constant $\alpha$ with respect to distance $d$ defined in Definition~\ref{def:path-distance}, then for any starting configuration $X_0$, we have
\[\TV(\gl^t(X_0),\mu)\le \delta\]
when $t\ge \log_{1-\alpha}(\epsilon\min_i c_i/(\sum_ic_i))$.
\end{lemma}

\begin{proof}
Now imagine $X_0$ is the initial configuration chosen by the algorithm and $Y_0\sim \mu$. Let $X_t = \gl^t(X_0)$ and $Y_t = \gl^t(Y_0)$ denote the random variables we get after applying $t$ steps of Glauber dynamics on $X_0$ and $Y_0$.  Using the Lemma above for $t$ times, we get after $t$ time steps, 
\[\E d(X_t,Y_t)\le (1-\alpha)^t\sum_ic_i\,.\]
Since the distance $d(X,Y)$ is at least $\min_i c_i$ when $X\not=Y$, we get an upper bound on the total variation distance between the law of $X_t$ and $\mu$:
\[\TV(X_t,\mu)\le \Pr(X_t\not=Y_t)\le \frac{\E  d(X_t,Y_t)}{\min_i c_i}\le (1-\alpha)^t\frac{\sum_ic_i}{\min_i c_i}\,.\]
It follows that to generate a sample that is $\epsilon$ close to $\mu$ in total variation, we can start with arbitrary initial configuration and run Glauber dynamics for $\log_{1-\alpha}(\epsilon\min_i c_i/(\sum_ic_i))$ time-steps.
\end{proof}

\subsection{Concentration from Mixing}\label{sec:conc-mix}
In this section we show that if a function can be chosen to represent the distance in path coupling, the function must concentrates.

\begin{lemma}[Concentration From Path Coupling]\label{lem:pathcoupling_concentration}
Let $\mu$ be a distribution over $\Omega=\{0,1\}^N$. Let $\gl$ be the Glauber dynamics of $\mu$. Suppose function  $f:\{0,1\}^N\rightarrow \mathbb{R}$ satisfies 
\[\sup_{x_1,\cdot,x_N,x_i'}|f(x_1,\cdots, x_i,\cdots,x_N)-f(x_1,\cdots, x_i',\cdots,x_N)|\le L_i\]
Let $D$ be the path distance on $\Omega$ defined by 
\[D(x,x') = \sum_{i:x_i\not=x_i'}L_i\,.\]
Suppose that $D$ satisfies path coupling with constant $\alpha$, i.e., suppose that for $x,x'$ that differ at precisely one coordinate, there exists a coupling between $\gl (x)$ and $\gl (x')$ such that\[\E D(\gl (x), \gl (x'))\le (1-\alpha)D(x,x')\,.\]
Then for any $\delta>0$, $t>0$,
\[\mu(f-\E_\mu(f)>t)\le \exp \B( \frac{-t^2/2}{4T\|L\|_2^2/N+\|L\|_{\infty}t/3} \B)+\delta\,,\]
where $T = \log_{1-\alpha}\frac{\delta\inf_i L_i}{\|L\|_1}$ is the mixing time of $\gl$. Further, if $\mu $ is $q$-marginally small, 
\[\mu(f-\E_\mu(f)>t)\le \exp \B( \frac{-t^2/2}{8Tq\|L\|_2^2/N+\|L\|_{\infty}t/3} \B)+\delta\,.\]
\end{lemma}
\begin{proof}
Consider running Glauber dynamics with initial condition  $X_0=\vec{0}$. Let $X_t=\gl^t(X_0)$, where $t=0,1,\cdots, T$.
Define a sequence of martingale 
\[g_t=\E [f(X_T)|X_t]\,.\]
The martingale difference  $Y_t=g_t-g_{t-1}$ can be written as
\[Y_t = f\b(\gl ^{T-t}(X_t)\b)-\E_{X_t'|X_{t-1}}f\b(\gl^{T-t}(X_t')\b)\,.\]

From the definition of $D$, for any $x,x'\in \Omega$, $D(x,x')\ge |f(x)-f(x')|$, hence
\begin{align*}
    Y_t&\le \E_{X_t'|X_{t-1}} D\b( \gl ^{T-t}(X_t),  \gl ^{T-t}(X_t')\b)\\
    &\le  \E_{X_t'|X_{t-1}} D(X_t,X_t')\\
    &\le D(X_t,X_{t-1})+\E_{X_t'|X_{t-1}} D(X_t', X_{t-1}) \,.
\end{align*}
The second inequality is by the data processing inequality and the third inequality is by the triangle inequality.
Since $X_t$ and $X_{t-1}$ differ in at most one coordinate,
 $Y_t\le 2\|L\|_{\infty}$ almost surely. Also, by the upper bound above, 
\[\E[Y_t^2|X_{t-1}]\le \E\b[ D(X_t,X_{t-1})^2\b] +3\b(\E[ D(X_t,X_{t-1})|X_{t-1}]\b)^2\le 4\E[ D(X_t,X_{t-1})^2 |X_{t-1}]\,.\]
The Glauber dynamics chooses a location $i\in[N]$ uniformly at random, when $i$ is chosen, $D(X_t,X_{t-1})\le L_i$. So 
\[\E[Y_t^2|X_{t-1}]\le4\sum_iL_i^2/N=4 \|L\|_2^2/N\,.\]
When $\mu$ is $q$-marginally small, the probability that we change coordinate $i$ in the Glauber dynamics is at most
\[\frac{\Pr((X_{t-1})_i=0)q+\Pr((X_{t-1})_i=1)}{N}\,.\]
Note that $\Pr((X_{t-1})_i=1)\le q$ is maintained when we run Glauber dynamics. So the above quantity is bounded by $2q/N$.
We have 
\[\E [Y_t^2|X_{t-1}]\le 4\sum_i 2qL_i^2/N = 8q\|L\|_2^2/N\,.\]
Therefore, we can apply Mcdiarmid's inequality (Lemma~\ref{lem:mcdiarmid}), yielding
\[\Pr(g_T-\E g_T>t)\le \exp\B(\frac{-t^2/2}{8qT\|L\|_2^2/N+\|L\|_\infty t/3}\B)\,.\]
By a standard path coupling argument, Lemma~\ref{lem:path-glauber},
$\TV(\law(X_T), \mu)\le \delta$. Since $g_T=f(X_T)$ by definition, we have
\[\mu(f-\mu(f)>t)\le \exp\B(\frac{-t^2/2}{8qT\|L\|_2^2/N+\|L\|_\infty t/3}\B)+\delta\,.\qedhere\]
\end{proof}

It turns out that we can apply path coupling to any $q$-marginally small distribution as long as $q$ is much less than the inverse degree of its graphical model.

\begin{theorem}\label{thm:small_marginal_concentration}
Let $P$ be a distribution over $\{0,1\}^N$ and let $G$ be the graphical model of $P$ with maximum degree $\Delta$. If $P$ is $q $-marginally small for $q\ll 1/\Delta$, then:
\begin{enumerate}
    \item If $X\sim P$, then
for any $f$ that is $L$-Lipschitz on $\{0,1\}^N$, $|f(X)-\E f(X)|= O (\sqrt{NqL^2}\log N)$ with high probability.
\item 
The Glauber dynamics for $P$ mixes in $O(N\log N)$ time.
\end{enumerate}
\end{theorem}


\begin{proof}
We can use Lemma~\ref{lem:pathcoupling_concentration} by defining $D(x,x')=L$ for adjacent configurations. 
In the following, we show that $D$ satisfies the path coupling condition with constant $\alpha=\Omega(1/N)$. 

Suppose $x$ and $x'$ differ only at position $i$, $x_i=1$, $x'_i=0$.
Consider the following coupling between $\gl (x)$ and $\gl (x')$. The Glauber dynamics chooses the same location in $x$ and $x'$, say $j\in [N]$. If $j=i$ or $j$ is not adjacent to $i$ in $G$, since $x_j|x_{\sim i}$ and $x_j|x'_{\sim i}$ have the same distribution, we can couple them by the same choice of $x_j$. If $j$ is a neighbor of $i$, we couple them arbitrarily. 

When the Glauber dynamics chooses $i$, the new distance is 0. When Glauber dynamics chooses a neighbor $j$ of $i$, the new distance is either $L$ or $2L$ depending on whether $\gl (x)$ and $\gl (x')$ chooses the same value on $j$. Since the distribution is $q$-small, they both choose $x_j=0$ with at least $1-2q$ probability, from the union bound. When Glauber dynamics chooses any other coordinate, the new distance is still $L$. Therefore, under this coupling,
\begin{align*}
\E D(\gl (x),\gl (x'))&\le \frac{N-\Delta-1}{N}\cdot L+ \frac{\Delta}{N}\cdot(1-2q)\cdot L +\frac{\Delta}{N}\cdot(2q)\cdot (2L)\\
&\le \B(1-\frac{1-6q\Delta}{N}\B)L \\
&= \B(1-\frac{1-6q\Delta}{N}\B)D(x,x')\,.
\end{align*}
Hence, $D$ satisfies the contraction required by path coupling with constant $(1-6q\Delta)/{N}=\Omega(1/N)$. 

Now we can apply Lemma~\ref{lem:pathcoupling_concentration}. For any constant $c$ and $\delta=N^{-c}$, the mixing time 
$T$ is $\log_{(1-6q\Delta)/N} (N^{-c-1})=O(N\log N)$. Choosing $t = \sqrt{cqTL^2\log N}=O(\sqrt{NqL^2\log ^2 N})$, we have 
\[\Pr(f(X)-\E f(X)>t)\le \exp\B(\frac{cqTL^2\log N/2}{8qTL^2+L\sqrt{cqTL^2\log N}/3}\B)+\delta \le 2N^{-c/32}\,.\]
By considering $-f(X)$, the same bound applies to $\Pr(f(X)-\E f(X)<-t)$.
\end{proof}

We have the following corollary as an immediate result, noting that $\mu_G$ is $O(p')-$marginally small under arbitrary conditioning.
\concentrationmuG*

A similar result can be obtained for functions that are not everywhere Lipschitz, but are when restricted to a high probability set.
\begin{corollary}\label{cor:small_marginal_concentration_extension}
Under the assumptions of Theorem~\ref{thm:small_marginal_concentration}, if there exists a $\cY\subset \{0,1\}^N$ such that $f$ restricted to $\cY$
 is $L$-Lipschitz, $X$ is in $\cY$ with high probability,
 and $1/q$ is polynomially bounded by $N$, we still have $|f(X)-\E[f(X)]|= O(\sqrt{NqL^2}\log N) $ with high probability.
\end{corollary}
\begin{proof}
By assumption,
 $\Pr(X\in \cY)\ge 1-1/N^c$ for some sufficiently large constant $c$. 
We take $$\bar f(x) = \inf_{y\in \cY}f(y)+d_H(x,y)\cdot L\,,$$
where $d_H(x,y)=|\{i:x_i\neq y_i\}|$ is the Hamming distance.
Note that (i) if $x\in \cY$, then $f(x) = \bar f(x)$, and (ii)
 $\bar f$ is everywhere $L$-Lipschitz. 
Thus, we can apply Theorem~\ref{thm:small_marginal_concentration} to $\bar f$. Since $\E f(X)$ and $\E \bar f(X)$ differ by at most 
$\sup_x |f(x)-\E f(X)|/N^c\le L/N^{c-1}$, we have that when $X\in \cY$ and $|\bar f(X)-\E \bar f(X)|=O(\sqrt{NqL^2}\log N)$, which happens with high probability, 
\begin{align*}
    |f(X)-\E f(X)|= |\bar f(X)-\E f(X)|\le |\bar f(X)-\E \bar f(X)|+L/N^{c-1} = O(\sqrt{NqL^2}\log N)+ L/N^{c-1} \,.
\end{align*}
This last quantity is $O(\sqrt{NqL^2}\log N)$ whenever $c$ is a large enough constant.
\end{proof}

\glaubermix*
\begin{proof}
The graphical model of $\mu_G$ has $\binom{n}{3}$ variables. Two triangles are connected if and only if they share an edge, so each variable has degree at most $3(n-2)$. Now, $\mu_G$ is $O(p')$-marginally small by Lemma~\ref{lem:mug_marginally_small}, so
Theorem~\ref{thm:small_marginal_concentration} implies that the Glauber dynamics mixes in $O(n^3\log n)$ time.
\end{proof}

\bibliography{ref}
\bibliographystyle{alpha}

\appendix

\section{Relation Between Models}
We will analyze the parameter regimes where each of the three models, {\ER}, RIG and RGT are close to each other.
The RGT can be thought of as partway between RIG and {\ER}.

\subsection{RGT vs \ER}

RGT is close to {\ER} with equal edge density when $p' \ll n^{-3/2}$.

\begin{theorem}\label{thm:er-rgt}
If $p$ is a positive constant, then
\begin{itemize}
    \item When $p'\gg n^{-3/2}$, $$d_{TV}\left(\rgt(n,p,p'),G\left(n,p+(1-p)\left(1-(1-p')^{n-2}\right)\right)\right)=1-o_n(1)$$
    and there is an efficient algorithm to distinguish the two distributions with vanishing error.
    \item When $p'\ll n^{-3/2}$, $$d_{TV}\left(\rgt(n,p,p'),G\left(n,p+(1-p)\left(1-(1-p')^{n-2}\right)\right)\right)=o_n(1)\,.$$
\end{itemize}
\end{theorem}
The lower bound on total variation is established through statistical testing using a signed version of triangle count. This test has also been employed to distinguish between \ER\ and random intersection graphs \cite{brennan2020phase}, as well as between \ER\ and random geometric graphs \cite{bubeck2016testing}.

The ordinary triangle count with adjacency matrix $A$ is given by $$\sum_{\{i,j,k\}\subset[n]}A_{ij}A_{jk}A_{ik}\,.$$ Let $q=p+(1-p)\left(1-(1-p')^{n-2}\right)=1-(1-p)(1-p')^{n-2}$ be the edge density. The \emph{signed triangle count} is given by $$\tau(A)=\sum_{\{i,j,k\}\subset[n]}(A_{ij}-q)(A_{jk}-q)(A_{ik}-q)$$ It is not hard to show that for {\ER}, 
\begin{lemma}\label{lem:er-signed-triangle}
For $A$ being the adjacency matrix of an \ER random graph $G(n,q)$,
\begin{align*}
    \E [\tau(A)]=0\quadand
    \var[\tau(A)]=\binom{n}{3}q^3(1-q)^3\,,
\end{align*}
\end{lemma}
The calculation is straightforward and can be found in Section 3 of \cite{bubeck2016testing}.
For RGT, we will prove
\begin{lemma}\label{lem:rgt-signed-triangle}
For $\rgt(n,p,p')$, 
\begin{align*}
    \E [\tau(A)]=\binom{n}{3}p'(1-q)^3+O(1)\quadand
    \var[\tau(A)]=O(n^3)\,.
\end{align*}
\end{lemma}


\begin{proof}
Let $\tau_{ijk} = (A_{ij}-q)(A_{jk}-q)(A_{ik}-q)$. Let $T$ be the set of triangles add in the construction of $\rgt(n,p,p')$. Each triple is in $T$ independently with probability $p'$. Let $X_{ijk}$ to be the indicator of whether the triple is included in $T$. Note that Given  $X_{ijk}$, edges $(i,j)$, $(j,k)$ and $(k,i)$ become independent since all other triangles do not contain two of them. We have
\begin{align*} 
    \E \tau_{ijk} &= \Pr(X_{ijk}=1)(1-q)^3 + \Pr(X_{ijk}=0)\b(\E[A_{ij}-q|X_{ijk}=0]\b)^3\\
    &= p'(1-q)^3+(1-p')\b(p+(1-p)(1-(1-p')^{n-3})\b)^3\\
    &=p'\b[(1-q)^3-p'^2(1-p)^3(1-p')^{3n-8}\b]
\end{align*}
Since $p'\le 1/n$, we have $\E \tau_{ijk} = p'(1-q)^3+O(1/n^3)$. So $\E \tau(A) = \binom{n}{3}p'(1-q)^3+O(1)$.

Now let's bound the variance. One can attempt to calculate the exact value, but here we only upper bound it by $O(n^3)$ for simplicity. Note that if two triples $i,j,k$ and $i',j',k'$ do not overlap, $\tau_{ijk}$ and $\tau_{i',j',k'}$ are independent.
\begin{align*}
    \var[\tau(A)] =&\  \sum_{\{i,j,k\}\subset [n]}\sum_{\{i',j',k'\}\subset [n]}\cov[\tau_{ijk},\tau_{i'j'k'}]\\
    =& \ \binom{n}{3}\var[\tau_{123}] + \binom{n}{4}\binom{4}{2}\cov[\tau_{123},\tau_{124}] + \binom{n}{5}\binom{5}{2}\binom{3}{1}\cov[\tau_{123},\tau_{145}]
\end{align*}
The first term is trivially bounded by $O(n^3)$ since $\var[\tau_{123}]$ is at most constant. It remains to bound the second and third term.

Let $T_0$ be a set of triangles of constant size that contains $i,j$, we have
\begin{equation}\label{eq:aij}
    \E[(A_{ij}-q)^2| X_{T_0}=\Vec{0}]=O(1)\text{ and }\E[q-A_{ij}|X_{T_0}=\Vec{0}]=O(p')
\end{equation}
The first equation is trivial as all the variables are bounded by constant, the second equation is equal to $(1-p)(1-p')^{n-2-|T_0|}-(1-p)(1-p')^{n-2}=O(|T_0|p')=O(p')$.

For the second term, $\cov[\tau_{123},\tau_{124}]$ Let $E_1$ be the set of edges included in triangle $(1,2,3)$ or triangle $(1,2,4)$, let $T_1$ be the set of all triangles that contains at least 2 edges from $E_1$, $|T_1|=4$.
Now $\E[\tau_{123}\tau_{124}] = \E_{X_{T_1}} [\tau_{123}\tau_{124}|X_{T_1}]$. Notice that given $X_{T_1}$, all the edges becomes independent. So 
\begin{align*}
    \E[\tau_{123}\tau_{124}] =&\ \E_{X_{T_1}} [\tau_{123}\tau_{124}|X_{T_1}]\\
    =&\ \E_{X_{T_1}}\B[ \E\b[ \b(\ind{(1,2)\in E(X_{T_1})}(1-q)^2+ \ind{(1,2)\notin E(X_{T_1})}(A_{12}-q)^2\b)|X_{T_1}\b]\\ &\ \prod_{e\in E(X_{T_1})\backslash \{1,2\}}(1-q)\prod_{e\in E_1\backslash E(X_{T_1})\backslash \{1,2\}}\E[(A_e-q)|X_{T_1}] \B]\\
    =&\ O(1)\cdot \sum_{X_{T_1}}\Pr(X_{T_1}) \prod_{e\in E(X_{T_1})\backslash \{1,2\}}(1-q)\prod_{e\in E_1\backslash E(X_{T_1})\backslash \{1,2\}}\E[(A_e-q)|X_{T_1}]
\end{align*}
The third equality is because the term on edge $(1,2)$ is always bounded by $O(1)$.
When $|X_{T_1}|\ge 2$, $\Pr(X_{T_1})\le p'^2$, the term inside summation is $O(p')$. When $|X_{T_1}|\le 1$, there are at least two edges in $E_1\backslash E(X_{T_1})\backslash \{1,2\}$, so by \eqref{eq:aij}, the term inside summation is still bounded by $O(p')$. There are constant number of terms, so we have $\E[\tau_{123}\tau_{124}]=O(p'^2)$.

Similar analysis holds for $\E[\tau_{123}\tau_{145}]$. Let $E_2$ be the set of edges included in one of the triangles and  $T_2$ be the set of triangles that contains at least two edges in $E_2$. $|T_2|=6$.
\begin{align*}
    \E[\tau_{123}\tau_{145}] =&\ \E_{X_{T_2}} [\tau_{123}\tau_{145}|X_{T_2}]\\
    =&\ \E_{X_{T_2}}\B[ \prod_{e\in E(X_{T_2}) }(1-q) \prod_{e\in E_2 \backslash E(X_{T_2})} \E[(A_e-q)|X_{T_2}] \B]\\
    =&\ O(1)\cdot \sum_{X_{T_2}}\Pr(X_{T_2}) \prod_{e\in E(X_{T_2})}(1-q)\prod_{e\in E_2\backslash E(X_{T_2})}\E[(A_e-q)|X_{T_2}]
\end{align*}
When $|X_{T_2}|\ge 2$, $\Pr(X_{T_2})\le p'^2$, the term inside summation is $O(p')$. When $|X_{T_2}|\le 1$, there are at least three edges in $E_2\backslash E(X_{T_2})$, so by \eqref{eq:aij}, the term inside summation is bounded by $O(p'^3)$. There are constant number of terms, so we have $\E[\tau_{123}\tau_{145}]=O(p'^2)$.

Therefore, $\var[\tau(A)]=O(n^3+n^4p'^2+n^5p'^2)$. Since $p'=O(1/n)$, we have $\var[\tau(A)]=O(n^3)$.


\end{proof}
So we can construct a statistical test between
\[H_0\sim G(n,p)\quad \text{and}\quad H_1\sim G(n,p,p')\]
 as follows: Calculate $\tau(A)$, if it's larger than $\binom{n}{3}p'(1-q)^3/2$, output alternative hypothesis, otherwise output null hypothesis. By Lemma~\ref{lem:er-signed-triangle} and Lemma~\ref{lem:rgt-signed-triangle}, when $p'\gg n^{-3/2},$ the test success with probability $1-o_n(1)$ by Chebyshev Inequality. So the two problem have total variation distance  $1-o_n(1)$.

Now we prove the upper bound assuming $p'\ll n^{-3/2}$.
The main lemma is given in \cite{brennan2020phase} that bounds the deviation we have from {\ER} when one triangle is added at a uniformly random place.
\begin{lemma}[Lemma 3.1, \cite{brennan2020phase}]
Let $G(n,t,p)$ denote the graph generated by an {\ER} graph $G(n,p)$ with a clique of size $t$ planted at a uniformly random place. When $t$ and $p$ are constants, 
$$\TV\B(G(n,t,p),G\b(n,p+(1-p)\textstyle{\binom{t}{2}\binom{n}{2}}^{-1}\b)\B)=O_n(n^{-3/2})\,.$$
\end{lemma}
We can use data processing inequality to show that adding $i$ cliques instead of one leads to a TV distance of $O_n(in^{-3/2})$ when compared with {\ER} with same edge density. 
\begin{corollary}\label{cor:TV-add-clique}
Let $G(n,i,t,p)$ denote the graph generated by an {\ER} graph $G(n,p)$ with $i$ cliques of size $t$ planted at a random place. When $t$ and $p$ are constants, 
$$\TV\b(G(n,i,t,p),G(n,q)\b)=O_n(in^{-3/2})$$
where $q=p+(1-p)\b[1-\b(1-\binom{t}{2}\binom{n}{2}^{-1}\b)^{i}\b]$.
\end{corollary}

Intuitively, when we add $o(n^{3/2})$ number of triangles, the resulting graph is still close to \ER. This corresponds to adding each triangle with probability $o(n^{-3/2})$. Before formally proving this, let's introduce a technical lemma.
\begin{lemma}[Corollary 5.1 in \cite{kim2018total}]\label{lem:distance-binomial}
For a positive integer $N$ and $0<p<q<1$, 
\[\TV(\bino(N,p),\bino(N,q))\le \gamma+3\gamma^2\,.\]
Here $\gamma = (q-p)\sqrt{\frac{N}{p(1-p)}}$.
\end{lemma}

To simulate the process of adding each triangle independently with probability $p'$, we can add $M\sim \poi(\binom{n}{3}\log (\frac{1}{1-p'}))$ number of triangles independently at uniformly random location. By Poission splitting, each triangle is added $\poi (\log (\frac{1}{1-p'}))$ times, which is equivalent to adding it with probability $p'$. By Corollary~\ref{cor:TV-add-clique} and convexity of total variation,
\[\TV\left(\rgt(n,p,p'), \E_M G\left(n,q(M) \right) \right) = O(n^{-3/2}\E_M M) = O(n^{3/2}p')=o(1)\,,\]
where $q(M) = 1-(1-p)\B(1-(n-2)\binom{n}{3}^{-1}\B)^M  $ is the edge density after adding $M$ triangles to $G(n,p)$.

By Lemma~\ref{lem:distance-binomial} and convexity of total variation, 
\[\TV\b(G(n,q), \E_M G\left(n,q(M) \right) \b) \le \E_M[\gamma(M)]=O(n\E_M [|q-q_M|])\,,\]
where $$\gamma(M) = |q-q(M)|\sqrt{\frac{\binom{n}{2}}{\min \{q(1-q),q(M)(1-q(M))\}}}=O(n|q-q(M)|)\,.$$
Now
notice that $\binom{n}{3}\log (\frac{1}{1-p'})=O(n^3p')$. By Chebyshev's inequality, 
\[\E M-\sqrt{n^3p'\log n }\le M\le \E M +\sqrt{n^3p'\log n}\]
with probability $O(\log ^{-1}n)$. For any $M$ within the above range,
\begin{align*}
    |q-q(M)| =&\  \B|(1-p)(1-p')^{n-2} - (1-p)\B(1-(n-2)\binom{n}{3}^{-1}\B)^{M}\B|
\\=&\ (1-p)\B|(n-2)p'-(n-2)\log (\frac{1}{1-p'}) +O(n^2p'^2)+O(\sqrt{\log n n^{-1}p'})\B| = o(n^{-1})\,.
\end{align*}
Here we used that $(1-a)^b=1-ab+O(a^2b^2)$ when $ab\ll 1$.
So we have 
\[\TV\b(G(n,q), \E_M G\left(n,q(M) \right) \b) \le \E_M[\gamma(M)]=o(1)+\log ^{-1}n=o(1)\,.\]
Combining this with the fact that $\rgt(n,p,p')$ is close to $\E_M G\left(n,q(M) \right)$, we have 
\[\TV(G(n,q), \rgt(n,p,p')=o(1)\]
for $p'\ll n^{3/2}$.

\subsection{RGT vs RIG}

We now consider
random intersection graphs. Using Corollary~\ref{cor:TV-add-clique}, we can prove that RIG is close to RGT when $d\gg n^{2.5}$, and is close to {\ER} when $d\gg n^3$.

\rigrgt*

\begin{theorem} [\cite{brennan2020phase}, Theorem 3.1]
When $d\gg n^{3}$, we have $$d_{TV}(\rig(d,\delta),G(n,1-(1-\delta^2)^d))=o(1).$$
\end{theorem}

The idea is to provide another viewpoint of generating $\rig$. We define an $n \times d$ matrix $M$ such that $M_{ij} = 1$ if $S_i$ contains $j$, and $M_{ij} = 0$ otherwise. An edge $i,j$ is added to the graph if they appear on the same column.

Equivalently, we can add $d$ cliques defined by the columns of $M$. The size of each clique is determined by a multinomial distribution with size $k$ having probability $\binom{n}{k}\delta^k(1-\delta)^{n-k}$. It can be shown that this distribution is closely approximated by adding a Poisson number of cliques of each size $k$. Therefore, we can view RIG as adding $\poi(\Theta(dn^2\delta^2))$ edges (2-cliques), $\poi(\Theta(dn^3\delta^3))$ triangles (3-cliques), $\poi(\Theta(dn^4\delta^4))$ 4-cliques, and so on.

Formally, we can define the following equivalent generation of RIG. Let $p_j:=\Pr(|S_i|=j) = \binom{n}{j}\delta^j(1-\delta)^{n-j}$.
\begin{enumerate}
    \item Sample $(M_0,M_1,\cdots,M_n)\sim \multino(d,p_0)$.
    \item For all $2\le j\le n$, independently uniformly sample a random subset of size $j$ from $[n]$ $M_j$ times and plant a $j$-clique at those $M_j$ places.
\end{enumerate}
We can use a series of Poission distribution to simulate the distribution of $M_2,\cdots, M_n$. Define $\rigp(n,d,\delta)$ to be the model generated by the following process.
\begin{enumerate}
    \item Sample $X\sim \poi(d(1-p_0-p_1))$. \item Sample $M_2,\cdots, M_n\sim \multino(X,\gamma p_2,\cdots, 
    \gamma p_n)$ where $\gamma=(1-p_0-p_1)^{-1}$.
    \item For all $2\le j\le n$, independently uniformly sample a random subset of size $j$ from $[n]$ $M_j$ times and plant a $j$-clique at those $M_j$ places.
\end{enumerate}
Note that by Poission splitting, the process of generating $M_j$ is equivalent to independent sampling $M_j\sim \poi(dp_j)$ for $2\le j\le n$. It is proven in \cite{brennan2020phase} that $\rigp$ is close to $\rig$ in total variation distance.
\begin{lemma}[\cite{brennan2020phase}, Proposition 3.1]\label{lem:rig-rigp}
For $d\gg n^2$ and $\delta=\Theta(1/\sqrt{d})$, it holds that 
\[\TV(\rig(n,d,\delta), \rigp(n,d,\delta))=o_n(1)\,.\]
\end{lemma}

To compare it with $\rig$, we separate the process of adding 3-cliques and define $\rigpf(n,d,\delta)$ by the distribution of $\rigp$ conditioned on $M_3=0$.
Similar to the argument of Theorem 3.1 in \cite{brennan2020phase}, we can prove it is close to \ER\ random graph when $d\gg n^{2.5}$.
\begin{lemma}\label{lem:er-rigpf}
For $d\gg n^{2.5}$ and constant edge density, i.e., $\delta=\Theta(1/\sqrt{d})$, it holds that 
\[
\TV(\rigpf(n,d,\delta), G(n, 1-e^{-d\delta^2+(n-2)d\delta^3(1-\delta)^{n-3}})) = o_n(1)\,.
\]
\end{lemma}
The proof is deferred to Appendix~\ref{sec:postponed-proofs} as it is similar to the proof of Theorem~\ref{thm:er-rgt}.

Notice that to transform from $\rigpf$ to $\rigp$, we only need to sample $M_3\sim \poi(dp_3)$ and add a uniformly random triangle to $\rigpf$ independently for $M_3$ times. By Poission splitting, if we let $X_{ijk}$ be the number of times that triple $i,j,k$ is sampled, we have $X_{ijk}\sim \poi(\frac{dp_3}{\binom{n}{3}})=\poi(d\delta^3(1-\delta)^{n-3})$. This means each triple is added to $\rigpf$ independently with probability $1-e^{-d\delta^3(1-\delta)^{n-3}}$. Combine this observation with Lemma~\ref{lem:er-rigpf}, we immediately have the following corollary. 
\begin{corollary}\label{cor:rigp-rgt}
For $d\gg n^{2.5}$ and constant edge density, i.e., $\delta=\Theta(1/\sqrt{d})$, we have
\[\TV(\rigp(n,d,\delta), \rgt(n,1-e^{-d\delta^2+(n-2)d\delta^3(1-\delta)^{n-3}}, 1-e^{-d\delta^3(1-\delta)^{n-3}}))
\]
\end{corollary}
Combining this and Lemma~\ref{lem:rig-rigp}, we have proved Theorem~\ref{thm:rig-rgt}.

\section{Deferred Proofs}\label{sec:postponed-proofs}



\subsection{Proof of Lemma~\ref{lem:Kuvx-vs-x} on the distance between $K_{uv}(X)$ and $X$}
\begin{proof}
Recall that we  define a partition of $T_G$ as follows. 
\[T_{G0} \defeq \{t\in T_G:t\text{ do not contain $u$ or $v$}\}\]
\[T_{Gv} \defeq \{t\in T_G:t\text{ contains  $v$}\}\]
\[T_{Gu} \defeq \{t\in T_G:t\text{ contains  $u$}\}\]
Recall the following lemma.

\goodxtg*

The same holds by exchanging $u$ and $v$, so we have the the following by combining both conditions.
\begin{lemma}\label{lem:good-xtg0-prime}With high probability over the randomness of $X_{T_{G0}}$, for any $w\not=u,v$, there are $\Theta(n^2p')$ different $w'$ that $(w',v)\in G$ and $(w',w)\in E(X_{T_{G0}})$, and $\tilde{\Omega}(n)$ different $w'$ that $(w',u)\in G$ and $(w',w)\in E(X_{T_{G0}})$
\end{lemma}

We will prove the following claim.
For any $X_{T_{G0}}$ satisfying the condition in Lemma~\ref{lem:good-xtg0-prime}, 
\begin{equation}\label{eq:kuvx-vs-x-conditioning}
 \TV(K_{uv}(X)|X_{T_{G0}}, X|X_{T_{G0}})=O\B(\frac{\log^3 n}{n^6p'^{11/2}}\B)=\tilde{O}\B(\frac{1}{\sqrt{n}}\B)\,.   
\end{equation}This would then imply Lemma~\ref{lem:Kuvx-vs-x} as the condition in Lemma~\ref{lem:good-xtg0-prime} happens with high probability.

To use Lemma~\ref{lem:concentration-TV}, we will show that the likelihood ratio between the two distributions concentrates. Note that the distribution of $X|X_{T_{G0}}$ is just $\mu_G(X|X_{T_{G0}})$. We can write the density of both distributions as
\begin{align*}
    &\Pr(K_{uv}(X)=x|X_{T_{G0}})\\
    &\quad = \sum_{\substack{t_1=(v,w,w_1),t_2=(u,w,w_2):t_1,t_2\in x\\ (w,w_1),(v,w_1),(w,w_2),(u,w_2)\in E(\xm )\\ (w,u),(w,v)\notin E(\xm )}} \frac{1}{|V_{uv}(\xm )|\cdot |\wed[E(\xm)](v,w)|\cdot |\wed[E(\xm)](u,w)|}\mu_G(\xm|X_{T_{G0}})\,,
\end{align*}
where $\xm$ stands for $x-t_1-t_2$,
and 
\[\Pr(X=x|X_{T_{G0}}) = \mu_G(X=x|X_{T_{G0}})\,.\]
So the derivative is 
\begin{equation}\label{eq:derivative-prime}
\begin{split}
       & \frac{\Pr(K_{uv}(X)=x|X_{T_{G0}})}{\Pr(X=x|X_{T_{G0}})}\\
      &\quad = \sum_{\substack{t_1=(v,w,w_1),t_2=(u,w,w_2):t_1,t_2\in x\\ (w,w_1),(v,w_1),(w,w_2),(u,w_2)\in E(\xm )\\ (w,u),(w,v)\notin E(\xm )}} \frac{1}{|V_{uv}(\xm )|\cdot |\wed[E(\xm)](v,w)|\cdot |\wed[E(\xm)](u,w)|} \frac{\mu_G(\xm)}{\mu_G(x)}
\end{split}
\end{equation}
Similarly to the proof of Lemma~\ref{lem:Kvx-vs-x}, we will show each term on the right hand side is close to constant with high probability, and the number of terms also concentrates.

Note that if $t$ is chosen in $K_{uv}$ to be added to $X$, it always increases the number of edges by 2, i.e., $e(x)-e(x-t_1-t_2)=2$. So
\[\frac{\mu_G(x-t_1-t_2)}{\mu_G(x)} = \frac{p'^{|x-t_1-t_2|}p^{-e(x-t_1-t_2)}}{p'^{|x|}p^{-e(x)}} = \frac{p^2}{p'^2}\,.\]

From definition, $|V_{uv}(X)|$, $|\wed[E(X)](v,w)|$ and $|\wed[E(X)](u,w)|$ are $O(1)$-Lipschtz, so we have the following lemma
\begin{lemma}\label{lem:var-Vv-Wvw-prime}
With high probability,
\[
\b|\,|\wed[E(X)](u,w)|-\E |\wed[E(X)](u,w)|\, \b|= O(\sqrt{n^2p'}\log n),\]
\[\b||\wed[E(X)](v,w)|-\E |\wed[E(X)](v,w)|\,\b|= O(\sqrt{n^2p'}\log n)\,,
\]
and
 \[\b|\,|V_{uv}(X)|-\E |V_{uv}(X)| \,\b|O(\sqrt{n^2p'}\log n)\,.\] 
\end{lemma}
\begin{proof}
$\law(X|X_{T_{G0}})$ is a distribution over $O(n^2)$ triangles that is $\Theta(p')$-marginally small. So we can apply Corollary~\ref{cor:concentration_muG}. With high probability, any $O(1)$ Lipschtz function is within distance  $O(\sqrt{n^2p'}\log n)$ from its mean with high probability.
\end{proof}

\begin{lemma}\label{lem:e-Vv-prime}
For any $X_{T_{G0}}$ satisfying condition in Lemma~\ref{lem:good_xtg0},
\[\E [V_{uv}(X)|X_{T_{G0}}]=\Theta(n)\,,\]
 and for any $w\not=u,v$
\[\E [\wed[E(X)](u,w)|X_{T_{G0}}]=\Theta(n^3p'^2)\text{ and }\E [\wed[E(X)](v,w)|X_{T_{G0}}]=\Theta(n^3p'^2)\,.\]
\end{lemma}
\begin{proof}
Recall that $V_{uv}(X)$ consists of vertices $w$ that $(w,u),(w,v)\notin E(X)$. Fix any $w\in W_{u,v}(G)$. 
Both $(w,u)$ and $(w,v)$ are included in at most $n-2$ triangles in $G$, each being one on $X$ with probability $O(p')$ as shown in Lemma~\ref{lem:mug_marginally_small}. So
\[\Pr[(w,u),(w,v)\notin E(X)|X_{T_{G0}}]\ge 1-2(n-2)O(p')=\Omega(1)\,,\]
as $p'\ll 1/n$. Because $|W_{u,v}(G)|\ge cn$, we have
\[\E[\wed[E(X)](u,w)|X_{T_{G0}}] = \sum_{w\in W_{u,v}(G)}\Pr((w,u),(w,v)\notin E(X)|X_{T_{G0}}) =\Omega(n)\,. \]


Fix any $w\not=u,v$. Consider a vertex $w'$ such that $(w,w')\in E(X_{T_{G0}}), (v,w')\in G$, since $G$ is $c$-uniformly 2-star dense, $(v,w')$ is included in $\Theta(n)$ triangles, each having probability $\Theta(p')$ of being in $X$. So
\[\Pr[(v,w')\in E(X)|X_{T_{G0}}]=\Theta(np')\,.\]
From the assumption on $X_{T_{G0}}$, there are $\Theta(n^2p')$ such $w'$, so 
\[\E [\wed[E(X)](v,w)|X_{T_{G0}}] = \Theta(n^2p')\cdot \Theta(np') = \Theta(n^3p'^2)\,.\]
The same holds for $\E [\wed[E(X)](u,w)|X_{T_{G0}}]$
\end{proof}

Note that $\wed[E(X)](v,w)\ge |\wed[E(X_{T_{G0}})](v,w)|=\Theta(n)$, so $\E [\wed[E(X)](v,w)|X_{T_{G0}}]=\Theta(n)$. By Lemma~\ref{lem:var-Vv-Wvw-prime} and Lemma~\ref{lem:e-Vv-prime}, for any $v$ and $w$, conditioned on any $X_{T_{G0}}$,
\[\B|\frac{|V_{uv}(X)|-\E |V_{uv}(X)|}{\E |V_{uv}(X)|}\B|=O(\sqrt{p'}\log n)\,,\qquad \B|\frac{|\wed[E(X)](v,w)|-\E |\wed[E(X)](v,w)|}{\E |\wed[E(X)](v,w)|}\B|=O\B(\frac{\log n}{n^2p'^{3/2}}\B)\,,\]
\[\text{and} \qquad \B|\frac{|\wed[E(X)](u,w)|-\E |\wed[E(X)](u,w)|}{\E |\wed[E(X)](u,w)|}\B|=O\B(\frac{\log n}{n^2p'^{3/2}}\B)\]
holds with high probability.

As $|V_{uv}(\Xm)| = |V_{uv}(X-t_1-t_2)|$ and $|V_{uv}(X)|$ differ by at most 2, so are $|\wed[E(\Xm)](v,w)|$ and $|\wed[E(X)](u,w)|$. So with high probability, each term in \eqref{eq:derivative-prime} concentrates,
\begin{align}\label{eq:each-term-prime}
\begin{split}
&\frac{1}{|V_{uv}(\Xm )|\cdot |\wed[E(\Xm)](v,w)|\cdot |\wed[E(\Xm)](u,w)|} \frac{\mu_G(\Xm)}{\mu_G(X)}
\\ &\qquad = 
\frac{1}{\E|V_{uv}(\Xm )|\cdot |\wed[E(\Xm)](v,w)| \cdot \E |\wed[E(\Xm)](u,w)|}\frac{p^2}{p'^2}\B(1+O\B(\frac{\log n}{n^2p'^{3/2}}\B)\B)\,.\\
\end{split}
\end{align}

It remains to show that the number of terms
also concentrates. Let $S$ be a set of ordered vertices
\begin{align*}
    &S(X)= \{(w;w_1;w_2):t_1=(v,w,w_1)\in X,t_2=(u,w,w_2)\in X,\\ &\qquad \qquad(w,w_1),(v,w_1),(u,w_2),(w,w_2)\in E(X-t_1-t_2), (u,w),(v,w)\notin E(X-t_1-t_2)\}\,.
\end{align*}

\begin{lemma}\label{lem:concentration-number-of-terms-Kuv}
$|S(X)|$ is $O(\log^2 n)$-Lipschitz with high probability.
under $\law(X|X_{T_{g0}})$, 
\[\E[|S(X)||X_{T_{G0}}]=\Theta(n^7p'^6)=\tilde{\Theta}(n)\,,\]
and with high probability,
\[\b||S(X)|- \E[|S(X)||X_{T_{G0}}] \b|=O(n\sqrt{p'}\log ^3 n)=\tilde O(\sqrt{n})\,.\]
\end{lemma}
\begin{proof}
We first prove that $|S(X)|$ is Lipschitz.
Suppose $X$ and $X'$ differ at one triangle, $X'=X\vee e_{t_0}$.
Let $T(u,w)$ denote the set of triangles that contains $u,w$, and similarly $T(v,w)$ for triangles that contain $v,w$. 
Let $S_w(X):=\{(w;w_1;w_2):(w;w_1;w_2)\in S(X)\}$. Since $t_1$ and $t_2$ have to be in $X$, we have $|S_w(X)|\le |X_{T(u,w)}||X_{T(v,w)}|$.
We will show $S(X')$ and $S(X)$ differ by at most $O(\sup_w |X_{T(u,w)}||X_{T(v,w)}|)$ elements, and $|X_{T(u,w)}||X_{T(v,w)}| = \tilde O(1)$ with high probability.

First consider $S(X)\backslash S(X')$. Suppose $(w;w_1;w_2)\in S(X)$ but $(w;w_1;w_2)\notin S(X+t_0)$. Then we must have either  $(w,v)\in E(X+t_0-t_1-t_2)-E(X-t_1-t_2)$ or  $(w,u)\in E(X+t_0-t_1-t_2)-E(X-t_1-t_2)$. With loss of generality, assume it is $(w,v)\in E(X+t_0-t_1-t_2)-E(X-t_1-t_2)$.
Since  $E(X+t_0-t_1-t_2)$ and $E(X-t_1-t_2)$ differ by at most 3 edges, we have $|\{w:\exists w', (w;w_1;w_2)\in S(X)\backslash S(X')\}|\le 3$. Notice that each for each $w$, at most $|S_w(X)|$ of pairs $w_1,w_2$ satisfy $(w;w_1;w_2)\in S(X)$, otherwise $(w,v)\in E(X-t_1-t_2)$. So $|S(X)\backslash S(X')|\le 3\sup_w |X_{T(u,w)}||X_{T(v,w)}|$.

Then consider $S(X')\backslash S(X)$. If $(w;w_1;w_2)\in S(X')\backslash S(X)$, either $(w_1,v)\in E(X+t_0-t_1-t_2)-E(X-t_1-t_2)$, $(w,w_1)\in E(X+t_0-t_1-t_2)-E(X-t_1-t_2)$, $(w_2,u)\in E(X+t_0-t_1-t_2)-E(X-t_1-t_2)$, or
$(w,w_2)\in E(X+t_0-t_1-t_2)-E(X-t_1-t_2)$. Let \[S_1=\{(w;w_1;w_2): (w;w_1;w_2)\in S(X'), (w_1,v)\in E(X+t_0-t_1-t_2)-E(X-t_1-t_2)\}\,,\]
\[S_2=\{(w;w_1;w_2): (w;w_1;w_2)\in S(X'), (w,w_1)\in E(X+t_0-t_1-t_2)-E(X-t_1-t_2)\}\,,\]
\[S_3=\{(w;w_1;w_2): (w;w_1;w_2)\in S(X'), (w_2,u)\in E(X+t_0-t_1-t_2)-E(X-t_1-t_2)\}\,,\]
and
\[S_4=\{(w;w_1;w_2): (w;w_1;w_2)\in S(X'), (w,w_2)\in E(X+t_0-t_1-t_2)-E(X-t_1-t_2)\}\,.\]
Then 
\[S(X')\backslash S(X)\subset S_1\cup S_2\cup S_3\cup S_4\,.\]
Still note that for each $|S_w|\le |X_{T(u,w)}||X_{T(v,w)}|$, and $E(X+t_0-t_1-t_2)-E(X-t_1-t_2)$ has at most 3 edges, we have $|S_2|,|S_4|\le 3\sup_w|X_{T(u,w)}||X_{T(v,w)}|$. For $S_1$ and $S_3$, notice that each $w_1$ has at most 2 different $w$ such that $\exists w_2,(w;w_1;w_2)\in S(X')$ and $(v,w_1)\in E(X+t_0-t_1-t_2)-E(X-t_1-t_2)$. Otherwise by definition of $S(X')$, at least 3 triangles in $X'$ contains edge $(v,w_1)$, so $(v,w_1)$ would belong to $E(X-t_1-t_2)$ for any $t_1,t_2$, contradicting that $(v,w_1)\in E(X+t_0-t_1-t_2)-E(X-t_1-t_2)$. Therefore, $|S_2|\le 6\sup_w|X_{T(u,w)}||X_{T(v,w)}|$, similarly, $|S_4|\le 6\sup_w|X_{T(u,w)}||X_{T(v,w)}|$.
We have $|S(X')\backslash S(X)|\le 18\sup_w|X_{T(u,w)}||X_{T(v,w)}|$.

We have proven $|S(X)|$ changes by at most $18\sup_w|X_{T(u,w)}||X_{T(v,w)}|$ when changing one triangle. Notice that by Lemma~\ref{lem:mug_marginally_small}, $|X_{T(u,w)}|$ and $|X_{T(v,w)}|$ are stochastically dominated by $\bino(n-2,O(p'))$, so both quantities are $O(\log n)$ with high probability. This means $|S(X)|$ is $O(\log^2 n)$-Lipschitz with high probability. From Corollary~\ref{cor:small_marginal_concentration_extension}, there are $O(n^2)$ triangles in $\law(X|X_{T_{G0}})$, the distribution is $O(p')$ marginally small, so $|S(X)|$ concentrates around its mean with distance $O(n\sqrt{p'}\log ^3 n)$ with high probability.

Next we show that $\E |S(X)| = \Theta(n^7p'^6)=\tilde{\Theta}(n)$. It is sufficient to prove there are $\Theta(n^5p'^2)$ pairs $(w;w_1;w_2)$ that $(w;w_1;w_2)$ is in $S(X)$ with probability $\tilde{\Theta}(1/n^2)$. Since $G$ is a $c$-uniformly 2-star dense graph, there are $\Omega(n)$ vertices $w$ such that $(w,u),(w,v)\in G$. From the assumption that for any $w$, $|\wed[E(X_{T_{G0}})](v,w)|=\Theta(n^2p')$ and $|\wed[E(X_{T_{G0}})](u,w)|=\Theta(n^2p')$, there are $ \Theta(n^5p'^2)$ pairs $(w;w_1;w_2)$ that $(w,u),(w,v)\in G$ and $(w,w_1),(w,w_2)\in E(X_{T_{G0}})$.

Next we show that for such $(w;w_1;w_2)$, it is in $S(X)$ with probability $\tilde{\Omega}(1/n^2)$.
Consider the event that $(w;w_1;w_2)\in S(X)$, i.e., 
\begin{align*}
   &t_1=(v,w,w_1),t_2=(u,w,w_2),t_1\in X,t_2\in X, \\&(u,w_2), (w,w_2), (v,w_1),(w,w_1)\in E(X-t), (w,u),(w,v)\notin E(X-t)\,. 
\end{align*}
Since $E(X-t_1-t_2)\supset E(X_{T_{G0}})$, $(w,w_1),(w,w_2)\in E(X-t)$ always holds.
First, $t_1,t_2\in X$ with probability $\Theta(p'^2)$, next consider the event conditioning on this. 
$(w,v)$ and $(w,u)$ are included in at most $n-3$ triangles, each being included with probability $\Theta(p')$. So $(w,v),(w,u)\notin E(X-t)$ with probability at least
\[\Pr((w,v),(w,u)\notin E(X-t)|t_1,t_2\in X)\ge (1-\Theta(p'))^{2n-6}=\Theta (1).\]
 Graph $G$ is $c$-uniformly 2-star dense, so there are $\Theta(n)$ triangles contains $(u,w_2)$ and $(v,w_1)$ and do not contain $(w,v)$ or $w_u$. By Lemma~\ref{lem:mug_marginally_small}, conditioning on the all the events above, each triangle is included in $X$ with probability $\Theta(p')$, so $(u,w_2),(v,w_1)\in E(X-t)$ with probability $\Theta(n^2p'^2)$.
Therefore, the event $(w;w_1;w_2)\in S(X)$ happens with probability $\Theta(n^2p'^4).$
\end{proof}

As for any $w$, 
\[\frac{1}{\E |V_{uv}(X)|\cdot \E |\wed[E(X)](u,w)|\cdot \E |\wed[E(X)](v,w)|} \frac{p^2}{p'^2}=O\B(\frac{1}{n^7p'^6}\B)\,.\]
Lemma~\ref{lem:concentration-number-of-terms-Kuv} shows the following function is $ O(\log^2 n/(n^7p'^6))$-Lipschitz with high probability. 
\[f(X) \defeq \sum_{(w;w_1;w_2)\in S(X)}\frac{1}{\E |V_{uv}(X)|\cdot \E |\wed[E(X)](u,w)|\cdot \E |\wed[E(X)](v,w)|} \frac{p^2}{p'^2}\,.\]
Again by Corollary~\ref{cor:small_marginal_concentration_extension}, $f(X)$ concentrates around its mean with distance at most $O(\log^3 n/(n^6p'^{11/2})))$ with high probability. 

Now combine this with \eqref{eq:each-term-prime}, we get for any $X_{T_{G0}}$ satisfying condition in Lemma~\ref{lem:good-xtg0-prime}, the likelihood ratio
\[L(x) := \frac{\Pr(K_{uv}(X)=x|X_{T_{G0}})}{\Pr(X=x|X_{T_{G0}})}\]
concentrates as
\[\left|L(X) - \E f(X) \right|=O\B(\frac{\log^3 n}{n^6p'^{11/2}}\B)+O\B(|S(X)|\frac{\log n}{n^{9}p'^{15/2}}\B)\,,\]
which is $O\B(\frac{\log^3 n}{n^6p'^{11/2}}\B)$ with high probability since $|S(X)|=\Theta(n^7p'^6)$ with high probability. 

We can use Lemma~\ref{lem:distance-expectation} (which states that concentration implies concentration about the mean for bounded random variables), as it is not hard to check that the likelihood ratio is always bounded by $\tilde{O}(n^3)$. Hence, with high probability,
\[\left| L(X)- \E L(X) \right|=O\B(\frac{\log^3 n}{n^6p'^{11/2}}\B)\,.\]
This completes the proof of \eqref{eq:kuvx-vs-x-conditioning}, using that $\E  L(X) = 1$ and Lemma~\ref{lem:concentration-TV}. \eqref{eq:kuvx-vs-x-conditioning} then implies Lemma~\ref{lem:Kuvx-vs-x} as stated before. 


\end{proof}

\subsection{Proof of Lemma~\ref{lem:er-rigpf}}
Recall that $\rigpf(n,d,\delta)$ is generated by sampling $M_j\sim \poi(dp_j)$ for $j=2,4,5,\cdots, n$ where $p_j = \binom{n}{j}\delta^j(1-\delta)^{n-j}$. And then add a $j$-clique uniformly for $M_j$ times. Denote $(M_4,\cdots M_12)$ by $\vec{M}$. The event that $M_j>0$ for $j>12$ happens with probability $O(dn^{13}\delta^{13})\ll 1$, so we can only consider the distribution conditioned on this event.
By Poisson splitting, the addition of 2-cliques is equivalent to the following. For  each $(i,j)\in \binom{[n]}{2}$, add edge $(i,j)$ with probability $1-e^{-dp_2/\binom{n}{2}} = 1-e^{-d\delta^2(1-\delta)^{n-2}}$. This results in a \ER\ graph, $G(n,1-e^{-d\delta^2(1-\delta)^{n-2}})$. 

By Corollary~\ref{cor:TV-add-clique}, and convexity of total variation distance, 
\begin{equation}\label{eq:rigpf-exer}
  \TV\B(\rigpf(n,d,\delta), \E_{\vec{M}}G(n,q(\vec{M})\B)= O(n^{-3/2}\E_{\vec{M}} \|\vec{M}\|_1)  \,.
\end{equation}
where \[1-q(\vec{M}) = e^{-d\delta^2(1-\delta)^{n-2}}\prod_{j=4}^{12}\B(1-\binom{n-2}{j-2}\binom{n}{j}\B)^{M_j} =  e^{-d\delta^2(1-\delta)^{n-2}}\prod_{j=4}^{12}\B(1-\frac{j(j-1)}{n(n-1)}\B)^{M_j}\]
is the edge density of $\rigpf$ conditioned on $\vec{M}$.
The right hand side of \eqref{eq:rigpf-exer} is $\E_{\vec{M}} \|\vec{M}\|_1 \le  \sum_{j\ge 4}dp_j=O(dn^4\delta^4)\ll 1$ for $d\gg n^{5/2}$ and $\delta=\Theta (\sqrt{1/d})$.

Now it remains to show that $\E_{\vec{M}}G(n,q(\vec{M})$ is close to \ER.
Let $q = 1-e^{-d\delta^2+(n-2)d\delta^3(1-\delta)^{n-3}}$, by Lemma~\ref{lem:distance-binomial}, 
\[\TV(G(n,q), \E_{\vec{M}}G(n,q(\vec{M})) = O(n\E_{\vec{M}}|q-q(\vec{M})|)\,.\]
 Since $M_j$ follows Poisson distribution, we have that with probability $O(1/\log n)$, for any $4\le j\le 12$,
\[\E M_j-\sqrt{dp_j\log n}\le M_j\le \E M_j+\sqrt{dp_j\log n}\,.\]
For any $\vec{M}$ within the above range, 
\begin{align*}
    |q-q(\vec{M})| =&\  \B|e^{-d\delta^2+(n-2)d\delta^3(1-\delta)^{n-3}} - e^{-d\delta^2(1-\delta)^{n-2}}\prod_{j=4}^{12}\B(1-\frac{j(j-1)}{n(n-1)}\B)^{M_j}\B|
\\=&\ e^{-d\delta^2(1-\delta)^{n-2}} \B|\prod_{j=4}^{n}e^{- \binom{n-2}{j-2}d\delta^{j}(1-\delta)^{n-j}}\prod_{j=4}^{12}\B(1-\frac{j(j-1)}{n(n-1)}\B)^{\E M_j}+O(\sqrt{dp_4\log n/n^2})\B|\\ =&\  o(n^{-1})\,.
\end{align*}
Combining this with \eqref{eq:rigpf-exer}, we have 
\[\TV(G(n,q),\rigpf(n,d,\delta))=o(1)\,.\]

\subsection{Proof of Lemma~\ref{lem:reverse-q}}

Consider $\E_{G\sim \rgt(n,p,p')}\Pr_{\Ab}[e\in \Ab(G+e)]$. Fix any edge  $e$ inside $S$.
The probability that  reverse transition $\Ab$ resamples $e$ is bounded by $O(np')$ by Corollary~\ref{cor:muge-marginal} when $G+e$ is $c$-uniformly 2-star dense. And by Lemma~\ref{lem:probability-of-good}, $G\sim \rgt(n,p,p')$ is $p^2/2$ uniformly 2-star dense with high probability. Overall, $e$ gets resampled with probability $O(np')$, which means $p_e = 1-O(np')$.

\section{Technical Lemmas}\label{sec:technical-lemmas}
\regularsubgraph*
\begin{proof}
We will use a random process to generate $\bwed[G](u,w)$ for all $w$ and argue the desired properties is satisfied with non-zero probability. Consider the following process.
\begin{enumerate}
    \item Start with $\bwed[G](u,w)=\emptyset$ for any $w\not\in \{u,v\}$.
    \item Loop through all pairs of different vertices $w_1,w_2\notin \{u,v\} $ and do the following. \\ If $w_1\in \wed[G](u,w_2)$ but $w_2\notin \wed[G](u,w_1)$,  put $w_1$ in $\bwed[G](u,w_2)$. \\ If $w_2\in \wed[G](u,w_1)$ but $w_1\notin \wed[G](u,w_2)$,  put $w_2$ in $\bwed[G](u,w_2)$.\\ If $w_1\in \wed[G](u,w_2)$ and $w_2\in \wed[G](u,w_1)$, with $1/2$ probability put $w_1$ in $\bwed[G](u,w_2)$  and  $1/2$ probability put $w_2$ in $\bwed[G](u,w_1)$.
    \item If there exists $w$ with $\bwed[G](u,w)<cn/3$, declare failure of the process. Otherwise, for all $w\in \{u,v\}$, delete an arbitrary set of vertices in $\bwed[G](u,w)$ so that it has size $cn/3$ after deletion.
\end{enumerate}
It is clear that the second condition of the lemma is always satisfied through the process, and if the algorithm does not fail in step 3, the first condition is also satisfied. It remains to prove that the probability of failure is less than 1.

Consider a fixed vertex $w\notin \{u,v\}$ and the set $\bwed[G](u,w)$ after step 2 of the algorithm. Since $G$ is $c$-uniformly 2-star dense, we know $\wed[G](u,w)\ge c(n-2)$. For each $w'\in \wed[G](u,w)$, it is included in $\bwed[G](u,w)$ in step 2 either with probability 1 or with probability 1/2, independently. Therefore, the size of $\bwed[G](u,w)$ stochastically dominates $\bino(c(n-2),1/2)$. And by Chernoff bound, this is smaller than $cn/3$ with probability $e^{-\Omega(n)}$. So by union bouding over all $w$, the probability of failure of the algorithm is bounded by $ne^{-\Omega(n)}$, which is smaller than 1 for sufficiently large $n$.
\end{proof}

\begin{lemma}\label{lem:binom-2-sample-with-replacement}
Sampling with replacement for $\bino(n,p)$ times and sampling without replacement for $\bino(n,p)$ times have total variation distance at most $p(1-p)+np^2$.
\end{lemma}
\begin{proof}
For sampling with replacement $k$ times, the probability of having a collision is 
\[1-\frac{n-1}{n}\cdot \frac{n-2}{n}\cdots \frac{n-k+1}{n}\le \frac{\sum_{j=1}^kj}{n}\le \frac{k^2}{n}\,.\]
Conditioned on not colliding, sampling with replacement has the same distribution as sampling without replacement, so the total variation distance is bounded by $\E[k^2/n]$ where $k\sim \bino(n,p)$. The lemma follows from that the second moment of $\bino(n,p)$ is $np(1-p)+n^2p^2$.
\end{proof}

\begin{lemma}\label{lem:influence_partition}
Consider a distribution on $\{0,1\}^\cX$.
Suppose $S$ and $S'$ can be partitioned as $S = S_1\cup S_2$ and $S'=S_1'\cup S_2'$. Then 
\begin{enumerate}
    \item[(i)] $\I{S}{S'}\le \I{S_1}{S'}+\I{S_2}{S'}$,
    \item[(ii)] $\I{S}{S'}\le \I{S}{S_1'}+\I{S}{S_2'}$,
    \item[(iii)] $ \I{S}{S'}\le \sum_{x\in S, x\in S'}\I{x}{x'}$.
\end{enumerate}
\end{lemma}
\begin{proof}
Suppose 
\[\I{S}{S'} = \TV(P_{x_{S'}|x_S^a,\omega}, P_{x_{S'}|x_S^b,\omega})\,.\]
For the first statement, we can use the triangle inequality to upper bound the last display by
\[\TV(P_{x_{S'}|x_{S_1}^a,x_{S_2}^a,\omega}, P_{x_{S'}|x_{S_1}^a,x_{S_2}^b,\omega})+ \TV(P_{x_{S'}|x_{S_1}^a,x_{S_2}^b,\omega}, P_{x_{S'}|x_{S_1}^b,x_{S_2}^b,\omega})\le \I{S_2}{S'}+\I{S_1}{S'}.\]

For the second statement, we will design a coupling $(x_{S'}^a,x_{S'}^b)$ between $P_{x_{S'}|x_S^a,\omega}$ and $P_{x_{S'}|x_S^b,\omega}$ so that $\Pr[x_{S'}^a\not= x_{S'}^b]$ is bounded by $\I{S}{S_1'}+\I{S}{S_2'}$. First let $(x_{S_1'}^a,x_{S_1'}^b)$ be an optimal coupling between $P_{x_{S_1'}|x_S^a,\omega}$ and $P_{x_{S_1'}|x_S^b,\omega}$, so that 
\[\Pr[x_{S_1'}^a\not = x_{S_1'}^b]\le \I{S}{S_1}\,.\]
When $x_{S_1'}^a = x_{S_1'}^b = x_{S_1'}$, let the conditional distribution of $(x_{S_2'}^a,x_{S_2'}^b)$ be the optimal coupling between $P_{x_{S_2'}|x_{S_1'},x_S^a,\omega}$ and $P_{x_{S_2'}|x_{S_1'},x_S^b,\omega}$, otherwise, couple $(x_{S_2'}^a,x_{S_2'}^b)$ arbitrarily, so \[\Pr[x_{S_2'}^a\not= x_{S_2'}^b|x_{S_1'}^a = x_{S_1'}^b]\le \I{S}{S_2}\,.\]
Therefore, we have
\begin{align*}
    \Pr[x_{S'}^a\not= x_{S'}^b]
    = \Pr[x_{S_1'}^a\not = x_{S_1'}^b] + \Pr[x_{S_2'}^a\not= x_{S_2'}^b|x_{S_1'}^a = x_{S_1'}^b]\le \I{S}{S_1}+\I{S}{S_2}\,.
\end{align*}
The proof for $\IM{S}{S'}\le \IM{S}{S_1'}+\IM{S}{S_2'}$ is identical, except $\omega$ is fixed to be empty. Finally, item (iv) follows by repeatedly applying (i) and (ii).
\end{proof}

\end{document}